\documentclass[a4paper]{amsart}

\usepackage{xcolor}
\definecolor{lgray}{RGB}{240,240,240}
\definecolor{webgreen}{rgb}{0,.5,0}
\definecolor{webbrown}{rgb}{.6,0,0}
\definecolor{RoyalBlue}{cmyk}{1, 0.50, 0, 0}
\usepackage[colorlinks=true, breaklinks=true, urlcolor=webbrown, linkcolor=RoyalBlue, citecolor=webgreen, backref=page]{hyperref}

\usepackage{epsfig, graphicx, subfigure}
\usepackage{verbatim, setspace}
\usepackage{newtxtext,newtxmath}

\usepackage{mathabx}
\usepackage{pifont}

\newtheorem*{pdleg}{Principle of Descent and Lower Envelope Theorem}
\newtheorem*{sdp}{Strong Domination Principle}
\newtheorem*{hst}{Harnack's Theorem}
\newtheorem*{rt}{Removability Theorem}
\newtheorem*{rrt}{Riesz Representation Theorem}
\newtheorem*{gmp}{Generalized Minimum Principle}

\newtheorem{theorem}{Theorem}[section]
\newtheorem{corollary}[theorem]{Corollary}

\newtheorem{lemma}[theorem]{Lemma}

\newtheorem{remark}[theorem]{Remark}

\makeatletter
\renewcommand{\tocsection}[3]{%
  \indentlabel{\@ifnotempty{#2}{\bfseries\ignorespaces#1 #2\quad}}\bfseries#3}
\renewcommand{\tocsubsection}[3]{%
  \indentlabel{\@ifnotempty{#2}{\ignorespaces#1 #2\quad}}#3}
\newcommand\@dotsep{4.5}
\def\@tocline#1#2#3#4#5#6#7{\relax
  \ifnum #1>\c@tocdepth 
  \else
    \par \addpenalty\@secpenalty\addvspace{#2}%
    \begingroup \hyphenpenalty\@M
    \@ifempty{#4}{%
      \@tempdima\csname r@tocindent\number#1\endcsname\relax
    }{%
      \@tempdima#4\relax
    }%
    \parindent\z@ \leftskip#3\relax \advance\leftskip\@tempdima\relax
    \rightskip\@pnumwidth plus1em \parfillskip-\@pnumwidth
    #5\leavevmode\hskip-\@tempdima{#6}\nobreak
    \leaders\hbox{$\m@th\mkern \@dotsep mu\hbox{.}\mkern \@dotsep mu$}\hfill
    \nobreak
    \hbox to\@pnumwidth{\@tocpagenum{\ifnum#1=1\bfseries\fi#7}}\par
    \nobreak
    \endgroup
  \fi}
\AtBeginDocument{%
\expandafter\renewcommand\csname r@tocindent0\endcsname{0pt}
}
\def\l@subsection{\@tocline{2}{0pt}{2.5pc}{5pc}{}}
\makeatother

\newcommand{\T}		{\mathbb{T}}
\newcommand{\D}		{\mathbb{D}}
\newcommand{\R}		{\mathbb{R}}
\newcommand{\C}		{\mathbb{C}}
\newcommand{\N}		{\mathbb{N}}

\newcommand{\cws}{\stackrel{*}{\to}}
\newcommand{\cic}  {\stackrel{\scriptsize\cp}{\rightarrow}}

\newcommand{\cp}{\mathrm{cap}}
\newcommand{\dist}{\mathrm{dist}}

\newcommand{\clos}{\mathrm{clos}}
\newcommand{\supp}{\mathrm{supp}}

\newcommand{\re}{\mathrm{Re}}

\newcommand{\qasq}{\quad \text{as} \quad}
\newcommand{\qandq}{\quad \text{and} \quad}

\newcommand{\ic}{\mathrm{i}}

\newcommand{\RS}{{\mathcal R}}
\newcommand{\K}{K_f}
\newcommand{\B}{\mathcal T}

\numberwithin{equation}{section}

\begin{document}

\title[Optimal Rational Approximants]{ n-th Root Optimal Rational Approximants \\to Functions  with Polar Singular Set}

\author{Laurent Baratchart}

\address{INRIA, Project FACTAS, 2004 route des Lucioles --- BP 93, 06902 Sophia-Antipolis, France}

\email{\href{mailto:laurent.baratchart@sophia.inria.fr}{laurent.baratchart@sophia.inria.fr}}

\author{Herbert Stahl\dag}

\author{Maxim Yattselev}

\address{Department of Mathematical Sciences, Indiana University Indianapolis, 402~North Blackford Street, Indianapolis, IN 46202}

\email{\href{mailto:maxyatts@iu.edu}{maxyatts@iu.edu}}

\thanks{The research of the last author was supported in part by a grant from the Simons Foundation, CGM-706591.}



\begin{abstract}
  Let \( D \) be a bounded Jordan domain and \( A \) be its complement on the Riemann sphere. We investigate the $n$-th root asymptotic behavior in \( D \) of best rational approximants, in the uniform norm on \( A \), to functions holomorphic on \( A \)  having  a multi-valued continuation to quasi every point of \( D \) with finitely many branches. More precisely, we study weak$^*$ convergence of the normalized counting measures of the poles of such approximants as well as their convergence in capacity. We place best rational approximants into a larger class of \( n \)-th root optimal meromorphic approximants, whose behavior we investigate using potential-theory  on certain compact bordered Riemann surfaces.
\end{abstract}

\maketitle

\tableofcontents

\section*{List of Symbols}

\begin{flushleft}

\noindent {\bf General Point Sets:} \smallskip \\
\begin{tabular}{ p{2cm} l }
$\T,\D,\C,\overline\C$ & unit circle, open unit disk,  complex plane, and extended complex plane \\
$\T_r,\D_r$ & circle and open disk of radius \( r \) centered at the origin \\
$T,D,A$ & Jordan curve, its interior domain, and the closure of its exterior domain in $\overline\C$\\
\end{tabular}

\smallskip

\noindent {\bf Riemann Surfaces:} \smallskip \\
\begin{tabular}{ p{2cm} l }
$\RS_*,p$ &  compact Riemann surface with natural projection \( p:\RS_*\to \overline \C \) \\
$\RS$ & \( \RS := \{z\in\RS_*: p(z) \in D\} \) \\
$\B$ & a connected component of \( p^{-1}(T) \) homeomorphic to \( T \) \\
\( \mathbf{rp}(\cdot)\) & the set of ramification points of a given Riemann surface \\
\( m(z) \) & ramification order of a point \( z \) on a Riemann  surface \\
\( M\) & total number of sheets of \( \RS \)
\end{tabular}

\smallskip

\noindent {\bf Classes of Functions:} \smallskip \\
\begin{tabular}{p{2cm} l}
  $C(E)$ & continuous functions on a set $E$\\
  $\mathcal H(A)$ & functions analytic in some neighborhood of \( A \) \\
$\mathcal S(A)$ & subclass of \( \mathcal H(A) \) of functions multi-valued and quasi everywhere analytic off \( A \) \\
\(\mathcal E(A)\) & subclass of \( \mathcal H(A) \) of functions single-valued and quasi everywhere analytic off \( A \) \\
$\mathcal F(\RS)$ & class of functions quasi everywhere analytic on \( \overline\RS \) \\
$\mathcal F(A)$ & subclass of approximated functions analytic on \( A \) \smallskip \\
$\mathcal P_n$ & algebraic polynomials of degree at most \( n \) \\
$\mathcal M_n(D)$  & monic algebraic polynomials of degree $n$ with all their zeros in $D$ \\
$\RS_n(D)$ & $\mathcal P_n\mathcal M_n^{-1}(D)$ \smallskip \\
$H^\infty(D)$ & space of bounded holomorphic functions in \( D \) \\
$H_n^\infty(D)$ & $H^\infty(D)\mathcal M_n^{-1}(D)$ \smallskip \\
$L^2(\T)$ & square integrable functions on the unit circle \\ 
$L^\infty(\T)$ & essentially bounded functions on the unit circle \\ 
$H^2$ & Hardy space of functions in $L^2(\T)$ with vanishing Fourier coefficients of negative index\\
$H^2_-$ & \( L^2(\T)\ominus H^2 \) 
\end{tabular}

\smallskip

\noindent {\bf Operators:} \smallskip \\
\begin{tabular}{p{2cm} l}
$\mathbb P_+,\mathbb P_-$ & orthogonal projections from \( L^2(\T) \) onto \( H^2 \), \( H^2_- \) \\
$\Gamma_f$ & Hankel operator from \( H^2 \) to \( H^2_-\), \(  h \mapsto \mathbb P_-(fh) \) \\
$s_n(\Gamma_f)$ & the \( n \)-th singular number of \( \Gamma_f \)
\end{tabular}

\smallskip

\noindent {\bf Potential Theory:} \smallskip \\
\begin{tabular}{p{2cm} l}
$\cp(K)$ & logarithmic capacity of $K$\\
$\cp_\Omega(K)$ & Greenian capacity of \( K \) relative to \( \Omega \) \\
$\cp(E,K)$ & capacity of the condenser \( (E,K) \) \\
$g_\Omega(\cdot,w)$ & Green function for $\Omega$ with pole at $w$ \\
$g(\sigma,\Omega;z)$ & Green potential of the measure \( \sigma \) relative to  \( \Omega \) \\
$V^\sigma(z)$ & logarithmic potential of the measure \( \sigma \) \\
$\mu_{\Omega,K}$ & Green equilibrium distribution on a set \( K\subset \Omega \) relative to  \( \Omega \) \\
$\mathcal B_v^E$ & balayage function of a superharmonic function \( v \) relative to a set \( E \) \\
$\sigma^E$ & balayage measure of a measure \( \sigma \) onto a set \( E \) \\
$\widehat\sigma$ & lift of a measure \( \sigma \) on \( D \) to \( \RS \) \\
$p_*(\sigma)$ & projection (pushforward) of a measure \( \sigma \) on \( \RS \) to \( D \) \\
$\partial_\mathsf{f}A,\clos_\mathsf{f}(E)$ & fine boundary and closure of a set \( E \) \\
$b(E),i(E)$ & base and the set of finely isolated points of a set \( E \)
\end{tabular}

\smallskip

\noindent {\bf Various Symbols:} \smallskip \\
\begin{tabular}{p{2cm} l}
$\phi$ & a conformal map from \( \D \) onto \( D \) \\
$\mathcal K_f$ & collection of ``branch cuts'' for \( f \) \\
$\K$ & ``branch cut'' of minimal Greenian capacity for \( f\in \mathcal S(A) \) \\
\( \|\cdot\|_K \) & essential supremum norm on \( K \) \\
\( \rho_n(f,A) \) & error of best rational approximation of \( f \) analytic on \( A \) by functions in \( \RS_n(D) \)
\end{tabular}

\end{flushleft}

\section{Introduction}

Rational approximation to holomorphic functions of one complex variable has long been a requisite chapter of classical analysis with notable applications to number theory \cite{Herm73,Siegel,KR}, spectral theory \cite{Nikolskii,Gustafson,Peller} and numerical analysis \cite{Gr65,IserlesNorsett,GoTre,SilYuTreBe}. Over the last decades it  became a cornerstone of modeling in Engineering \cite{TeToGa,Tj_PRA77,Pozzi,AvFaRe,GuSem}, and it can also be viewed today as a technique to regularize inverse potential problems \cite{GoNoHe,Isakov,BMSW06}. Finding best rational approximants of  prescribed degree  to a specific function, say in the uniform norm on a given set, seems out of reach except in rare, particular cases. Indeed, such approximants depend in a rather convoluted manner, both on the approximated function and on the set where approximation should take place. Accordingly, the constructive side of the theory has focused on estimating optimal convergence rates as the degree grows large and devising approximation schemes coming close to meet them \cite{Walsh,Braess,GonRakh87,St97}, or else studying the behavior of natural, computationally appealing candidate approximants like Pad\'e interpolants and their variants \cite{BakerGravesMorris,St89,Nu2,Lub,AAA}.

When a function is holomorphic in some neighborhood of a continuum $A$, the optimal speed of convergence  for rational approximants on $A$ is at least geometric in the degree. Then, a coarse but manageable estimate of this speed proceeds via asymptotics of the $n$-th root of the error in approximation by rational functions of degree $n$. For functions analytically continuable off $A$ except over a polar set containing branchpoints (throughout \emph{polar} means of logarithmic capacity zero), and provided that $A$ does not divide the extended complex plane, Gonchar and Rakhmanov constructed,  using multipoint Pad\'e interpolants and dwelling on work by the second author, a sequence of rational approximants whose $n$-th root error on $A$ is asymptotically the smallest possible. They further showed that these interpolants converge in capacity on the complement of  a compact set $K$ minimizing the capacity of the condenser $(A,K)$ under the constraint that the function is analytic off $K$, and proved that the normalized counting measures of their poles converge to the condenser equilibrium distribution on $K$ \cite{GonRakh87}. It is remarkable that \( K \) solves a  geometric extremal problem from logarithmic potential theory, close in spirit to the Lavrentiev type \cite{Kuz80}, and that it depends merely on the set where approximation takes place, on the branchpoints of the approximated function and its monodromy, but nothing else. Such a structure emerges because only the $n$-th root of the error is considered, rather than the error itself.  Since then, it has been an open issue  whether any $n$-th root optimal sequence of approximants -- in particular a sequence of best approximants -- has the same behavior. The present paper answers this question in the positive, at least when the branchpoints are finite in number and order. In particular, our results connect, apparently for the first time, the singularities of best uniform approximants to those of the approximated function. We also address the case of no branchpoints, when the approximated function is analytic except over a  polar set, in which  the speed of best rational approximation on $A$ is known to be faster than geometric with the degree. We prove that $n$-th root optimal approximants converge in capacity outside the singular set of the function, and that the ``effective'' poles converge in a sense to that set.

The gist of the paper becomes transparent upon observing that the behavior of rational approximants can be surmised when the function to be approximated extends analytically to a multiply-sheeted Riemann surface over the complex plane. As rational functions are single-valued, this topological discrepancy leads the approximants to mark out their domain of approximation by accumulating poles so as to form a cut, thereby preventing single-valued continuation in the limit. In the case of (diagonal) multipoint Pad\'e interpolants, this cut has been characterized as being of smallest weighted capacity in a field that depends on the limiting distribution of the interpolation points, and poles asymptotically distribute according to the weighted equilibrium measure  of that cut; moreover, the Pad\'e interpolants converge in capacity on the extremal domain thus demarcated. This was established in \cite{GonRakh87}, dwelling on the works \cite{St85a,St85b,St86,St89,St97} that deal with classical  Pad\'e  interpolants and correspond to the zero  field and unit weight; see  also \cite{NuSi} for early developments along these lines, and \cite{BStYa} for applications to $L^2$-best rational approximation  on the circle. Subsequently, by choosing interpolation points adequately and performing surgery to eliminate spurious poles, the authors of \cite{GonRakh87} construct, on any continuum  $A$ not dividing the extended plane and contained in the analyticity domain of a function indefinitely continuable except over a closed polar set containing branchpoints, a sequence of rational approximants converging uniformly to that function on $A$ as their degree $n$ grows large and whose $n$-th root error has a {\it liminf} which is smallest possible, as well as a true limit. For this weakly optimal choice of interpolation points (meaning that the choice is optimal in the $n$-th root sense), the cut $K$ of minimum weighted capacity is also the cut minimizing the condenser capacity of $(A,K)$, as well as the cut of minimum Greenian capacity in the complement of $A$. The smallest value for the limit of the $n$-th root error is a simple, explicit function of this Greenian capacity, and the poles of the approximants thus constructed distribute asymptotically according to the Green equilibrium measure of that cut.

Now, assuming in addition that the branchpoints of the continuation off $A$ of the function to be approximated  are finite in number and of algebraic type, we shall  prove that \emph{any}  sequence of rational (or meromorphic) approximants  of increasing degree $n$ whose $n$-th root error on $A$ converges to the smallest possible limit -- {\it a fortiori} every sequence of best approximants -- has the same asymptotic distribution of poles as the particular sequence constructed in \cite{GonRakh87}. More precisely, if a function analytic  in a simply connected neighborhood of a continuum $A$  in the extended complex plane is indefinitely continuable off that neighborhood except over a closed polar set containing finitely many branchpoints, all of algebraic type, then \emph{the normalized counting measures of the poles of any sequence of rational approximants of increasing degree $n$  with asymptotically optimal $n$-th root error on $A$ do converge weak-star, as $n$ grows large, to the Green equilibrium distribution of the compact set of minimum Greenian capacity outside of which the function is single-valued; moreover, convergence holds in capacity everywhere off that compact set.} Here, Green functions are understood with respect to the complement of the continuum  $A$ where approximation takes place. Finally, if there are no branchpoints, that is, if the approximated function is single-valued and analytic on the extended complex plane except possibly over a closed polar set $E$, then there are rational approximants  converging on $A$ faster than geometrically with the degree. We shall prove that such  sequences of approximants (as well as their meromorphic analogs)  converge in capacity on the  extended plane deprived from $E$, and that retaining the singular part that comes close to $E$ generates new sequences of approximants, still converging faster than geometrically with the degree, while satisfying in addition that  any weak-star limit point of the normalized counting measures of their poles is supported on $E$.

Previously cited references, which deal with Pad\'e or multipoint Pad\'e interpolants, exploit the connection between denominators thereof and non-Hermitian orthogonal polynomials on a system of arcs encompassing the singular set of the interpolated function.  Indeed, the core of the work in \cite{St86,GonRakh87} is to derive asymptotics of such polynomials on extremal systems of arcs like those constructed in  \cite{St85a,St85b}, so as to qualify the behavior of the poles of the interpolants when the degree grows large and deduce from it the desired convergence properties. Here, we proceed in the opposite direction: we assume that the optimal rate is met in the $n$-th root sense and deduce from it the behavior of the poles. For this we cannot make use of orthogonal polynomials, and in fact it is not even known if interpolation takes place at all in the case of best approximants. However,  the construction from \cite{St85a,St85b} will  still be basic to our purposes.

\emph{This work was initiated jointly by the three authors in 2009, but the untimely passing away of the second one on April 22nd, 2013 prevented him from seeing its completion. Still, some fundamental ideas are his. }

\section{Preliminaries and Main Results}
\label{sec:2}

Given a function  \( f \) holomorphic in a neighborhood of a closed set \( A\subset\overline{\C} \), the error of approximation of $f$ on $A$ by rational functions of degree $n$ is 
\begin{equation}
\label{besterror}
\rho_n(f,A) := \inf_{r\in\RS_n(\C\setminus A)} \|f-r\|_A,
\end{equation}
where \( \|\cdot\|_A\) stands for the supremum norm on $A$ and, for  \(\Omega\subset\overline{\C}\), we let \( \RS_n(\Omega) \) be the class of rational functions of type \( (n,n) \) with all their poles in \( \Omega \). That is, if \( \mathcal P_n \)  denotes the space of algebraic polynomials of degree at most \( n \) and  \( \mathcal M_n(\Omega) \) the monic polynomials of degree \( n \) with all zeros in \( \Omega \), then \( \RS_n(\Omega) := \mathcal P_n\mathcal M_n^{-1}(\Omega) \). It was shown by Walsh \cite{Walsh,Bag69}, using interpolation techniques, that
\begin{equation}
\label{walsh}
\limsup_{n\to\infty}\rho_n^{1/n}(f,A) \leq \exp\big\{-1/\cp(A,K)\big\},
\end{equation}
where \( K \) is any closed set disjoint from $A$ in the complement of which \( f \) is holomorphic and \( \cp(A,K) \) denotes  the capacity of the condenser \( (A,K) \). A definition of  condenser capacity can be found in \cite[Chapter VIII, Section 3]{SaffTotik}; for our purposes, it is  enough to know that if $\overline{\C}\setminus A$ is connected, then  \( \cp(A,K) \) coincides with the Greenian capacity \(\cp_{\overline{\C}\setminus A}(K)  \) defined in Section~\ref{ssec_cap},  see \cite[Chapter VIII, Theorem 2.6 \& Corollary 2.7]{SaffTotik} for this equivalence.

It is known that Walsh's inequality \eqref{walsh} cannot be improved \cite{LevTikh67}. On the other hand, it was conjectured by Gonchar \cite{MR734178} and proven by Parf\"enov when $A$  is a continuum with connected complement \cite{Par86} (also  later by Prokhorov for any compact set \( A \) \cite{Pro93}) that
\begin{equation}
\label{prokhorov}
\liminf_{n\to\infty}\rho_n^{1/n}(f,A) \leq \exp\big\{-2/\cp(A,K)\big\}.
\end{equation}
Hence, $\rho_n(f,A)$ has no limit in general as $n\to\infty$, and when the limit exists it cannot exceed the right-hand side of \eqref{prokhorov}. For certain classes of functions  and certain {\it loci}  of approximation $A$, it was nevertheless shown  that $\rho_n(f,A)$ does have a limit, which is  equal to the right-hand side of \eqref{prokhorov}. More precisely, let $\mathcal H(A)$ denote the space of functions holomorphic on a  (variable) neighborhood of $A$, and  \( \mathcal S(A)\subset \mathcal H(A) \) those functions continuable analytically into the complement of \( A \) along any path that avoids some compact polar\footnote{see Section~\ref{ssec_fine} for a definition and basic properties of polar sets, that may be defined as sets of zero logarithmic capacity.} subset of \( \overline{\C}\setminus A \) (which may depend on the  function); we require in addition  that this continuation is not single-valued, namely that there are paths with the same endpoints leading to different analytic branches. Now, when   $A$ is a continuum that does not separate the plane and \( f\in\mathcal S(A) \), it follows from  work by the second author in \cite{St85a,St85b,St85c} and it was explicitly stated by Gonchar and Rakhmanov  in \cite[Theorem~1$^\prime$]{GonRakh87} that
\begin{equation}
\label{bestrate}
\lim_{n\to\infty} \rho_n^{1/n} (f,A) = \inf_K \exp\big\{-2/\cp(A,K)\big\},
\end{equation}
where the infimum is taken over all compact sets \( K \) such that \( f \) admits a single-valued  analytic continuation to \( \overline\C\setminus K \).

Hereafter, we let \( T \subset\C \) be a Jordan curve with interior domain \( D \), and we put \( A:=\overline\C\setminus D \). In this setting \(\infty\in A\), which is no loss of generality for a preliminary M\"obius transform can ensure this; in contrast, our requirement that $D$ be a Jordan domain  is a regularity assumption on the set $A$  where  approximation takes place. Given \( f\in\mathcal H(A) \),  let \( \mathcal K_f \) be the collection of all compact sets \( K \subset D\) such that \(f\), initially defined on $A$, admits a single-valued analytic continuation to \( \C\setminus K\). It follows from \cite{St85a,St85b} that there exists \( \K \in \mathcal K_f \), unique up to  addition and/or removal of a polar set, with
\begin{equation}
\label{mincapset}
\cp_D(\K) \leq \cp_D(K), \quad K\in\mathcal K_f.
\end{equation}
We  can and will normalize $K_f$ to be the smallest possible, i.e., we make it  the intersection of all $K\in\mathcal{K}_f$ for which \( \cp_D(K) \) is minimal, see \cite{St85b}. As \( \cp(K,A) = \cp_D(K) \), in light of equation \eqref{bestrate}, our main goal is to investigate the asymptotic behavior of sequences \( \{r_n\} \) of rational functions of  type \( (n,n) \) meeting this optimal $n$-th root rate:
\begin{equation}
\label{optimal}
\lim_{n\to\infty} \|f-r_n\|_A^{1/n} = \exp\big\{-2/\cp_D(\K)\big\}.
\end{equation}
We call any such  sequence \( \{r_n\} \) a sequence of \emph{$n$-th root optimal rational approximants to \( f\) on~\( A \)}. In order to study  \( \{r_n\} \), we are led to consider more general sequences of \emph{meromorphic}  approximants of the form $r_n+h_n$, where $h_n$ is holomorphic in \(D\) and continuous on \(\overline{D}\), see Section~\ref{sec:meroa} for the definitions. Even though best meromorphic approximants may look less natural than rational ones, they make contact with  both the spectral theory of Hankel  operators (through AAK theory) and  Green potentials (because they generate errors with constant modulus on $T$), while remaining essentially equivalent to rational approximants as far as $n$-th root error rates are concerned \cite{Par86}. This is why $n$-th root optimal meromorphic approximants (meeting \eqref{optimal} in place of $r_n$) are of principal importance in  our study. Yet, the potential-theoretic tools on Riemann surfaces that we use  only allow us to handle compact surfaces so far, and this induces  some finiteness conditions on the functions from the class \( \mathcal S(A) \) that we can deal with. These are set forth in the next section.

\subsection{Class of Approximated Functions}
\label{laclasse}

We  consider functions in  $\mathcal{H}(A)$  such that
\begin{itemize}
\item[(i)] they can be continued into \( D \) along any path originating on \( T \) that stays in \( \overline D \)  while avoiding a closed polar subset of \( D \) (which may depend on the function);
\item[(ii)] they are not single-valued, meaning  there are continuations along at least two paths as in (i) with the same initial and terminal points that lead to distinct function elements, but still they are finite-valued in that the number of such function elements lying above a point of $D$  is uniformly bounded (the bound  may depend on the function);
\item[(iii)] their number of \emph{branchpoints} (points in any neighborhood of which some analytic continuation along a closed path in \( D \) encircling that point while avoiding the exceptional polar set leads to a different function element) is finite.
\end{itemize}
       
In view of (i) and (ii), such functions lie in $\mathcal{S}(A)$.  Note that (iii) is not superfluous, for there are functions meeting (i) and (ii) with infinitely many branchpoints. For instance, an open Riemann surface $\mathcal{X}$ made of two copies of \(\overline{\C}\setminus\{0\}\), glued along a sequence of disjoint cuts in \(\D\)  shrinking to the point $0$, has projection \(p:\mathcal{X}\to\overline{\C}\setminus\{0\}\) a two sheeted covering with infinitely many branchpoints of order $2$. As $\mathcal{X}$  carries a holomorphic function $f$ assuming more than one value on $p^{-1}(z)$ for $z$ not a critical value of $p$ \cite[Theorem 26.7]{Forster}, we deduce  on putting $D=\D$ and $A=\overline{\C}\setminus\D$ that each branch of $f\circ p^{-1}$ is of the announced type.

We formalize (i), (ii) and (iii) as follows. Let \( \RS_* \) be an auxiliary algebraic Riemann surface, whose set of ramification points \( \mathbf{rp}(\RS_*) \) lies on top of \( D \). That is, there exists an irreducible polynomial in two variables \( P(z,a) \), of degree at least 2 in \(a\),  such that \( \RS_* = \{ (z,a): P(z,a)=0 \} \) and all branchpoints of the algebraic function \( a(z) \) lie in \( D \). We denote by \( p:\RS_*\to\overline\C \) the natural projection \( p((z,a)) = z \), and we let $\RS\subset \RS_*$ be the (open) Riemann surface defined as
\[
\RS := \{ z\in\RS_*: p(z) \in D \};
\]
here and below, whenever it causes no confusion, we use letter \( z \) to denote both points in \( \C \) and on \( \RS_* \). Clearly, the ramification points of \(\RS_*\) and \(\RS\) are identical: \( \mathbf{rp}(\RS_*) = \mathbf{rp}(\RS)\). Let us denote by \( \overline\RS \) the closure of \( \RS \) in \( \RS_* \), and define a class of functions \(\mathcal F(\RS) \) by
\begin{equation}
\label{defFs}
\begin{array}{ll}
\mathcal F(\RS) := \big\{ f: & f~\text{is holomorphic and single-valued on}~\overline\RS\setminus E_f,\smallskip\\
& E_f~\text{ is closed},~p(E_f)~\text{is polar and contained in}~D,\smallskip\\
& f(z_1)\neq f(z_2)~\text{for some}~z_1,z_2~\text{with}~p(z_1)=p(z_2)\big\}.
\end{array}
\end{equation}
In \eqref{defFs}, we wrote \( E_f \) for the singular set of \( f \) on \( \RS \) but it would have been more appropriate to write \( E_f(\RS) \), as the complete Riemann surface of \( f \) could be significantly larger than \( \RS \) and its singular set bigger than \(E_f\). Since all ramification points of \( \RS_* \) lie on top of \( D \), the simple-connectedness of $A$  implies that the preimage \( p^{-1}(T) \) consists of finitely many homeomorphic copies of \( T \) under $p^{-1}$; we generically denote by \( \B \) such a copy, so that \( p:\B\to T \) is a homeomorphism. Then, denoting with a subscript $\mathcal{b}E$ the restriction to a set $E$, the class of functions that we study is defined as
\begin{equation}
\label{defFge}
\mathcal F(A) := \bigg\{ f:~f~\text{is holomorphic on}~A~ \text{and} ~\exists\,\RS_*,\,\B,\,\widehat f\in\mathcal F(\RS)~ \text{with} ~f_{\mathcal bT}=\widehat f\circ (p_{\mathcal b\B})^{-1}  \bigg\}.
\end{equation}
From  \eqref{defFs} and \eqref{defFge}, one sees  that \( \mathcal F(A)\subset\mathcal S(A)\subset\mathcal H(A) \) and members of \( \mathcal F(A)\) meet (i), (ii), (iii). Conversely, when $f\in\mathcal{H}(A)$ satisfies (i), (ii) and (iii), one can check that \(f\in \mathcal F(A)\). Indeed, if $B$ is the closed polar subset of $D$ defined by  (i), we get from (ii) because $D\setminus B$ is connected, see Section~\ref{ssec_fine}, that the number of sheets of the Riemann surface of $f$ above $D\setminus B$ is a finite constant, say \(M\). Therefore, since the branchpoints are finitely many by (iii), the algebraic surface $\mathcal{R}_*$ can be constructed by a classical glueing process described in Section~\ref{ssec_reduc}. The fine point, when applying to the present case this familiar procedure based on glueing pairwise in a certain order the banks of $M$ copies of a system of cuts joining the branchpoints, is that any two points of $D$ can be joined by a smooth simple arc entirely contained in \(D\setminus B\), except  for its endpoints if they lie in \(B\). It is so because $D\setminus B$ is a connected open set and each point of $B$ is the center of a circle of arbitrary small radius contained in \(D\setminus B\), as well as the endpoint of a segment contained in \(D\setminus B\) (that may even be chosen to have quasi-any direction). These properties hold because \(B\) is polar, and therefore thin at each point of $\C$, see Section~\ref{ssec_thin}.

We also consider functions in \(\mathcal{H}(A)\) meeting (i) but not (ii). These are analytic and single-valued in \(\overline{\C}\setminus E\), where $E\subset D$ is  closed and polar, i.e., there are no branchpoints. This case complements the previous one on putting  $\mathcal R_*=\overline{\C}$ and omitting the last requirement in \eqref{defFs}; we denote the corresponding class of functions by \(\mathcal E(A)\). Since  \(\cp(A,E)=\cp_D(E)=0\) when $E\subset D$ is polar, see Section~\ref{ssec_cap}, we get  from  \eqref{prokhorov} and \eqref{walsh} that
\begin{equation}
\label{ftg}
\lim_{n\to\infty}\rho_n^{1/n}(f,A)=0,\qquad f\in\mathcal E(A).
\end{equation}
That is to say, some sequence $\{r_n\}_{n\in\N}$, $r_n\in\mathcal{R}_n$, of rational functions, converges \emph{faster than geometrically}  towards \(f\) on \(A\) as the degree grows large, meaning that
\begin{equation}
\label{optimalE}
\lim_{n\to\infty} \|f-r_n\|_A^{1/n} = 0.
\end{equation}
We call any such sequence a sequence of \emph{$n$-th root optimal rational approximants to \( f\in\mathcal{E}(A)\) on~\( A \)}.

\subsection{Optimal Rational Approximants}
\label{ssec_optimal}

Notions of potential theory in \( D \) and \( \RS \) play a fundamental role in what follows, and the  reader  might want to consult Appendix~\ref{sec_app} for a comprehensive account thereof. Let us here recall the definition of Green potentials and Green equilibrium distributions. 

The \emph{Green function}  \( g_D(\cdot,w) \) of the domain \( D \) with pole at \( w \in D \) is the unique non-negative harmonic function in \( D\setminus\{w\} \), with logarithmic singularity at \( w \), whose largest harmonic minorant is zero. The \emph{Green potential} of a positive Borel measure \(\nu\)   in \( D \) is  \( g(\nu,D;z) := \int g_D(z,w)d\nu(w) \). Putting \(|\nu|\) for the total mass of \(\nu\), the \emph{Greenian capacity}  (in $D$)  of a Borel set  \( B \subset D \) is defined by
\begin{equation}
  \label{dcG}
\cp_D(B) := \bigg( \inf_{\nu\geq0,|\nu|=1,\supp\,\nu\subseteq B} I_D(\nu) \bigg)^{-1},  \quad I_D(\nu) := \int\int g(z,w)d\nu(w)d\nu(z);
\end{equation}
the infimum above is taken over all probability Borel measures $\nu$ supported on \( B \). For any set \( S \subset D\), the \emph{outer Greenian capacity} of \(S\) in \(D\) is defined as \( \cp_D(S) = \inf \cp_D(U) \), where the infimum is taken over all open sets \( U \supset S \) in \(D\) (using again the symbol \(\cp_D\) causes no confusion for the outer Greenian capacity is known to coincide with the Greenian capacity on Borel sets). Polar subsets of \( D \) are those whose outer Greenian capacity is \( 0 \). If \( K \) is a  non-polar compact subset of \( D \), then there exists a unique Borel probability measure \( \mu_{D,K} \) supported on $K$, called the \emph{Green equilibrium distribution on \( K \) relative to \( D \)}, such that \( \cp_D(K) = 1/ I_D(\mu_{D,K})  \). It is characterized by the property that its Green potential is bounded on \( D \) and equal to its maximum (which is then necessarily  \( 1/\cp_D(K) \)) \emph{quasi everywhere}  (that is, up to a polar set) on \( K \). To recap:
\begin{equation}
\label{GreenEqPot}
g(\mu_{D,K},D;z) \begin{cases}
\leq 1/\cp_D(K), & z\in D, \\
= 1/\cp_D(K), & \text{for q.e. } \; z\in K, \\
< 1/\cp_D(K), & z\in D\setminus K,
\end{cases}
\end{equation}
where the last inequality follows from the (generalized) maximum principle for harmonic functions. 

We are concerned with two types of asymptotics for sequences  of rational approximants: the weak$^*$ behavior of the normalized counting measures of their poles, and the convergence in capacity of the functions themselves. More precisely, given a rational function \( r \)  of type \( (n,n) \), we define
\begin{equation}
\label{count_meas}
\mu(r) := \frac1n\sum_{z:r(z)=\infty} \delta_z,
\end{equation}
where each pole \( z \) appears in the sum as many times as its order. Equivalently $-2\pi\mu(r)$ is the distributional Laplacian $\Delta(\log|q|^{1/n})$, with $q$  the denominator of an irreducible form of $r$. One  says that a sequence \( \{\nu_n\} \) of finite Borel measures on \( D \) converges \emph{weak$^*$} to a measure \( \nu \),  denoted as
\[
\nu_n \overset{w*}{\to} \nu \qasq n\to\infty,
\]
if \( \int hd\nu_n \to \int hd\nu \) for every continuous compactly supported function \( h \) on \( D \). We further say that a sequence of functions \( h_n \) converges \emph{in Greenian capacity} to a function \( h \)  on a set \( U \subset D\) if
\begin{equation}
  \label{defconvcapG}
\lim_{n\to\infty} \cp_D\big( \{z\in F:|h_n(z)-h(z)|>\epsilon\} \big) = 0
\end{equation}
for each  \( \epsilon>0 \) and every compact \( F\subset U \); we denote this claim by \( h_n \overset{\mathrm{cap}}{\to} h \) in \( U \). The convergence \eqref{defconvcapG} is said to hold \emph{at a geometric  rate} if we can replace \( \epsilon \) in \eqref{defconvcapG} by \( a^n \) for some positive $a=a(F)<1$; the convergence rate is called \emph{faster than geometric} if \eqref{defconvcapG} holds with \( \epsilon \) replaced by \( a^n \) for any $a>0$.

\begin{theorem}
\label{thm:main1}
Let \( T\subset\C\) be a Jordan curve, \(D\) its interior and \( A \) the complement of \(D\cup T\) on the Riemann sphere. Given \( f\in\mathcal F(A) \), let \( \K\in\mathcal K_f \) be such that \eqref{mincapset} holds and \( \{r_n\} \) be  a sequence of \( n \)-th root optimal rational approximants to \( f \) on \( A \), as defined in \eqref{optimal}. Then,
\begin{equation}
  \label{wcopg}
\mu(r_n) \overset{w*}{\to} \mu_{D,\K},
\end{equation}
where \( \mu(r_n) \) is the normalized counting measure of the poles of \( r_n \) and \( \mu_{D,\K} \) is the Green equilibrium distribution on \( \K \) relative to \( D \). Furthermore, it holds that
\begin{equation}
  \label{cvicr}
\frac1{2n} \log|f-r_n| \overset{\mathrm{cap}}{\to} g(\mu_{D,\K},D;\cdot) - \frac1{\cp_D(\K)} \quad \text{in} \quad D\setminus \K\quad \text{as}\quad  n\to\infty ,
\end{equation}
where \( g(\mu_{D,\K},D;\cdot) \) is the Green potential of \( \mu_{D,\K} \) in \( D \).
\end{theorem}

\begin{remark}
 \label{convlogcapv}
\emph{Theorem~\ref{thm:main1} and equation \eqref{GreenEqPot} imply that $n$-th root optimal rational approximants converge to \( f \) in capacity in \( D\setminus \K \), at a geometric rate equal to $\exp\{g(\mu_{D,\K},D;\cdot)-1/\cp_D(\K)\}$  pointwise. In fact, as shown in Section~\ref{ssec:conv_cap}, convergence in Greenian capacity in \( D\setminus \K \) and uniform convergence on \( A \), along with the limiting behavior \eqref{wcopg} for the poles, together imply the seemingly stronger convergence in logarithmic capacity on \( \overline D\setminus \K \), at a geometric rate which is pointwise no slower than the previous one. Here, convergence in logarithmic capacity is defined as in \eqref{defconvcapG}, except that the logarithmic capacity replaces the Greenian one, see Section~\ref{ssec_cap} for a definition of logarithmic capacity. This remark equally applies to the forthcoming Theorem~\ref{thm:main2}, as well as to Theorems~\ref{thm:main1p} and~\ref{thm:main2p} in which \( \K \) gets replaced by a polar set \( E \) and convergence takes place at  faster than geometric rate. Moreover, if in Theorem~\ref{thm:main2} the error $f-M_n$ is asymptotically circular on $T$, namely if $\lim_n\inf_T|f-M_n|^{1/n}=  \exp\{-2/\cp_D(K_f)\}$ (in particular when $M_n$ are their own Nehari modifications, see Section~\ref{ssec_nehari}), then one can draw  the more  precise conclusion that the exact pointwise rate of geometric convergence of $M_n$ to $f$ in logarithmic capacity on \( \overline D\setminus \K \) is the same as in Greenian capacity  on \( D\setminus \K \), namely  $\exp\{g(\mu_{D,\K},D;\cdot)-1/\cp_D(\K)\}$.}
\end{remark}

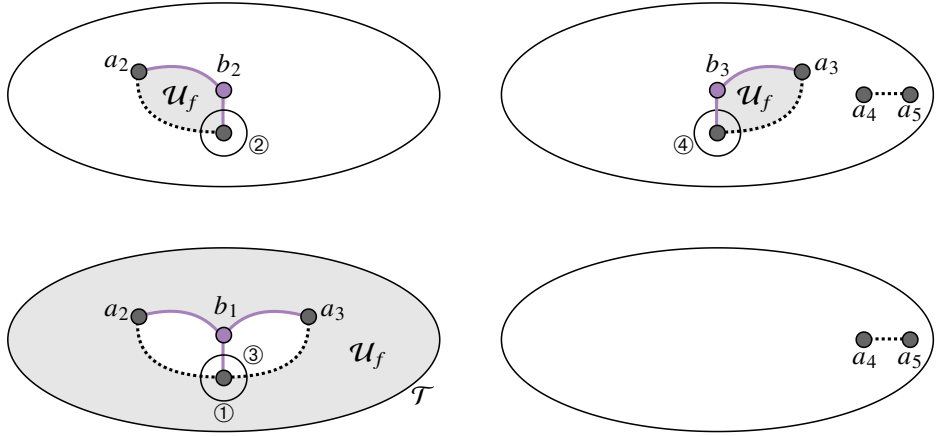
\begin{figure}[ht!]

\includegraphics[scale=.8]{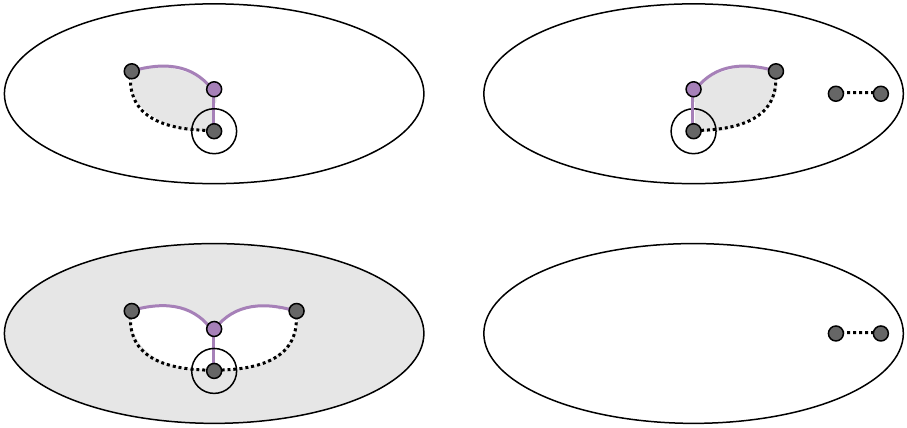}

\begin{picture}(0,0)
\put(-96,17){\ding{192}}
\put(-83,118){\ding{193}}
\put(-85,39){\ding{194}}
\put(76,118){\ding{195}}
\put(-137,57){\(a_2\)}
\put(-56,57){\(a_3\)}
\put(-137,151){\(a_2\)}
\put(129,150){\(a_3\)}
\put(143,132){\(a_4\)}
\put(160,132){\(a_5\)}
\put(143,39){\(a_4\)}
\put(160,39){\(a_5\)}
\put(-96,58){\(b_1\)}
\put(-95,149){\(b_2\)}
\put(88,149){\(b_3\)}
\put(-22,24){\(\B\)}
\put(-45,40){\(\mathcal U_f\)}
\put(100,137){\(\mathcal U_f\)}
\put(-115,137){\(\mathcal U_f\)}
\end{picture}

\caption{\small Surface \( \RS \) with five ramification points \( \mathbf{rp}(\RS) = \{a_1,a_2,a_3,a_4,a_5\}\) (\(a_1\) is not labeled) and four sheets. Ramification point \( a_1 \) has order \( 3 \) and the remaining points have order \( 2 \). The sequence \ding{192}-\ding{193}-\ding{194}-\ding{195}-\ding{192} represents the monodromy around \( a_1 \) while encircling it clockwise, where the transitions happen across the dashed curves that stand for the cuts between different sheets of \( \RS \). The  domain \( \mathcal U_f \) is depicted as a shaded region. In this example \( E_0 = \varnothing \), \( E_{11} = \{ p(a_1),p(a_2),p(a_3) \} \) are the active branch points, and \( E_{10} = \{ b \} \), where \( b=p(b_1)=p(b_2)=p(b_3) \). The set \( \K \) is a threefold equal to the natural projection of the solid (purple) curves. The solid curves also represent a different choice of the transition cuts between different sheets of \( \RS \), with the latter the domain \( \mathcal U_f \) will lie entirely on one sheet.}
\label{fig:RS}
\end{figure} 

In view of its importance, let us describe in greater detail the set \( \K \). As mentioned before, the problem of finding elements of minimal capacity in \( \mathcal K_f \) was extensively studied by the second author \cite{St85a,St85b,St85c}. The existence of such sets for \( f\in \mathcal H(A) \) was proven in \cite{St85a}, and  one can choose \( \K \) so that \( \K\subseteq \tilde K \) for any \( \tilde K\in\mathcal K_f \) with \( \cp_D(\tilde K)=\inf_K \cp_D(K) \), which makes \( \K \) unique \cite{St85b}. The topological structure of \( \K \) was investigated in \cite{St85c}, where it is shown  that
\begin{equation}
\label{Kf}
\K = E_0 \cup E_1 \cup \bigcup_i J_i
\end{equation}
where the \( J_i \) are open analytic arcs, \( E_1 \)  comprises the endpoints of the arcs \( J_i \), and \( E_0 \) is a subset of the singular set of \( f \) in \( D \) (the singular set consists of those points in \( D \) to which some continuation of \( f \) from \( T \) has a singularity). As soon as \( f\in \mathcal S(A) \), the set \( E_0 \) is polar by definition. To understand this decomposition better when \( f\in\mathcal F(A) \), let \( J_f := \cup_i \overline J_i \) so that \( E_1\subset J_f \). Then \( f\) possesses a single-valued continuation into \( D\setminus J_f \) with singular set \( E_0 \), consisting  of polar and essential singularities (but no branching singularities). If \( \B \), \( \widehat f \) are as in \eqref{defFge} and \(  \mathcal U_f \) is the connected component of \( p^{-1}(D\setminus J_f) \) containing \( \B \) in its boundary, then we can further decompose:
\[
E_0 = p(E_{\widehat f}\cap \mathcal U_f), \quad E_1 = E_{10}\cup E_{11}, \quad E_{11} := p\big(\mathbf{rp}(\RS)\cap\overline {\mathcal U}_f\big), \quad E_{10} := E_1\setminus E_{11}.
\]
That is, \( E_{11} \) is the set of ``active'' branchpoints of \( f \), i.e., branchpoints that can be  reached by continuation of \( f \) from \( T \) into \( D\setminus J_f \), as opposed to those points in \(\mathbf{rp}(\RS) \) that cannot be so reached. Also, each \( e\in E_{10} \) is an endpoint of at least three arcs \( J_i \), and generically \( f \) possesses analytic continuations to \( e \) from any direction within \( D\setminus J_f \) (unless by chance \( e \) is a singularity of \( f \) as well), see Figure~\ref{fig:RS}. As \( \mathbf{rp}(\RS) \) is finite,  so is the collection \( \{ J_i \} \) in cases that we consider.

For $f\in\mathcal{F}(A)$, Theorem \ref{thm:main1} asserts two things: (i) the weak$^*$ convergence of $\mu(r_n)$ to $\mu_{D,K_f}$ whenever $\{r_n\}$ is an $n$-th root optimal sequence of rational approximants to $f$ on $A$, and (ii) the convergence in capacity of $r_n$ to \(f\)  at a geometric rate on $D\setminus K_f$. If now $f\in\mathcal{E}(A)$ and \( \{r_n\} \) is a sequence of $n$-th root optimal rational approximants to $f$ on $A$, i.e., a sequence  meeting \eqref{optimalE} and thus converging faster than geometrically to $f$ on $A$, then we shall see that  \(r_n\) converges in capacity  to $f$ at faster than geometric rate in $D$. However, one can no longer expect a specific behavior of $\mu(r_n)$ in this case, for if $R_n$ is a sequence in $\mathcal{R}_n(D)$ that converges to zero faster than geometrically on $T$, then $r_{n/2}+R_{n/2}$ is again a sequence of  $n$-th root optimal rational approximants to $f$ on $A$, whereas the weak$^*$ limit points of $\mu(R_n)$ can be arbitrary amongst positive measures of mass at most $1$ on $D$. Thus, faster than geometric convergence does not qualify rational approximants enough to imply much on the behavior of their poles. Still, those poles of $r_n$ that stay away from the singular set of $f$ cannot account for the rate of convergence. This is made precise in the following result, which  complements Theorem \ref{thm:main1} in the case of no branchpoints.

\begin{theorem}
\label{thm:main1p}
Let \( T \), \(A\) and $D$ be as in Theorem~\ref{thm:main1}. Given \( f\in\mathcal E(A) \), let \( \{r_n\} \) be a sequence of rational functions of type \((n,n) \) meeting \eqref{optimalE}. Then, it holds that
\[
r_n \overset{\mathrm{cap}}{\to} f \quad \text{in} \quad D \setminus E
\]
at faster than geometric rate, where $E\subset D$ a closed polar set outside of which $f$ is analytic and single-valued. Moreover, for any neighborhood \( V \) of \( E \) there is a sequence  $\{R_{k_n}\}$, $R_{k_n}\in\mathcal{R}_{k_n}(V)$, $k_n\leq n$, such that the poles of $R_{k_n}$ are among the poles of $r_n$ lying in $V$ and
 \begin{equation}
\label{optimalEs}
\lim_{n\to\infty} \|f-R_{k_n}\|_A^{1/n} = 0.
\end{equation}
\end{theorem}

Using  neighborhoods $V$ shrinking to $E$ and  a diagonal argument, Theorem~\ref{thm:main1p} yields a corollary of independent interest.

\begin{corollary}
If \( f\in\mathcal E(A) \) and $E\subset D$ a closed polar set outside of which $f$ is analytic and single-valued, then there is a sequence of rational functions $\{r_n\}$, $r_n\in\mathcal{R}_n(D)$, converging faster than geometrically to $f$ on $A$, and such that every weak$^*$ cluster point of the sequence $\{\mu(r_n)\}$ of normalized counting measures of the poles is supported on $E$.
\end{corollary}

\subsection{Optimal Meromorphic Approximants}
\label{sec:meroa}

Let \( H^\infty(D) \)  denote the space of bounded analytic functions on \( D \) and \(\mathcal{A}(D)\) the subspace of those extending continuously to $\overline{D}$.  When $T$ is rectifiable each $h\in H^\infty(D)$ has a non-tangential limit   almost everywhere on $T$ with respect to arclength, that we still call $h$, and putting \( \|\cdot\|_T \) for the essential supremum norm on \( T \) (with respect to arclength) it holds that $\|h\|_T=\|h\|_D$ \cite[Theorems 10.3 \& 10.5]{Duren}. When $T$ is a non-rectifiable Jordan curve, however, limiting values of $H^\infty(D)$-functions on \(T\) generally exist at sectorially accessible points only, and such points may reduce to a set of zero linear measure \cite{McMPi73}. This will force us into a somewhat careful discussion of meromorphic approximants.   Remember the set \( \mathcal M_n(D) \)  of monic polynomials of degree \( n \) whose zeros belong to \( D \), and put
\[
\begin{cases}
H_n^\infty(D) & := \big\{h/q:~h\in H^\infty(D),\; q\in\mathcal M_n(D) \big\}, \smallskip \\
\mathcal{A}_n(D) &:= \big\{h/q:~h\in\mathcal{A}(D),\; q\in\mathcal M_n(D) \big\}.
\end{cases}
\]
That is, \( H_n^\infty(D) \) is the set of meromorphic function with at most \( n \) poles in \( D \) that are bounded near \( T \), and \(\mathcal{A}_n(D)\) is the subset of those extending continuously to $T$. We shall say that \( \{M_n\} \) is a sequence of \emph{\( n \)-th root optimal meromorphic approximants to \( f\in\mathcal S(A) \) on \( T \)} if \( M_n\in \mathcal{A}_n(D) \) and
\begin{equation}
\label{opt_merom}
\limsup_{n\to\infty} \|f-M_n\|_T^{1/n} \leq \exp\big\{-2/\cp_D(\K)\big\}.
\end{equation} 
Any sequence of \( n \)-th root optimal rational approximants is  a particular sequence of \( n \)-th root optimal meromorphic ones, by  \eqref{optimal}. When $T$ is rectifiable, Corollary~\ref{cor:main1} to come will entail that in \eqref{opt_merom} the condition $M_n\in\mathcal{A}_n(D)$ can be traded for the  seemingly weaker requirement  $M_n\in H^\infty_n(D)$. However, it is important to require that $M_n\in\mathcal{A}_n(D)$ when $T$ is not rectifiable, for otherwise the left-hand side of \eqref{opt_merom} may no longer make sense for the essential supremum  norm (with respect to arclength) and the considerations below would not apply. In his proof of \eqref{prokhorov}, Parf\"enov \cite{Par86} has shown that the limit inferior of \( \rho_n^{1/n}(f,A) \) is the same as the one of the \( n \)-th root of the error in best meromorphic approximation, see \eqref{nehari_takagiT} for a definition of best (not just $n$-th root optimal) meromorphic approximants. Hence, in light of \eqref{optimal}, replacing rational approximants with meromorphic ones does not improve the rate of decay of the \( n \)-th root error,  and  \( n \)-th root optimal meromorphic approximants to functions in \( \mathcal S(A) \) could be defined with the limit superior and the inequality sign replaced by a full limit and the equality sign in \eqref{opt_merom}. 

\begin{theorem}
\label{thm:main2}
With  the notation of Theorem~\ref{thm:main1},  let \( \{M_n\} \) be a sequence of \( n \)-th root optimal meromorphic approximants to\( f\in\mathcal F(A) \) on \( T \), see \eqref{opt_merom}. Then, the conclusions of Theorem~\ref{thm:main1} hold with \( r_n \) replaced by \( M_n \).
\end{theorem}

Let us reiterate that $n$-th root optimal rational approximants are just an instance of $n$-th root optimal meromorphic approximants, and therefore Theorem~\ref{thm:main1} is a special case of Theorem~\ref{thm:main2}. 

When $f\in\mathcal{E}(A)$, we define meromorphic approximants in a manner similar to \eqref{opt_merom} but this time with an eye on \eqref{ftg}. Namely, we say that \( \{M_n\} \) is a sequence of \emph{\( n \)-th root optimal meromorphic approximants to \(f\in \mathcal{E}(A)\) on \( T \)} if \( M_n\in \mathcal{A}_n(D) \) and
\begin{equation}
\label{opt_meromp}
\lim_{n\to\infty} \|f-M_n\|_T^{1/n} =0.
\end{equation} 
The following complements Theorem \ref{thm:main2} in the case of no branchpoints, and subsumes Theorem~\ref{thm:main1p}.

\begin{theorem}
\label{thm:main2p}
With  the notation of Theorem~\ref{thm:main1p},  let \( \{M_n\} \) be a sequence of \( n \)-th root optimal meromorphic approximants on \( T \) to \( f\in\mathcal E(A) \), see \eqref{opt_meromp}. Then, the conclusions of Theorem~\ref{thm:main1p} hold with \( r_n \) replaced by \( M_n \).
\end{theorem}

Besides best rational approximants, a noteworthy  instance of $n$-th root optimal meromorphic approximants are the AAK (short for Adamyan-Arov-Krein) approximants. These are best meromorphic approximants with at most $n$ poles. Specifically, consider the following (Nehari-Takagi) problem: given \(f\in L^\infty(\T)\), find \( M_n^\infty \in H_n^\infty(\D) \) such that
\begin{equation}
\label{nehari_takagi}
\|f-M_n^\infty\|_\T = \inf_{M\in H_n^\infty(\D)}\|f-M\|_\T.
\end{equation}
If \( n=0 \), then \eqref{nehari_takagi} reduces to the question of best analytic approximation of bounded functions on the unit circle by elements of \( H^\infty(\D) \), which is the so-called \emph{Nehari problem} named after \cite{Neh57} (that deals with an equivalent issue). It is known that \( M_n^\infty \)  always exists and that it is unique when \( f \) lies in $C(\T)+H^\infty(\D)$, see \cite[Chapter~4]{Peller}. Moreover, if $f$ is Dini-continuous on $\T$, then $M_n^\infty$ is continuous on $\T$. Indeed, if we write $M_n^\infty=r_n+g$ where $r_n\in\mathcal{R}_n(\D)$ and $g\in H^\infty(\D)$, we see that $g$ must be the best Nehari approximant to $f-r_n$, which is Dini-continuous on $\T$, and the claim follows from \cite{CarJac72}. When \( T \) is a rectifiable Jordan curve, one can readily replace \( \D \) by \( D \) in \eqref{nehari_takagi} and carry over to $T$ all the properties of  best meromorphic approximants on \(\T\) by conformal mapping.  When \( T \) is non-rectifiable, the very existence of approximants depends on the analyticity of $f$ on $T$, and follows from the  proof of the next corollary to Theorem~\ref{thm:main2}.

\begin{corollary}
\label{cor:main1}
Let \( T \), \(A\) and $D$ be as in Theorem \ref{thm:main1}. Given \( f\in\mathcal F(A) \) or \(\mathcal{E}(A)\), to each integer \(n \) there exists a unique \( M_n^\infty \in \mathcal A_n(D) \) such that
\begin{equation}
\label{nehari_takagiT}
\|f-M_n^\infty\|_T = \inf_{M\in \mathcal A_n(D)}\|f-M\|_T,
\end{equation}
and if \( T \) is rectifiable  then \( \mathcal A_n(D) \) can be replaced by \( H_n^\infty(D) \) in \eqref{nehari_takagiT} without changing \( M_n^\infty \). Of necessity, the conclusions of Theorem~\ref{thm:main2} and~\ref{thm:main2p}  hold with \( M_n \) replaced by \( M_n^\infty \).
\end{corollary}

Much interest in best meromorphic approximants stems from their striking connection to operator theory. Denote by \( L^2(\T) \) the space of square integrable functions on \( \T \), and let \( H^2\subset L^2(\T) \) be the Hardy space of functions whose Fourier coefficients with negative index do vanish. It is known that \( H^2 \) can be identified with  (non-tangential limits  on $\T$ of) analytic functions in \( \D \) whose $L^2$-means on circles centered at the origin are uniformly bounded,  see \cite[Theorem 3.4]{Duren}. Set \( H^2_- := L^2(\T)\ominus H^2 \) to be the orthogonal complement of \( H^2 \), which is the Hardy space of \(L^2\)-functions whose Fourier coefficients with nonnegative index are equal to zero. The latter can be identified with functions analytic in $\overline{\C}\setminus\overline{\D}$ that vanish at infinity, and whose $L^2$-means with respect to normalized arclength on circles centered at the origin are uniformly bounded. Let  $\mathbb{P}_-:L^2(\T)\to H^2_-$ be the orthogonal projection. Given \( f\in L^\infty(\T) \), one defines the \emph{Hankel operator with symbol} \( f \) to be
\begin{equation}
\label{Hankel}
\Gamma_f :~H^2 \to H^2_-, \quad \quad \Gamma_f(g) := \mathbb P_-(gf).
\end{equation}
For \(n\)  a non-negative integer, let \( s_n(\Gamma_f) \) be the \( (n+1) \)-th singular number of the operator \( \Gamma_f \), that is
\( s_n(\Gamma_f):=\inf_{\mathrm{rank}\, R\leq n}\|\Gamma_f-R\|\), where the infimum is taken over all operators $R:H^2\to  H^2_-$ of rank at most $n$ and $\|\cdot\|$ stands for the operator norm. Then, one has that
\begin{equation}
\label{arsv}
\|f-M_n^\infty\|_\T = s_n(\Gamma_f).
\end{equation}
If, in addition, $f\in C(\T)+H^\infty(\D)$, then $\Gamma_f$ is compact so that  $s_n^2(\Gamma_f)$ is the $(n+1)$-st eigenvalue of  $\Gamma_f^*\Gamma_f$ when these are arranged in non-increasing order, and \eqref{arsv} becomes a pointwise equality:
\begin{equation}
\label{circ_error2}
|(f-M_n^\infty)(z)| = s_n(\Gamma_f)\qquad \text{a.e.\ on}\ \T.
\end{equation}
Moreover, if $v_n$ is an $(n+1)$-st singular vector of $\Gamma_f$, i.e., an eigenvector of $\Gamma_f^*\Gamma_f$ with eigenvalue $s_n^2(\Gamma_f)$, then \(M_n^\infty\) is explicitly given in terms of $f$ and $v_n$ by the formula
\begin{equation}
\label{forerreur}
f-M_n^\infty=\frac{\Gamma_f(v_n)}{v_n}.
\end{equation}
Though not obvious at first glance, the right-hand side of \eqref{forerreur} is independent of which $n$-th singular vector $v_n$  is chosen and it has constant modulus on $\T$, in accordance with \eqref{circ_error2}. The next corollary, of independent interest, follows from \eqref{circ_error2}, \eqref{opt_merom}, \eqref{ftg}, and  Theorems~\ref{thm:main2}--\ref{thm:main2p}.

\begin{corollary}
\label{cor:main2}
Let \( \Gamma_f \) be the Hankel operator with symbol \( f \in\mathcal F\big(\overline\C\setminus \D\big) \). Then, it holds that
\[
\lim_{n\to\infty} s_n^{1/n}(\Gamma_f) = \exp\big\{-2/\cp_\D(\K)\big\}.
\]
Moreover, if \(f\in\mathcal{E}(\overline\C\setminus \D)\) , then the above limit is equal to \( 0 \).
\end{corollary}

\section{Proof of Theorem~\ref{thm:main2}}
\label{TheProof}

\subsection{Existence of Best Meromorphic Approximants}

Below, we establish the existence and uniqueness part of Corollary~\ref{cor:main1}, along with the assertion that \( H_n^\infty(D) \) may replace \( \mathcal A_n(D) \)  when \( T \) is rectifiable. The rest of the corollary  will follow from Theorem~\ref{thm:main2} upon conclusion of its proof.

Let \( \phi:\D\to D \) be a conformal map. Since \( D \) is a Jordan domain, \( \phi \) extends to a homeomorphism from \( \overline\D \) to \( \overline D \) by Carath\'eodory theorem \cite[Section~2]{Pommerenke}. Let \( L^2(\T) \), \( H^2 \), and \( H^2_- \) be as defined after Corollary~\ref{cor:main1} and $\mathbb{P}_+:L^2(\T)\to H^2$, $\mathbb{P}_-:L^2(\T)\to H^2_-$ be the orthogonal projections. If we pick a continuous function $f$ on $T$, then \( f\circ\phi \) is continuous on \(\T\) and, in particular, it lies in \(L^2(\T) \). Set \( F:=\mathbb P_-(f\circ \phi) \) and \( G:=\mathbb P_+(f\circ\phi) \), so that \( F\in H^2_- \) and  \( G\in H^2 \). Since $\mathbb{P}_++\mathbb{P}_-$ is the identity operator,  it holds that
\begin{equation}
\label{reflectH2}
F(z) = (f\circ\phi)(z) - G(z), \quad \mathrm{a.e.\ }z\in\T.
\end{equation}
Now, if $f\in \mathcal{H}(A)$, then $f\circ\phi(z)$ is holomorphic in $r<|z|<1$ and continuous in $r\leq|z|\leq1$,  for \(r\) close enough to \( 1 \). Hence, the right-hand side of \eqref{reflectH2} is holomorphic in  $r<|z|<1$ with uniformly bounded $L^2$-means on circles centered at the origin, while the left-hand side lies in $H^2_-$ and both sides have the same non-tangential limit on \(\T\). By an easy variant of Morera's theorem \cite[Chapter II, Exercise 12]{Garnett}, the function equal to \(F(z)\) for $|z|>1$ and to $ (f\circ\phi)(z) - G(z)$ for $r<|z|<1$ is holomorphic across $\T$, in particular $F$ extends analytically across $\T$ and $G$ extends continuously to $\overline{\D}$.  Then, as described after \eqref{nehari_takagi}, the best meromorphic approximant \( M_n^\infty \in H_n^\infty(\D) \) to \( f\circ\phi \) exists and is unique, moreover it is readily checked that \( M_n^\infty \) is equal to the sum of \( G \) (a member of \(\mathcal{A}(\D)\)) and of the best approximant to  \( F \) from $H_n^\infty(\D)$,   which lies in $\mathcal{A}_n(\D)$ because $F$ is analytic across \( \T \) and so is {\it a fortiori} Dini-continuous on \( \T \); hence, we get that \( M_n^\infty\in\mathcal A_n(\D) \). If now \( M \in \mathcal A_n(D) \), then \( M\circ\phi \in \mathcal A_n(\D) \) and
\[
\|f-M\|_T = \|f\circ\phi - M\circ \phi\|_\T \geq \|f\circ\phi - M_n^\infty\|_\T
\]
by definition of \( M_n^\infty \). As $M_n^\infty\circ\phi^{-1}\in\mathcal A_n(D)$, it is the unique best meromorphic approximant to \( f \) we are looking for. The previous argument also shows that, when \( f\in\mathcal H(A) \), the best meromorphic approximant to $f\circ\phi$ necessarily belongs to \( \mathcal A_n(\D) \). Because composition with $\phi$ is an isometric isomorphism $L^\infty(T)\to L^\infty(\T)$ (understood with respect to arclength measure) when $T$ is rectifiable,  one can equivalently use \( H_n^\infty(D) \) instead of \( \mathcal A_n(D) \) in definition \eqref{nehari_takagiT}.

\subsection{Reduction to the Unit Disk}
\label{ssec_reduc}

Let \( K \) be a compact subset of \( \D \). Since Green potentials are non-negative superharmonic functions whose largest harmonic minorant is zero, while a characteristic property of Green equilibrium potentials is to be constant quasi everywhere on the support of their defining measure while being no greater than this constant everywhere in the domain, it holds that
\[
g(\mu_{\D,K},\D;z) = g(\mu_{D,\phi(K)},D;\phi(z)) \qandq \cp_\D(K) = \cp_D(\phi(K))
\]
where \( \phi:\D\to D \) is a conformal map, see \eqref{GreenEqPot} as well as Sections~\ref{ssec_pot} and~\ref{ssec_cap}. One has in this case that \( \mu_{D,\phi(K)}=\phi_*(\mu_{\D,K}) \), the pushforward of \(\mu_{\D,K}\) under $\phi$. Hence, we get that
\[
\frac1n \sum_{i=1}^n \delta_{z_i} \cws  \mu_{\D,K} \quad \text{if and only if} \quad \frac1n \sum_{i=1}^n \delta_{\phi(z_i)} \cws  \mu_{D,\phi(K)},
\]
where the weak$^*$ convergence is understood for \( n\to\infty \). Furthermore, by conformal invariance, a sequence of functions \( h_n \) converges in Greenian capacity to a function \( h \) in \( D \) if and only if the functions \( h_n\circ \phi \) converge in Greenian capacity to the function \( h\circ\phi \) in \( \D \).

Let now \( \RS_* \) and \( \RS \) be as in  Section~\ref{laclasse}. Denote by \( M \) the number of sheets of \( \RS  \), so that a generic point in \( D \) has \( M \) preimages under the natural projection \( p:\RS\to D\). Let \( J\subset D \) be a smooth oriented Jordan arc joining two points of \( p({\bf rp}(\RS)) \) while passing through all others exactly once. Constructing \( J \) is tantamount to enumerating the points of \( p({\bf rp}(\RS)) \). Write \( J=\cup_l J_l \) where each \( J_l \) connects exactly two points in  \( p({\bf rp}(\RS)) \). The surface \( \RS \) can be realized as \( M \) copies \( U_1,\ldots,U_M \) of \( D\setminus J \), suitably glued to each other along  the banks of the  cuts  \( J_l \) in each copy  \( U_i \) (the glueing  rule can be encoded, for example, as a collection of 4-tuples \( (l,i,j,k) \) telling one that the left and right banks of the cut \( J_l \) in \( U_i \) need to be glued to the right bank of the cut \( J_l \) in \( U_j \) and the left bank of the cut \( J_l \) in \( U_k \) respectively, see \cite{Magnus} for a discussion of Hurwitz's theorem). Using the same gluing rule, we can construct another  \( M \)-sheeted surface \( \mathcal S \) out of the domains \( \phi^{-1}(U_1),\ldots,\phi^{-1}(U_M) \). The map \( \phi \) can then be lifted to a conformal map \( \Phi:\mathcal S \to \RS \) for which \( \phi(\pi(z)) = p(\Phi(z)) \), \( z\in\mathcal S \), where \( \pi:\mathcal S\to \D \) is the natural projection.

Let \( f\in\mathcal F(A) \) and \( \widehat f \in \mathcal F(\RS) \) be as in \eqref{defFge} and \eqref{defFs}. Clearly \( \widehat f\circ\Phi \in\mathcal F(\mathcal S) \),  and the surface \( \mathcal S_* \) can be constructed by gluing \( M \) copies of \( \overline\C\setminus \D \) to \( \mathcal S\) along the \( M \) homeomorphic copies of \( \T \) that comprise the boundary of \( \mathcal S \). Argueing as we did after \eqref{reflectH2}, we find that \( F:=\mathbb P_-(f\circ \phi) \) lies in \(\mathcal F(\overline\C\setminus \D) \) with the corresponding \( \widehat F\in\mathcal F(\mathcal S) \) given by \( \widehat f\circ\Phi-G\circ\pi \), where \( G:=\mathbb P_+(f\circ\phi) \). Note that the conformal equivalence of Greenian capacities implies that \( K_F = \phi^{-1}(\K) \). Hence, if \( \{ M_n \} \) is an \( n \)-th root optimal sequence of meromorphic approximants of \( f \) in \( D \) as defined in \eqref{opt_merom}, then \( \{ \widetilde M_n:=M_n\circ\phi - G\} \) is an \( n \)-th root optimal sequence of meromorphic approximants to \( F \) on \( \T \) and
\begin{equation}
  \label{correspmaf}
(f-M_n)(\phi(z)) = (F-\widetilde M_n)(z), \quad z\in \overline\D\setminus K_F.
\end{equation}
Therefore, it is sufficient to study the asymptotic behavior of \( F-\widetilde M_n \) as well as the limit distribution of poles of \( \widetilde M_n \). That is, it is enough to prove Theorem~\ref{thm:main2} on the unit disk.

\subsection{Nehari Modifications}
\label{ssec_nehari}

For $f\in \mathcal{H}(\overline\C\setminus\D)$ and \( M_n\in \mathcal{A}_n(\D) \), let \( h_{f-M_n} \) be the best holomorphic (Nehari) approximant of \( f-M_n \) in \( H^\infty(\D)\). That is, \( \|f-M_n-h_{f-M_n}\|_\T=\inf_{h\in H^\infty(\D)}\|f-M_n-h\|_T \), and  \( h_{f-M_n} \in H^\infty(\D)\). Let us  set
\begin{equation}
\label{opt_Nehari}
N(M_n)(z) := (M_n+h_{f-M_n})(z),
\end{equation}
and call  \( N(M_n) \) the \emph{Nehari modification} of \( M_n \).  The discussion after \eqref{nehari_takagi} shows that \( N(M_n)\) lies in \(\mathcal A_n(\D) \). Indeed, we may write \( M_n=g_n + r_n \) with \( r_n\in\mathcal R_n(\D) \) and \( g_n\in\mathcal A(\D) \). Then, one can readily check that \( h_{f-M_n} = g_n + h_{f-r_n} \). As \( f-r_n \) is analytic across \( \T \) and in particular Dini-continuous there, it follows that \( h_{f-r_n} \) belongs to \( \mathcal A(\D) \) and so does \( h_{f-M_n} \).

Since \( \|f- N(M_n)\|_\T \leq \|f-M_n\|_\T \)  and $N(M_n)$ lies in $\mathcal{A}_n(\D)$, the sequence \( \{N(M_n)\} \) is one of \( n \)-th root optimal meromorphic approximants to $f$, whenever  \( \{M_n\} \) is such a sequence. It is beneficial for us to consider Nehari modifications because they enjoy the additional property that the error they generate  has  constant modulus on $\T$,  i.e., it follows from \eqref{arsv} that
\begin{equation}
\label{circ_error}
|(f-N(M_n))(z)| = \|\Gamma_{f-M_n}\| \quad \text{ for a.e.} \quad z\in \T.
\end{equation}
\emph{We claim} that it is enough to prove Theorem~\ref{thm:main2} for Nehari modifications only, as we now show.

Assume that Theorem~\ref{thm:main2} holds for \( \{ N(M_n) \} \). As the poles of \(M_n \) and \( N(M_n) \) are the same, this automatically yields the statement about weak$^*$ convergence of the counting measures of the poles. Moreover, given \( \epsilon>0 \) and \(K\subset \D\setminus\K\) a compact set, let us put
\[
E(K,\epsilon,N(M_n)):= \left\{z\in K : \left| \frac1{2n} \log|(f-N(M_n))(z)| - g(\mu_{\D,\K},\D;z) + \frac1{\cp_\D(\K)} \right|>\epsilon \right\}.
\]
Define \( E(K,\epsilon,M_n)\) analogously. According to our assumption it holds that
\begin{equation}
\label{assumeNehari}
\lim_{n\to\infty} \cp_\D\big( E(K,\epsilon,N(M_n)) \big) =0,
\end{equation}
and we need to show that \eqref{assumeNehari} holds with \( N(M_n) \) replaced by \( M_n \). Given $\varepsilon\in(0,1)$, define
\[
F_{n,\varepsilon}:=\left\{z\in K:\,\,\frac1{2n}\log | (f-N(M_n))(z) | > g\big(\mu_{\D,\K},\D;\cdot\big) - \frac1{\cp_\D(\K)}-\varepsilon    m_K\right\},
\]
where $m_K:=\min_K g\big(\mu_{\D,\K},\D;\cdot\big)>0$. Since $N(M_n)-M_n$ is analytic in \( \D \), we get from the maximum modulus principle and the triangle inequality that \(|N(M_n)(z)-M_n(z)|\leq 2\|f-M_n\|_\T \) for $z\in\D$. Since \( \exp\{(1-\varepsilon) m_K\}>1 \), relation \eqref{opt_merom} implies that
\[
2\|f-M_n\|_\T< \frac{1}{2}\exp\left\{(1-\varepsilon)2n\,m_K-\frac{2n}{\cp_\D(\K)}\right\}
\]
for all \( n \) large enough. Hence, by the previous estimates, it holds for all such \( n \) and \(z\in  F_{n,\varepsilon} \)  that
\[
\left| \frac{N(M_n)(z)-M_n(z)}{f(z)-N(M_n)(z)} \right| < \frac 12 e^{2n(m_K-g(\mu_{\D,\K},\D;z))} \leq \frac12.
\]
In particular, for any \( 0<\varepsilon<\varepsilon^\prime \), there exists \( n_0 \) depending on \( K \) and \( \varepsilon^\prime-\varepsilon \) such that
\begin{equation}
\label{incupl}
\frac{1}{2n}\left|\log\left|\frac{f(z)-M_n(z)}{f(z)-N(M_n)(z)}\right|\right|<(\varepsilon^\prime-\varepsilon) m_K
\end{equation}
for all $z\in F_{n,\varepsilon}$ and \( n\geq n_0 \). Therefore, we get from the triangle inequality that
\[
E(K,\varepsilon^\prime m_K,M_n) \subseteq ( K\setminus F_{n,\varepsilon}) \cup E(K,\varepsilon m_K,N(M_n))  =E(K,\varepsilon m_K,N(M_n))
\]
for all \( n\geq n_0 \). The above inclusion clearly yields that \eqref{assumeNehari} holds with \( N(M_n) \) replaced by \( M_n \) for \( \epsilon=\varepsilon^\prime m_K \). As \( \varepsilon \), \( \varepsilon^\prime \), and \( K \) were arbitrary, \emph{the claim follows}.

\subsection{Notation}
\label{ssec:notation}

We fix \( f\in \mathcal F(\overline\C\setminus\D) \) and, with a slight abuse of notation, we keep denoting by \( f \) the corresponding function in \( \mathcal F(\RS) \) (that was denoted by \( \widehat f \) in \eqref{defFge}). Take \( \{M_n\} \) to be a sequence of \( n \)-th root optimal meromorphic approximants of \( f \) and let \( \{ N(M_n) \} \) be the sequence of corresponding Nehari modifications. Since \( \{ N(M_n) \} \) is also \( n \)-th root optimal, it holds that
\begin{equation}
\label{f6a}
\lim_{n\to\infty} \frac1n\log\|f-N(M_n)\|_\T = -\frac 2{\cp_\D(\K)},
\end{equation}
see the discussion after \eqref{opt_merom}. We shall need an exhaustion  of \( \RS\setminus E_f \) by open sets with ``nice'' boundaries. That is, we consider a sequence \( \{ \Omega_m \}_{m\geq1} \) of open sets such that 
\begin{equation}
\label{exhaustion}
\Omega_m\subset\RS\setminus E_f,\quad\overline\Omega_m\subset\Omega_{m+1}, \quad \partial\RS\subset\partial\Omega_m, \quad \RS\setminus E_f = \cup_m\Omega_m,
\end{equation}
and each \( \Omega_m \), when viewed as an open subset with compact closure of \( \RS_* \), is regular for the Dirichlet problem, see Section~\ref{ssec_reg}. We will require in addition that \( \lambda( \partial\Omega_m \setminus\partial\RS)=0\) for some Radon measure $\lambda$ on $\RS$ that will be specified at the beginning of Section~\ref{sec_sripped_error} (a Radon measure is a positive Borel measure which is finite on compact sets). To design regular $\Omega_m$  that meet \eqref{exhaustion} is possible because $E_f$, being compact in $\RS$, is a countable intersection of compact sets $K_l\subset \RS$ with smooth boundary. Indeed, there is a smooth function $h\geq0$ on $\RS$ such that $E_f$ is the zero set of $h$  and $h\geq c>0$ outside a  compact neighborhood of $E_f$, hence we can pick $K_l$ to be the sublevel set $\{z:\,h(z)\leq t_l\}$, where $\{t_l\}$ a sequence of regular values of $h$ tending to 0 (almost every positive number is a regular value by Sard's theorem). In fact, using smooth partitions of unity and local coordinates, existence of such $h$ quickly  reduces to the corresponding issue in Euclidean space, where it follows easily from a combination of  \cite[Chapter~VI, Theorem~2]{Stein} and  \cite[Theorem~I]{Whitney} (this result is named after H. Whitney). Furthermore, since the sets $C_{a,b}:=\{z:\,a\leq h(z)\leq b\}$ are compact for  $0<a<b<\varepsilon$ and  \( \varepsilon>0 \) small enough,  $\lambda(C_{a,b})<\infty$. So, $\lambda(\{z:\, h(z)=t\})\neq0$ for at most countably many positive $t<\varepsilon$, and we can assume that $t_l$ chosen above are not such values. The domains $\Omega_m$ thus constructed satisfy all our requirements.

Let \( \Omega \) be a subdomain of \( \D \) or \( \RS \). Given a Borel measure $\sigma$ on \( \Omega \), we denote the Green potential of $\sigma$ relative to $\Omega$ by $g(\sigma,\Omega;\cdot)$, see  Section~\ref{ssec_green}. Hereafter, every measure is Borel unless otherwise stated. If $\sigma$ is a measure on a Borel set containing $\Omega$, we write for simplicity $g(\sigma,\Omega;\cdot)$ to mean  $g(\sigma_{\mathcal{b}\Omega},\Omega;\cdot)$. Conversely, for a measure $\sigma$ on a Borel set $B_0\subset\Omega$, we still denote by $\sigma$ the measure on $\Omega$ mapping a Borel set $B$ to $\sigma(B\cap B_0)$, and  write  $g(\sigma,\Omega;\cdot)$ for its potential. When \( \sigma=\sum_j \delta_{z_j}\) is a (possibly infinite) sum  of Dirac delta measures, we put
\begin{equation}
\label{gen-bl}
b(\sigma,\Omega;z) = \exp\big\{ -g^*(\sigma,\Omega;z) \big\}
\end{equation}
to stand for the corresponding generalized Blaschke product, where $g^{\ast}(\sigma,\Omega;\cdot)$ is a complexified Green potential, i.e., it is locally holomorphic in $\Omega\setminus\supp(\sigma)$ and $\operatorname{Re}g^{\ast}(\sigma,\Omega;\cdot)=g(\sigma,\Omega;\cdot)$.  If \( g(\sigma,\Omega;\cdot) \equiv+\infty\), which can happen when the points \( z_j\) accumulate in \( \Omega \) or to the boundary \( \partial\Omega \) sufficiently slowly, then \( b(\sigma,\Omega;\cdot) \) is identically zero. Otherwise,  \( b(\sigma,\Omega;\cdot) \) is well defined up to a unimodular constant (because the periods of a conjugate function of $g(\sigma,\Omega;\cdot)$ are integral multiples of $2\pi$ by Gauss' theorem), holomorphic in \( \Omega \), unimodular quasi everywhere on \( \partial \Omega \) (everywhere if the latter is regular and the points $z_j$ are finite in number), and it vanishes only at the points $z_j$ (with multiplicities represented by repetition).

\subsection{Stripped Error of Approximation}
\label{sec_sripped_error}

We shall study the asymptotics of  the error functions \( |f-N(M_n)\circ p| \). In this section, we strip off their poles and zeros to take logarithms and obtain harmonic functions whose limiting behavior we then investigate. To this end, we set
\begin{equation}
\label{f2a}
\widetilde\mu_{n}:=\sum\delta_{v_{n,j}}, \;\; \mu_n:=\widetilde\mu_n/n, \qandq \widetilde\nu_n:=\sum\delta_{u_{n,j}}, \;\; \nu_n := \widetilde \nu_n/n,
\end{equation}
where $\{v_{n,j}\}\subset \D$ are the poles of $N(M_n)$ and $\{u_{n,j}\}\subset\RS\setminus E_f$  the zeros of $f-N(M_n)\circ p$, with multiplicities counted by repetition. Before we proceed, let us specify the measure \( \lambda \) appearing in the definition of the exhaustion \( \{\Omega_m\} \) in the previous subsection. To this end, recall that on a locally compact space $X$, a sequence of Radon measures \(\{\sigma_n\}\) converges \emph{vaguely} to a Radon measure $\sigma$ if \( \int gd\sigma_n \to \int gd\sigma \) for every \( g\) in \( C_c(X) \), the space  of continuous functions with compact support on $X$ endowed with the $\sup$-norm. Moreover, if the measures \( \sigma_n \) are locally bounded on \( X \), then they do contain a vaguely convergent subsequence,  see for example \cite[Theorem 1.41]{EvansGariepy} for an argument on $\R^n$ which is applicable to any $\sigma$-compact locally convex space\footnote{Vague convergence is called weak convergence in \cite{EvansGariepy}.}. Hence, since the measures \( \mu_n \) have mass at most 1, we get if \(\nu_n\) is locally bounded along some sequence of integers that there exists a subsequence \( \mathcal N\subseteq \N \), a measure \( \nu^* \) on \( \RS\setminus E_f \) and a measure \( \mu \) on \( \D \) such that \( \nu_n \) and \( \mu_n \) converge vaguely to \( \nu^* \)  and  \( \mu \), respectively, along \( \mathcal N \); if the measures \( \nu_n \) have no locally bounded subsequence, i.e., if there exists a compact set \( K\subset\RS \setminus E_f \) such that \( \nu_n(K)\to\infty \) as \( \N\ni n\to\infty \), then we put $\nu^*=0$ and we only require the vague convergence  $\mu_n\to\mu$ along $\mathcal{N}$. In any case we take \( \lambda \) to be \( \nu^*+\widehat\mu \), where \( \widehat\mu \) is the lift of \( \mu \) to \( \RS \) defined via \eqref{lifted-measureA}.

Using \eqref{gen-bl}, we define Blaschke products vanishing at the poles of \( N(M_n) \) and the zeros of \( f-N(M_n)\circ p \). Namely, we put
\[
b_n^{pole}(z) := b(\widetilde\mu_n,\D;z) \quad \mbox{and} \quad  b_{n,m}^{zero}(z) := b(\widetilde\nu_n,\Omega_m;z).
 \] 
These functions are not identically zero, as the number of poles of \( N(M_n) \) is at most \( n \)   while the number of zeros of \( f-N(M_n)\circ p \) in each \( \Omega_m \) is finite. To see the latter point, recall from \eqref{circ_error} that \( |f-N(M_n)| \) is constant on \( \T \),  hence the error \( f-N(M_n) \) can be meromorphically continued across  \( \T \) by reflection. It implies that \( N(M_n) \) can be meromorphically continued across \( \T \) and this continuation is necessarily analytic in some neighborhood of \( \T \). Thus, $f-N(M_n)\circ p$ is analytic in a neighborhood of $\partial\RS$ by the analyticity of $f$ there, so the zeros of \( f-N(M_n)\circ p \) can only accumulate on  $E_f$ and not on $\partial\RS$. Thus, there are at most finitely many of them in each  $\Omega_m$. 

Using these Blaschke products, we define
\begin{equation}
\label{hnm}
h_{n,m}(z) := \frac1n\log\left\vert \left(f-N(M_n)\circ p\right)(z) (b_{n}^{pole}\circ p)(z)/b_{n,m}^{zero}(z)\right\vert, \quad z\in\Omega_m,
\end{equation}
which is harmonic in \( \Omega_m \). Recall that superharmonic functions on a hyperbolic Riemann surface are either identically \( +\infty \) or finite quasi everywhere, and any two of them that coincide almost everywhere (with respect to Lebesgue measure in local coordinates) are in fact equal (the weak identity principle), see Section~\ref{ssec_sub}.

\begin{lemma}
\label{lem:hr-e}
There exist a subsequence \( \mathcal N^\prime \subseteq \mathcal N \) and a non-negative superharmonic function \( u^\prime(z) \) on \( \RS \) such that
\begin{equation}
\label{lef-1}
-h_{n_m,m}(z) \to u^\prime(z) \quad \text{as} \quad m\to\infty,
\end{equation}
locally uniformly in \(z\in \RS\setminus E_f \), for any subsequence \( \{n_m\}_{m=1}^\infty\subseteq\mathcal N^\prime \). If \( u^\prime \) is finite quasi everywhere, then one has a decomposition
\begin{equation}
\label{uprime}
u^\prime = g(\nu^{\prime},\RS;\cdot) + h^\prime,
\end{equation}
where \( \nu^\prime \) is a finite positive Borel measure supported on \( E_f \) and \( h^\prime \) is a non-negative harmonic function on \( \RS \).
\end{lemma}
\begin{proof}
The regularity of \( \partial \Omega_m \) implies that \( |b_{n,m}^{zero}(z)|\equiv 1 \) for \( z\in\partial \Omega_m \), and likewise \(|b_n^{pole}(z)|=1\) for \( z \) on $\T$. In particular we get that
\[
\big|N(M_n)(z) b_n^{pole}(z)\big|\leq \|N(M_n)\|_\T,  \quad z\in \D,
\]
by the maximum modulus principle for $H^\infty$-functions. Set \( \Gamma_n:=\Gamma_{f-M_n} \) be the Hankel operator with symbol \( f-M_n \), see \eqref{Hankel}. It follows from the definition of the singular values, \eqref{arsv}, and \eqref{opt_Nehari} that \( \|\Gamma_n\|=s_0(\Gamma_n) = \|f-N(M_n)\| _\T \leq \|f-M_n\|_\T\). Since the norms \( \|f-M_n\|_\T \) tend to zero by assumption, we get that
\begin{equation}
\label{nmn-bound}
\|N(M_n)\|_\T \leq \|f\|_\T + \|\Gamma_n\| \leq C_f
\end{equation}
for some constant \( C_f \) that depends only on \( f \) and the sequence $\{M_n\}$. Thus, we get from the maximum principle for harmonic functions that
\begin{equation}
\label{hnm-bound}
h_{n,m}(z) \leq \frac1n\log\big( \|f\|_{\partial \Omega_m}  +  C_f \big),\qquad z\in \Omega_m.
\end{equation}

Set \( \mathcal N_0:=\mathcal N \). Proceeding  inductively on \( m\geq1 \), we deduce from \eqref{hnm-bound} that the sequence \( \{h_{n,m}\}_{n\in \mathcal N_{m-1}} \) is uniformly bounded above in \( \Omega_m \) and therefore, by \hyperref[hst]{Harnack's theorem}, see Section~\ref{ssec_sub},  there exists a subsequence \( \mathcal N_m\subseteq \mathcal N_{m-1} \) such that
\begin{equation}
\label{hnmhm}
h_{n,m}(z) \to h_m(z) \quad \text{as} \quad \mathcal N_m \ni n\to\infty
\end{equation}
locally uniformly in \( \Omega_m \), where \( h_m(z) \) is either identically \( - \infty \) or a non-positive harmonic function.  Define \( \mathcal N_* \) to be the diagonal of the table \( \{ \mathcal N_m\}_{m=1}^\infty \); that is, the \( m \)-th element of \( \mathcal N_* \) is the \( m \)-th element of \( \mathcal N_m \). Then, \eqref{hnmhm} holds along \( n\in \mathcal N_* \) for each $m$.

 Additionally, it follows from the maximum modulus principle for holomorphic functions that
\[
|b_{n,m_1}^{zero}(z)|>|b_{n,m_2}^{zero}(z)|, \quad z\in\Omega_{m_1}, \quad m_1<m_2.
\]
Therefore, we get from \eqref{hnm} and \eqref{hnm-bound} that \( h_{m_1}(z)\leq h_{m_2}(z)\leq0 \) for \( z\in\Omega_{m_1} \) when  \( m_1<m_2 \). Thus, if a finite limit \( h_{m_*}(z) \) exists for some index \( m_* \), then it exists for all \(m>m_* \). Hence,  either the functions \( h_{n,m} \) converge to \( -\infty \) as \( \mathcal N_*\ni n\to\infty \) locally uniformly in each \( \Omega_m \), in which case $h_m\equiv-\infty$ for all $m$, or else the functions \( h_m \) are finite and harmonic  for all \( m \) large enough.

If \( h_m\equiv-\infty\) for all $m$, select for each $m$ some \( n_m\in\mathcal N_* \) such that \( h_{n,m}(z)<-m \) for \( z\in\overline\Omega_{m-1} \) and all \( n\geq n_m \); we may require in addition that $n_m> n_{m-1}$ for $m\geq1$. Then \eqref{lef-1} holds with \( \mathcal N^\prime := \{n_m\}_{m=1}^\infty \) and \( u^\prime\equiv+\infty \).

If on the contrary the functions \( h_m \) are finite, they form an increasing sequence on $\Omega_{\ell}$ for $m\geq\ell$ and fixed $\ell$. As they are non-positive, they converge locally uniformly in \( \RS\setminus E_f \) to a non-positive harmonic function, say \( -u^\prime \), again by \hyperref[hst]{Harnack's theorem}. Since \( E_f \) is a closed polar set and \( u^\prime \) is non-negative, it follows from the \hyperref[rt]{Removability theorem}, see Section~\ref{ssec_fine}, that \( u^\prime \) extends to a superharmonic function on \( \RS \) that we keep denoting by \( u^\prime \).  Because \( u^\prime \) is superharmonic on $\RS$ and harmonic on $\RS\setminus E_f$, its Laplacian is a negative Radon measure $-\nu^\prime$ supported on $E_f$, which is necessarily finite since $E_f$ is compact. Thus, by the \hyperref[rrt]{Riesz representation theorem}, see Section~\ref{ssec_pot}, equation \eqref{uprime} takes place with \( h^\prime \) the largest harmonic minorant of \( u^\prime \).

Given \( m \), choose \( \tilde{n}_m \in \mathcal N_* \) such that \( | h_{n,m}(z) - h_m(z) | \leq 1/m \) for \( z\in\overline\Omega_{m-1} \) and all \( n\geq \tilde{n}_m \). Define \( \mathcal N^\prime := \{\tilde{n}_m \}_{m=1}^\infty \), where we again additionally require that \( \tilde{n}_m>\tilde{n}_{m-1} \). Given a compact set \( K\subset\RS\setminus E_f \) and \( \epsilon>0 \), we can pick \( m \) large enough that
\[
K\subseteq\overline\Omega_{m-1}, \quad 1/m\leq \epsilon/2, \quad \text{and} \quad |u^\prime(z) + h_m(z)| \leq \epsilon/2, \quad z\in K.
\]
Then, it follows from the last two inequalities that \( |h_{n,m}(z) + u^\prime(z)| \leq \epsilon \) for \( z\in K \) and any \( \mathcal N^\prime  \ni n\geq \tilde{n}_m \). Since \( \epsilon \) and \( K \) were arbitrary, this finishes the proof of \eqref{lef-1}.
\end{proof}

\begin{lemma}
\label{lem:hr}
If $u'\not\equiv+\infty$ in Lemma \ref{lem:hr-e}, then \( h^\prime \) in
\eqref{uprime} continuously extends to \( \partial\RS \) and
\begin{equation}
\label{hr-gamma}
h^\prime(z) = 
\begin{cases}
\displaystyle\frac{2}{\cp_\D(\K)}, & z\in\B, \smallskip \\
0, & z\in\partial\RS\setminus\B.
\end{cases}
\end{equation}
\end{lemma}
\begin{proof}
Let \( \Omega_* := p^{-1}(\{z:r<|z|<1\}) \), with \( r >0\) close enough to \( 1 \) that \( \overline\Omega_* \setminus \partial\RS\subset \Omega_m \) for each \( m \). It follows from the proof of Lemma~\ref{lem:hr-e} that \( h_{n,m}(z) \to h_m(z) \) as \( \mathcal N^\prime \ni n\to\infty \), locally uniformly in \( \overline\Omega_*\setminus\partial\RS \). Given a connected component \( \Omega \) of \( \Omega_* \), let \( \delta_z^{\RS\setminus\Omega} \) be its harmonic measure, see Section~\ref{ssec_bal}. Then
\begin{equation}
\label{log_sing}
\log|p(z)-(1+\epsilon)p(\xi)| = \int  \log|p(\zeta)-(1+\epsilon)p(\xi)|d\delta_z^{\RS\setminus\Omega}(\zeta), \quad z\in\Omega,
\end{equation}
for any \( \xi\in\partial \Omega\cap\partial\RS \) and \( \epsilon>0 \), see \eqref{balcont1}. It then follows from the monotone convergence theorem that we can take \( \epsilon=0 \) in \eqref{log_sing}. Recall that \( f-N(M_n)\circ p \) is analytic across \( \partial \RS\cap\partial \Omega \) and hence is non-vanishing there except perhaps for finitely many zeros counting multiplicities. Since $h_{n,m}$ is harmonic in $\Omega$, is continuous on $\partial\Omega\setminus\partial \RS$, and is equal to \( \frac1n\log|f-N(M_n)\circ p| \) on \( \partial \RS\cap\partial \Omega \) (i.e., it is continuous on $\partial\Omega\setminus\partial \RS$ except perhaps for finitely many logarithmic singularities), we get from  \eqref{log_sing} and \eqref{balcont1} that
\[
h_{n,m}(z) = \int_{\partial \Omega\setminus\partial \RS}h_{n,m}d\delta_z^{\RS\setminus\Omega} + \int_{\partial \RS\cap\partial \Omega}\frac1n\log|f-N(M_n)\circ p|d\delta_z^{\RS\setminus\Omega}, \quad  z\in\Omega.
\]
By \eqref{f6a} the pointwise limit of \( \frac1n\log|f-N(M_n)\circ p| \)  on \( \partial \RS \) is minus the right-hand side of \eqref{hr-gamma} except perhaps for a  finite subset of $\partial\mathcal{R}\setminus\mathcal{T}$, where $|f-N(M_n)\circ p|$ may go to zero, contained in
\[
\big\{\zeta|~\exists\,\eta\neq\zeta:~p(\zeta)=p(\eta), \; f(\zeta)=f(\eta)\big\}\setminus\B.
\]
As \( \delta_z^{\RS\setminus\Omega} \) does not charge polar, thus finite sets, the convergence in fact holds almost everywhere with respect to \( \delta_z^{\RS\setminus\Omega} \) for each fixed \( z \). So, if we can justify  the second equality in the following relation:
\begin{multline}
h_m(z) = \lim_{\mathcal N^\prime\ni n\rightarrow\infty}\left(\int_{\partial \Omega\setminus\partial \RS}h_{n,m}d\delta_z^{\RS\setminus\Omega} +  \int_{\partial \RS \cap \partial\Omega} \frac1n \log|f-N(M_n)\circ p|d\delta_z^{\RS\setminus\Omega} \right) \\ = \int_{\partial \Omega\setminus\partial \RS} h_md \delta_z^{\RS\setminus\Omega} + \int_{\partial \RS \cap \partial\Omega} \lim_{\mathcal N \ni n\rightarrow\infty} \frac1n \log|f-N(M_n)\circ p|d\delta_z^{\RS\setminus\Omega},
\label{exchange}
\end{multline}
we shall get that \( h_m(z) \) solves  a Dirichlet problem on \( \Omega \) with boundary data equal to \( h_{m\mathcal b\partial \Omega\setminus\partial\RS} \) on  \( \partial \Omega\setminus\partial\RS \), and to the negative of the right-hand side of \eqref{hr-gamma} on \( \partial \RS\cap\partial\Omega \). As such, $h_m$ must extend continuously to \( \partial\Omega \) where it is equal to the boundary data, since \( \partial\Omega \) is non-thin at any of its points. Subsequently, as $h_m$ converges to $-u^\prime$ locally uniformly on  $\RS\setminus E_f$ (see the proof of Lemma~\ref{lem:hr-e}), passing to the limit in the leftmost and rightmost sides of \eqref{exchange}  when $m\to\infty$  yields that $u^\prime$ extends continuously to $\partial\Omega\cap\partial\RS$ with values given there by the right-hand side of \eqref{hr-gamma}. This will give us the desired conclusion, because \( \nu^\prime \) is compactly supported in \( \RS \) and therefore \( g(\nu^\prime,\RS;\cdot) \) continuously extends by zero to \( \partial \RS \)  since $\RS$ is regular.

Altogether, it only remains to justify the swapping of the limit and integration signs in \eqref{exchange}. On \( \partial\Omega\setminus \partial \RS \), one can invoke the dominated convergence theorem. Thus, we only need to consider the integral over \( \partial \RS\cap\partial\Omega \). According to Vitali's convergence theorem, it is enough to show that for every \( \epsilon>0 \) there exists \( \delta>0 \) and $n_\varepsilon\in\N$ for which
\begin{equation}
\label{Vitali}
\int_E\left|\frac1n\log|f-N(M_n)\circ p|\right| d\delta_z^{\RS\setminus\Omega} < \epsilon \quad \mathrm{as\ soon\ as} \quad |\delta_z^{\RS\setminus\Omega}(E)|<\delta\quad\mathrm{and}\quad n\geq n_\varepsilon.
\end{equation}
For this, we first deduce from \eqref{nmn-bound} that 
\begin{equation}
\label{Vitali1}
\int_{\partial \RS}\frac1n\log|f-N(M_n)\circ p|d\delta_z^{\RS\setminus\Omega} \leq \frac1n\log\big(\|f\|_{\partial \RS}+C_f\big)
\end{equation}
as \( \delta_z^{\RS\setminus\Omega} \) is a probability measure. Now, if \( \partial \RS\cap\partial \Omega = \B \), then it follows from \eqref{f6a} that
\begin{equation}
\label{Vitali2}
\int_E\frac1n\log|f-N(M_n)\circ p|d\delta_z^{\RS\setminus\Omega} > -c\delta_z^{\RS\setminus\Omega}(E),
\end{equation}
for any \( E \subseteq\B \) and some positive constant \( c \). If \( \partial \RS\cap\partial \Omega = \B^\prime\neq\B \), set \( d(\eta):=f(\eta)-f(\zeta) \) where \( \eta\in\B^\prime \), \( \zeta\in\B \), and \( p(\eta) = p(\zeta) \). Then, we get from \eqref{nmn-bound} that
\begin{multline}
\label{Vitali-aux}
\log |(f-N(M_n)\circ p)(\eta)| \geq  \log \big||d(\eta)| - \|\Gamma_n\|\big| = \log \left| \frac{|d(\eta)|^2-\|\Gamma_n\|^2}{|d(\eta)|+\|\Gamma_n\|} \right| \\  \geq  \log \big||d(\eta)|^2-\|\Gamma_n\|^2\big| - \log(\|f\|_{\B^\prime}+C_f).
\end{multline}
Since \( E_f \), the singular set of \( f \), and \( {\bf rp}(\RS) \), the ramification set of \( \RS \), are closed and lie on top of \( \D \), \( d(\eta) \) extends to a holomorphic function, non-identically zero in a neighborhood of \( \B^\prime \). Then, 
\[
\mathcal D(\eta) :=d(\eta)\overline{d(p^{-1}(1/\overline{ p(\eta)})} \quad \text{satisfies} \quad \mathcal D(\eta) = |d(\eta)|^2, \quad \eta\in  \B^\prime,
\]
and is  holomorphic about \( \B^\prime \). Pick an open set $W\supset \B^\prime$ such that $\overline{W}\cap {\bf rp}(\RS) = \varnothing$ and $\mathcal D$ is holomorphic in $\overline{W}$ with no zero on $\partial W$; then, so is $\mathcal D-\|\Gamma_n\|^2$ for $n$ large as \( \|\Gamma_n\|^2 \to 0 \). Let \( \ell \) and \(\ell_n \) be minimal degree polynomials, normalized by
imposing $\|\ell\|_\T=\|\ell_n\|_\T=1$,   such that
\[
\frac{\mathcal D(\eta)}{\ell(p(\eta))} \quad \text{and} \quad \frac{\mathcal D(\eta)-\|\Gamma_n\|^2}{\ell_n(p(\eta))}
\]
are holomorphic and non-vanishing in $W$. Since \( \|\Gamma_n\|^2 \to 0 \), the zeros of $\mathcal D-\|\Gamma_n\|^2$ in $W$ tend to those of $\mathcal D$ by Rouch\'e's theorem, and so  our normalization implies that \( \ell_n\to\ell \) uniformly in $\overline{W}$ as \( n\to\infty \). Hence, $(\mathcal D-\|\Gamma_n\|^2)/(\ell_n\circ p)$ converges to $\mathcal D/(\ell_n\circ p)$ uniformly in $\overline{W}$, in particular it is uniformly bounded away from zero there.  Consequently,  it follows from \eqref{Vitali-aux} that
\begin{equation}
\label{bii2}
\log |(f-N(M_n)\circ p)(\eta)| \geq C + \log|\ell_n(p(\eta))|,\quad \eta\in  \B^\prime,
\end{equation}
for some finite constant  \( C \). Note that $\log|\ell_n|\leq0$ in $\D$ according to our normalization. Let us write  $\ell_n(x) = a_n\prod_i(x-x_{i,n})$ and define the reciprocal polynomial $\tilde\ell_n$ of $\ell_n$ by 
\[
\tilde\ell_n(x) := a_n\prod_i\left\{ \begin{array}{ll} x-x_{i,n} & \mathrm{if\ }|x_{i,n}|\geq 1, \smallskip \\ 1-x\bar x_{i,n} & \mathrm{if\ }|x_{i,n}|<1. \end{array} \right. 
\]
Clearly, \( |\ell_n(\xi)| = |\tilde\ell_n(\xi)| \) for \( |\xi|=1 \), and the maximum principle for harmonic functions implies that
\begin{equation}
\label{deftilde}
\int_{\B^\prime}\log|\tilde{\ell}_n(p(\eta))|d\delta_z^{\RS\setminus\Omega} \geq \log \big|\tilde \ell_n(p(z))\big|, \quad z\in \Omega,
\end{equation}
since both sides of \eqref{deftilde} are harmonic in \( \Omega \) and have the trace \( \log|\ell_n(p(\eta))| \) on \( \B^\prime \) while the left-hand side has zero trace on \( \partial\Omega\setminus\B^\prime \) and the right-hand side satisfies   $\log|\ell_n\circ p|\leq 0$ there. Thus, we get from \eqref{bii2} and \eqref{deftilde} that
\begin{eqnarray}
\int_{\B^\prime}\frac1n\log|f-N(M_n)\circ p|d\delta_z^{\RS\setminus\Omega} & \geq & \frac Cn + \frac1n\int_{\B^\prime}\log|\tilde{\ell_n}(p(\eta))|d\delta_z^{\RS\setminus\Omega} \nonumber \\
\label{Vitali3}
& \geq & \frac Cn + \frac{1}{n} \log \big|\tilde{\ell}_n(p(z))\big|.
\end{eqnarray}
As the last term goes to zero uniformly on $\Omega$ and since \( \tilde \ell_n \to \tilde\ell \), where \( \tilde\ell \) is the reciprocal polynomial of \( \ell \) defined similarly to \( \tilde\ell_n \), the estimate \eqref{Vitali} now follows from \eqref{Vitali1}, \eqref{Vitali2}, and \eqref{Vitali3}.
\end{proof}

\subsection{Asymptotic Distributions of Poles and Zeros}
\label{sec_poles_zeros}

Recall the measures \( \mu_n \) introduced in \eqref{f2a}. Since these measures have mass at most \( 1 \), it follows from Helly's selection theorem \cite[Theorem~0.1.3]{SaffTotik} that there exists a Borel measure \( \mu^\prime \), supported in \( \overline \D \) with mass at most \( 1 \), such that
\begin{equation}
\label{mun-muprime}
\mu_n \overset{w*}{\to}\mu^\prime \qasq \mathcal N^\prime \ni n\to\infty,
\end{equation}
perhaps at a cost of further restricting \( \mathcal N^\prime \), where $\overset{w*}{\to}$ stands for weak$^*$ convergence of finite (signed) measures (a sequence of Borel measures \( \{ \sigma_n \} \) on a locally compact space \( X \) converges weak$^*$ to a measure \( \sigma \) if  \( \int gd\sigma_n \to \int gd\sigma \) for every continuous function \( g \) in \( C_0(X) \), the completion of \( C_c(X) \) in the supremum norm). Observe that \( \mu^\prime_{\mathcal b \D} = \mu \), where \( \mu \) was defined as the vague limit \( \mu_n \) in \( \D \) along \( \mathcal N \supseteq \mathcal N^\prime\). In particular, \( \mu_n \overset{w*}{\to}\mu \) in \( \D \).

\begin{lemma}
\label{lem:pdlet-disk}
For any subsequence \( \{n_m\}_{m=1}^\infty\subseteq \mathcal N^\prime \),  it holds that
\begin{equation}
\label{lef-2}
\liminf_{m\to\infty}\frac1{n_m}\log \big| b_{n_m}^{pole}(z) \big| ^{-1} \left\{\begin{array}{c} \geq \\ = \end{array}\right\} g(\mu,\D;z),
\end{equation}
where the inequality holds for every  \(z\in \D \) and equality holds for quasi every $z\in \D$. 
\end{lemma}
\begin{proof}
Observe that \( \frac1n\log\big| b_{n}^{pole}(z)\big|^{-1} = g(\mu_n,\D;z) \), see \eqref{gen-bl}. Since $\mu_n\overset{w*}{\to}\mu$ on $\D$, the conclusion follows from the \hyperref[PDLEG]{Principle of Descent} and the \hyperref[PDLEG]{Lower Envelope Theorem}, see Section~\ref{ssec_reg}.
  \end{proof}

We cannot immediately get an analog of the previous lemma for the measures \( \nu_n =\widetilde\nu_n/n \), because we do not know  if these Radon  measures on $\RS\setminus E_f$ have uniformly bounded masses. Instead, we shall study the asymptotic behavior of their Green potentials in the style  of Lemmas \ref{lem:hr-e} and \ref{lem:hr}. 

\begin{lemma}
\label{lem:vague}
If the Radon measures \(\sigma_n\) converge vaguely to $\sigma$, $K\subset E$ is compact, and $\sigma(\partial K)=0$, then the restrictions $\sigma_{n\mathcal{b}K}$ have uniformly bounded masses, $\sigma_{n\mathcal{b}K}\overset{w*}{\to}\sigma_{\mathcal{b}K}$ on $K$, and \( \sigma_n(\partial K)\to 0 \).
\end{lemma} 
\begin{proof}
For each $\varepsilon>0$ there is an open set $V\supset \partial K$ such that $\sigma(V)<\varepsilon$ (by outer regularity of $\sigma$), and an open set $W$ with compact closure satisfying $V\supset\overline{W}\supset W\supset\partial K$ together with a continuous function $\varphi\geq0$, supported in $V$, which is $1$ on $\overline{W}$ (by Urysohn's lemma). Thus, $K_\varepsilon:=K\setminus W$ is a compact subset of $\mathrm{int}\, K$ (the interior of $K$) such that, for $n$ large enough that $|\int\varphi d\sigma_n-\int\varphi d\sigma|<\varepsilon$,
\[
\sigma_n(\partial K) \leq \sigma_n(K\setminus K_\varepsilon)\leq \sigma_n(W)\leq\int\varphi d\sigma_n<\varepsilon+\int\varphi d\sigma\leq2\varepsilon.
\]
Therefore $\int g d\sigma_n\to\int gd\sigma$ as $n\to\infty$ for any bounded continuous function $g$ on $\mathrm{int}\, K$ \cite[Proposition~6.18]{Demengel}, and since $\sigma_n(\partial K)\to 0$ while $\sigma(\partial K)=0$ it implies the weak$^*$ convergence of $\sigma_{n\mathcal{b}K}$  to $\sigma_{\mathcal{b}K}$. The uniform boundedness of the masses $\sigma_n(K)=\int_K1d\sigma_n$ now follows. 
\end{proof}

From now on we employ standard notation \( \D_r:=\{z:|z|<r \} \) and \( \T_r := \partial \D_r \). 

\begin{lemma}
\label{lem:pdlet-surface}
There exist a subsequence \( \mathcal N^{\prime\prime}\subseteq \mathcal N^\prime \) and a non-negative superharmonic function \( u^{\prime\prime}(z) \) on \( \RS \) such that
\begin{equation}
\label{lef-3}
\liminf_{m\to\infty}\frac1n\log\left\vert b_{n_m,m}^{zero}(z)\right\vert ^{-1} \left\{\begin{array}{c} \geq \\ = \end{array}\right\} u^{\prime\prime}(z)
\end{equation}
for any \( \{n_m\}_{m=1}^\infty\subseteq\mathcal N^{\prime\prime} \), where the inequality in \eqref{lef-3} holds everywhere on \( \RS\setminus E_f \) while the equality needs only to hold quasi everywhere. When \( u^{\prime\prime} \not\equiv+\infty\),  it holds that
\begin{equation}
\label{uprime2}
u^{\prime\prime} = g(\nu^{\prime\prime},\RS;\cdot) + h^{\prime\prime}
\end{equation}
for some Radon measure \( \nu^{\prime\prime} \) carried  by \( \RS \) and some  non-negative function \( h^{\prime\prime} \) harmonic on \( \RS \). 
\end{lemma}
\begin{proof}
Observe that \( \frac1n\log\big| b_{n,m}^{zero}(z)\big|^{-1} = g(\nu_n,\Omega_m;z) \) according to \eqref{gen-bl}. Below, we distinguish two cases: (i) when the measures \( \nu_n \) possess a subsequence which is locally bounded in \( \RS \setminus E_f \), i.e., having uniformly bounded masses on each compact subset of \( \RS \setminus E_f \), and (ii) when  there exists a compact set \( K\subset\RS \setminus E_f \) such that \( \nu_n(K)\to\infty \) as \(\mathcal N^\prime\ni n\to\infty \).

In case (ii) relation \eqref{lef-3} holds with \( u^{\prime\prime}\equiv\infty \) and \( \mathcal N^{\prime\prime} = \mathcal N^\prime \) because \( \min_{w\in K} g_{\Omega_\ell}(z,w)>0 \) for \( z\in\Omega_\ell \) and every \( \ell\) such that \( K\subset\Omega_\ell \), and therefore
\[
g(\nu_{n_m^\prime},\Omega_m;z) \geq g(\nu_{n^\prime_m},\Omega_\ell;z) \geq \nu_{n_m^\prime}(K)\min_{w\in K} g_{\Omega_\ell}(z,w) \to \infty
\]
as \( m\to\infty \) and \( \{n_m^\prime\}_{m=1}^\infty\subseteq\mathcal N^\prime \), where the first inequality holds for $m\geq\ell$.

In case (i), the measures \( \nu_n \) converge vaguely to \( \nu^* \) in \( \RS \setminus E_f \) along  \( \mathcal N^\prime \).  Let \( \{r_l\}_{l=1}^\infty \) be a positive real sequence  increasing to \( 1 \) with $r_1$ large enough that  \( p(\partial\Omega_m\setminus\partial\RS)\subset \D_{r_1} \) and $\nu^*(p^{-1}(\T_{r_l}))=0$ for each $l$. This is possible, because for $0<a<b<1$ the set $\cup_{a\leq r\leq b}p^{-1}(\T_r)$ is compact, so there are at most countably many $r\in[a,1)$ with $\nu^*(p^{-1}(\T_r))\neq0$.

We now argue by double induction over $m$ and  $l$:  the reasoning below should be applied inductively in \( m\geq1 \), so as to define a sequence of integers $\mathcal{N}_m$ for each $m$. Let \( \Omega_{m,l} :=\Omega_m \cap p^{-1}(\D_{r_l}) \) and proceed inductively in \( l\geq 1 \), starting with \( \mathcal N_{m,0} := \mathcal N_{m-1} \) where $\mathcal{N}_0=\mathcal N^\prime$. Since $\nu_n(\overline{\Omega}_m)<\infty$ for each $n$, $m$ by definition of $\nu_n$,  we can define
\begin{equation}
\label{pdlet-0}
h_{n,m,l}(z) := g(\nu_n,\Omega_m;z) -  g\big(\nu_{n\mathcal{b}\overline\Omega_{m,l}},\Omega_m;z\big),\quad z\in\Omega_{m,l},
 \end{equation}
which is a non-negative harmonic function in \( \Omega_{m,l} \). By \hyperref[hst]{Harnack's theorem}, either there is a subsequence $\mathcal{N}_{m,l}\subseteq \mathcal{N}_{m,l-1}$ of indices $n$ along which \( h_{n,m,l} \to h_{m,l} \), locally uniformly in \( \Omega_{m,l} \), for some   non-negative harmonic function \( h_{m,l} \), or else $h_{n,m,l}$ tends to infinity with $n\in\mathcal{N}_{m,l-1}$, locally uniformly in \( \Omega_{m,l} \).  In the latter  case, we set \( \mathcal N_{m,l}:=\mathcal N_{m,l-1} \) and $h_{m,l}\equiv+\infty$. Clearly, $h_{n,m,l}\geq h_{n,m,l+1}$ and  so, for fixed $m$, either $h_{m,l}\equiv+\infty$ for all $l$ or the $h_{m,l}$ are finite for $l$ large enough.  Let \( \mathcal N_m \) be the diagonal of the table \( \{\mathcal N_{m,l}\}_{l=1}^\infty \). Since \( \mathcal N_m \) is eventually a subsequence of every \( \mathcal N_{m,l} \), it holds that \( h_{n,m,l} \to h_{m,l} \) as \( \mathcal N_m\ni n\to\infty \) for every \( l\geq1 \), locally uniformly in $\Omega_{m,l}$. 

In another connection, since $\nu^*(\partial \Omega_{m,l})=0$ by construction and the \( \nu_{n\mathcal{b}\overline{\Omega}_{m,l}}\) have uniformly bounded mass over $n$ by the assumptions of the considered case, we deduce from Lemma~\ref{lem:vague} that
\[
\nu_{n\mathcal{b}\overline{\Omega}_{m,l}} \overset{w*}{\to} \nu^*_{\mathcal{b}\overline{\Omega}_{m,l}} \quad \text{on} \quad \overline{\Omega}_{m,l} \qasq \mathcal N_0\ni n\to\infty.
\]
Now, $\nu_{n\mathcal{b}\overline{\Omega}_{m,l}}$ defines a measure on $\Omega_m$ in a natural way, and the weak$^*$ convergence above implies weak$^*$ convergence on $\Omega_m$, because every continuous function with compact support in $\Omega_m$  restricts to a continuous function on $\Omega_m\cap \overline{p^{-1}(\D_{r_l})}$ which itself extends to $\overline{\Omega}_{m,l}=\overline{\Omega}_m\cap \overline{p^{-1}(\D_{r_l})}$ continuously by zero. As $\Omega_m$ is a regular open set with compact closure on the  surface \( \RS_* \), the \hyperref[PDLEG]{Principle of Descent} and the \hyperref[PDLEG]{Lower Envelope Theorem} yield for any subsequence \( \mathcal N^* \subseteq \mathcal N_0 \) that
\[
\liminf_{\mathcal N^*\ni n\to\infty} g\big(\nu_{n\mathcal{b}\overline{\Omega}_{m,l}},\Omega_m;z\big) \left\{\begin{array}{c} \geq \\ = \end{array}\right\} g\big(\nu^*_{\mathcal{b}\overline{\Omega}_{m,l}},\Omega_m;z\big),
\]
where the inequality holds everywhere in \( \Omega_m \) and the equality may only hold quasi everywhere. 

In view of \eqref{pdlet-0}, the above inequality and the very definition of \( \mathcal N_m \) imply that
\begin{equation}
\label{pdlet-1}
\liminf_{\mathcal N^*\ni n\to\infty} g\big(\nu_n,\Omega_m;z\big) \left\{\begin{array}{c} \geq \\ = \end{array}\right\} g\big(\nu^*_{\mathcal{b}\overline\Omega_{m,l}},\Omega_m;z\big) + h_{m,l}(z),\qquad z\in \Omega_{m,l},
\end{equation}
along any subsequence \( \mathcal N^* \subseteq \mathcal N_m \), where the inequality holds everywhere in \( \Omega_{m,l} \) while the equality needs only to hold quasi everywhere. As the left-hand side of \eqref{pdlet-1} does not depend on \( l \) and the right-hand side is superharmonic, we get from the weak identity principle that successive right-hand sides  are superharmonic continuations of each other when $l$ increases. Let \( u_m(z) \) be the superharmonic function in \( \Omega_m \) given on each \( \Omega_{m,l} \) by the right-hand side of \eqref{pdlet-1}. Since a smooth function with compact support in $\Omega_m$ is eventually supported in $\Omega_{m,l}$ for large $l$, we get from the definition that  either $u_m\equiv+\infty$ or \( \Delta u_m =-\nu^*\). In the latter case,
the \hyperref[rrt]{Riesz representation theorem} yields that
\begin{equation}
\label{pdlet-2}
u_m(z) = g(\nu^*,\Omega_m;z) + h_m(z),\qquad z\in\Omega_m,
\end{equation}
for some non-negative harmonic function \( h_m \), which is the largest harmonic minorant of $u_m$. 

Let \( \widetilde{\mathcal N} \)  be the diagonal of the table \( \{ \mathcal N_m \}_{m=1}^\infty \). As  \( \widetilde{\mathcal N} \) is eventually a subsequence of each \( \mathcal N_m \), we get from \eqref{pdlet-1} that for each $m$ and  any subsequence \( \mathcal N^* \subseteq \widetilde{\mathcal N} \) it holds that
\begin{equation}
\label{pdlet-3}
\liminf_{\mathcal N^*\ni n\to\infty} g(\nu_n,\Omega_m;z) \left\{\begin{array}{c} \geq \\ = \end{array}\right\} u_m(z),
\end{equation}
where the inequality takes place everywhere in \( \Omega_m \) and equality at least quasi everywhere. Because the left-hand side of \eqref{pdlet-3} increases with $m$, we have that \(  u_m(z) \leq u_{m+1}(z) \) for quasi every \(z\in \Omega_m \). Thus, either  $u_m\equiv+\infty$ for all $m$ large enough or else $u_m$ is finite quasi everywhere on $\Omega_m$ for all $m$. In the latter case, since  $\Delta u_m=\Delta u_{m+1\mathcal{b}\Omega_m}$ ($=-\nu^*_{\mathcal{b}\Omega_m}$), we get  that $u_{m+1}-u_m$ is harmonic on $\Omega_m$. Hence,  \( 0\leq u_m \leq u_{m+1} \) everywhere on $\Omega_m$, and so $u^{\prime\prime}:=\lim_m u_m$ is positive and superharmonic on $\RS\setminus E_f$. If $u^{\prime\prime}\equiv+\infty$ we are done, for we get \eqref{lef-3} from \eqref{pdlet-3} with \(\mathcal{N}^{\prime\prime} = \widetilde{\mathcal N} \). Otherwise  $u^{\prime\prime}$ is locally integrable and therefore \eqref{pdlet-2}, together with the \hyperref[rrt]{Riesz representation theorem}, imply that 
\begin{equation}
\label{pdlet-5}
u^{\prime\prime}(z) = g(\nu^*,\RS\setminus E_f;z) + \tilde h(z), \quad z\in\RS\setminus E_f,
\end{equation}
where \( \tilde h(z) \) is a non-negative function that is the largest harmonic minorant of $u^{\prime\prime}$ on $\RS\setminus E_f$.

As \( E_f \) is  polar and compact in $\RS$, we deduce from the \hyperref[rt]{Removability theorem} and the \hyperref[rrt]{Riesz representation theorem} that
\begin{equation}
\label{pdlet-6}
\tilde h(z) = h^{\prime\prime}(z) + g(\tilde\nu,\RS;z), \quad z\in \RS\setminus E_f,
\end{equation}
where \( h^{\prime\prime} \) is a non-negative harmonic function on \( \RS \) and \( \tilde\nu \)  a finite positive measure supported on \( E_f \). Moreover, since for  \( z\in \RS\setminus E_f \) the function \( g_\RS(z,\cdot) -g_{\RS\setminus E_f}(z,\cdot) \) is non-negative harmonic on $\RS\setminus E_f$ and bounded above near $E_f$, the \hyperref[rt]{Removability theorem} for harmonic functions yields that \( g_\RS(z,\cdot) -g_{\RS\setminus E_f}(z,\cdot) \equiv 0\) as it extends to a non-negative harmonic minorant of $g_\RS(z,\cdot)$ on $\RS$. Hence,
\begin{equation}
\label{pdlet-7}
g(\nu^*,\RS\setminus E_f;z) = g(\nu^*,\RS;z), \quad z\in \RS\setminus E_f,
\end{equation}
and equations \eqref{pdlet-5}--\eqref{pdlet-7} imply that \( u^{\prime\prime}(z) \) extends superharmonically to the entire surface \( \RS \) by
\begin{equation}
\label{pdlet-8}
u^{\prime\prime}(z) = g(\nu^{\prime\prime},\RS;z) + h^{\prime\prime}(z), \quad \nu^{\prime\prime} = \nu^* + \tilde\nu.
\end{equation}
Now, for any subsequence \( \{\tilde n_m\}_{m=1}^\infty\subseteq \widetilde{\mathcal N} \), it holds in view of \eqref{pdlet-3} that for each $m_0\in\mathbb{N}$ and \( z\in\Omega_{m_0} \)
\begin{equation}
  \label{pdlet-9}
\liminf_{m\to\infty} g\big(\nu_{\tilde n_m},\Omega_m;z\big) \geq \liminf_{m\to\infty} g\big(\nu_{\tilde n_m},\Omega_{m_0};z\big) \geq u_{m_0}(z),
\end{equation}
and we obtain the inequality in \eqref{lef-3} by letting $m_0$ tend to infinity.

Thus, it only remains to prove the equality quasi everywhere in \eqref{lef-3} when \( u^{\prime \prime}(z) \) is quasi everywhere finite. As before, the argument should be applied inductively on \( m \) with \( \widetilde{\mathcal N}_0 := \widetilde{\mathcal N} \). The functions \( g_{n,l,m} := g(\nu_n,\Omega_l;\cdot)-g(\nu_n,\Omega_m;\cdot) \) are  non-negative and harmonic in \( \Omega_ m \) for  \( l>m \). Therefore, by \hyperref[hst]{Harnack's theorem}, there are subsequences \( \widetilde{\mathcal N}_{m,l}\subseteq \widetilde{\mathcal N}_{m,l-1} \), \( \widetilde{\mathcal N}_{m,0} := \widetilde{\mathcal N}_{m-1} \) such that  $g_{n,l,m}$  converges locally uniformly to some function $H_{l,m}$ harmonic in \( \Omega_m \)  as \( \widetilde{\mathcal N}_{m,l}\ni n\to \infty \) (note that \(g_{n,l,m}\) cannot go to $+\infty$  otherwise so would $g(\nu_n,\Omega_l;\cdot)$, and in view of \eqref{pdlet-3} $u_l(z)$ would be infinite for quasi every $z$, contradicting that $u_l\leq u^{\prime\prime}<\infty$). Of necessity, \(H_{l,m}= u_l-u_m \) by \eqref{pdlet-3}, and a diagonal argument gives us a single subsequence \( \widetilde{\mathcal N}_m^* \subseteq \widetilde{\mathcal N} \) along which the convergence $g_{n,l,m}\to H_{l,m}$ takes place for any \( l>0 \). Now, for fixed \( m \) and each \( l>m \), select \( \tilde n_l \in \widetilde{\mathcal N}_m^* \) such that
\begin{equation}
  \label{triagnlm}
\big| g_{n,l,m}(z) - \big(u_l(z)-u_m(z)\big) \big| \leq 1/l, \quad z\in\overline\Omega_{m-1}, \quad \tilde n_l\leq n\in \widetilde{\mathcal N}_m^*.
\end{equation}
Since the functions \( u_l-u_m \) are harmonic in \( \Omega_m \) and increase with $l$, they converge locally uniformly to \( u^{\prime\prime}-u_m \) there by \hyperref[hst]{Harnack's theorem} and the definition of $u^{\prime\prime}$. Thus,  taking \eqref{triagnlm} into account, for any  \( \epsilon>0 \) there exists \( L>0 \) such that
\begin{equation}
  \label{limgnlm}
\big| g_{n,l,m}(z) - \big(u^{\prime\prime}(z)-u_m(z)\big) \big| \leq \epsilon, \quad z\in\overline\Omega_{m-1}, \quad l\geq L,\quad \tilde n_l\leq n\in \widetilde{\mathcal N}_m^*.
\end{equation}
Define \( \widetilde{\mathcal N}_m := \{\tilde n_l\}_{l=1}^\infty \). Then it follows from \eqref{limgnlm} and \eqref{pdlet-3} that
\begin{equation}
\label{pdlet-4}
\liminf_{l\to\infty} g\big(\nu_{n_l},\Omega_l;z\big) = \liminf_{l\to\infty} g\big(\nu_{n_l},\Omega_m;z\big) + \lim_{l\to\infty} g_{n_l,l,m}(z) = u^{\prime\prime}(z)
\end{equation}
for quasi every \( z\in \overline \Omega_{m-1} \), whenever \( \{n_l\}_{l=1}^\infty\subseteq \widetilde{\mathcal N}_m \). Finally, let \( \mathcal N^{\prime\prime} \) be the diagonal sequence of the table \( \big\{\widetilde{\mathcal N}_m\big\}_{m=1}^\infty \). Since \( \mathcal N^{\prime\prime} \) is eventually a subsequence of every \( \widetilde{\mathcal N}_m \) it follows from \eqref{pdlet-4} that 
\begin{equation}
\label{pdlet-4bis}
\liminf_{m\to\infty} g\big(\nu_{n_m},\Omega_m;z\big) = u^{\prime\prime}(z), \quad \{n_m\}_{m=1}^\infty\subseteq \widetilde{\mathcal N}^{\prime\prime},\quad \text{for q.e.} \quad z\in\RS\setminus E_f,
\end{equation}
which is the equality case  in \eqref{lef-3}.
\end{proof}

\subsection{Logarithmic Error Function}
\label{ssec:ler}

Hereafter, we redefine \( \mathcal N \) to be \( \mathcal N^{\prime\prime} \) constructed in Lemma~\ref{lem:pdlet-surface}. By Lemmas~\ref{lem:hr-e},~\ref{lem:pdlet-disk} and~\ref{lem:pdlet-surface},  the limits \eqref{f6a}, \eqref{lef-1}, \eqref{lef-2}, and \eqref{lef-3} hold along this new sequence.

Since $\mu$ is finite, there is a  $G_\delta$ polar set $\widetilde{N}_0\subset\D$ such that $g(\mu,\D;x)<+\infty$ for $x\in\D\setminus \widetilde{N}_0$, see Sections \ref{ssec_pot} and \ref{ssec_fine}. Let us put $N_0:=p^{-1}(\widetilde{N}_0)$, which is a $G_\delta$ polar subset of $\RS$, see Section~\ref{ssec_fine}. We now introduce the  function $ler:\RS\setminus N_0 \to [-\infty,+\infty)$ (``ler''  for ``logarithmic error''), by putting
\begin{equation}
\label{ler}
ler(z) := g(\mu,\D;p(z)) - u^\prime(z) - u^{\prime\prime}(z), \qquad z\in\RS\setminus N_0,
\end{equation}
where $u^\prime$ and $u^{\prime\prime}$ are as in Lemmas \ref{lem:hr-e} and \ref{lem:pdlet-surface}. Clearly, $ler(z)$ is a \( \delta\)-subharmonic function (the difference of two subharmonic functions), and it is well defined for $z\notin N_0$ since $g(\mu,\D;p(z))$ is finite there. As introduced, $ler$ depends on the choice of the subsequence \( \mathcal N \), but later we shall see that it is in fact unique.

\begin{lemma}
\label{lemB}
There exists a polar set $A_0\subset\RS\setminus N_0$ such that, whenever $z_1,z_2$ are distinct points in $\RS\setminus N_0$ with $p(z_1)=p(z_2)$ and $ler(z_i)<0$ for $i=1,2$, then $z_1,z_2\in A_0$.
\end{lemma}
\begin{proof}
By the equality quasi everywhere in \eqref{lef-2},  there exists a polar set $A_1\subset\mathbb D$ such that, for every $x\in\mathbb D\setminus A_1$, one can find a sequence $\mathcal N_x\subseteq\mathcal N$ along which
  \[
\lim_{\mathcal N_x\ni n\rightarrow\infty}\frac1n\log \big| b_n^{pole}(x) \big|^{-1}=g(\mu,\mathbb{D};x).
\]
Together with \eqref{hnm}, \eqref{lef-1}, the inequality in \eqref{lef-3} and \eqref{ler}, this gives us 
\begin{equation}
\label{ineq}
\limsup_{\mathcal N_x\ni n\rightarrow\infty}\frac1n\log\left\vert \left(  f-N(M_n)\circ p\right)  (z)\right\vert \, \leq \, ler(z)
\end{equation}
for every $z\notin N_0$ such that $z\in p^{-1}(x)$ with $x\not\in A_1\cup p(E_f)$. Assume now that $z_1,z_2\in \RS\setminus N_0$ satisfy $z_1\neq z_2$ and $p(z_1)=p(z_2)$, as well as $ler(z_i)<0$ for $i=1,2$. Let $x\in\D$ be such that $z_1,z_2\in p^{-1}(x)$. If $x\not\in A_1\cup p(E_f)$, it follows from \eqref{ineq} that
\[
f(z_1) = \lim_{\mathcal N_x\ni n\to\infty}N(M_n)(p(z_1)) = \lim_{\mathcal N_x\ni n\to\infty}N(M_n)(p(z_2)) = f(z_2)
\]
and necessarily $z_1,z_2\in A_2:=\big\{\zeta|~\exists\,\eta\neq\zeta:~p(\zeta)=p(\eta), \; f(\zeta)=f(\eta)\big\}$. Now, the conditions placed on the class \( \mathcal F(\RS) \) imply that the set $A_2$ is finite, and therefore the lemma holds with $A_0:=p^{-1}\big(A_1\cup p(E_f) \cup p(A_2)\big)$ which is polar, as inverse image under $p$ of a polar set.
\end{proof}

\begin{lemma}
\label{lem:finite}
The inequality $ler(z) > -\infty$ holds for quasi every $z\in\RS\setminus N_0$. In particular \( u^\prime\not\equiv+\infty\) and \(u^{\prime\prime}\not\equiv+\infty\) in Lemmas \ref{lem:hr-e} and \ref{lem:pdlet-surface}, so \( h^\prime\) and \(h^{\prime\prime} \) are finite non-negative harmonic functions on \( \RS \). 
\end{lemma}
\begin{proof}
  Since \( g(\mu;\D;p(\cdot)) \) is finite on $\RS\setminus N_0$ while \( u^\prime,u^{\prime\prime} \) are non-negative and either identically \(+\infty\) or finite quasi everywhere, $ler$ is in turn either identically $-\infty$ or finite quasi everywhere on $\RS\setminus N_0$. The former possibility contradicts Lemma~\ref{lemB}, therefore the latter prevails so that  $u^\prime\not\equiv+\infty$ and $u^{\prime\prime}\not\equiv+\infty$. Hence, Lemmas  \ref{lem:hr-e} and \ref{lem:pdlet-surface} imply that \( h^\prime\) and \(h^{\prime\prime} \) are finite  on \( \RS \).
\end{proof}

Due to the previous  lemma, we can rewrite \eqref{ler} as
\begin{equation}
\label{ler0}
ler(z) = g(\mu,\D;p(z)) - g(\nu,\RS;z) - h_\RS(z), \quad z\in\RS\setminus N_0,
\end{equation}
where we have set \( \nu:=\nu^\prime+\nu^{\prime\prime} \), which is a locally finite measure on \( \RS \) with quasi everywhere finite potential, and \( h_\RS := h^\prime + h^{\prime\prime} \) which is a positive harmonic function on  \( \RS\).

\begin{lemma}
\label{lem:hprime2}
It holds that \( \lim_{z\to\zeta}h_\RS(z) = 0 \) for every \( \zeta\in\partial\RS \setminus \B \).
\end{lemma}
\begin{proof}
Fix \( \zeta\in\partial\RS \setminus \B \), and let \( \xi\in\B \) be such that \( p(\zeta)=p(\xi) \). \emph{We claim}  that
\begin{equation}
\label{continuum0} 
\liminf_{z\to\zeta}h_\RS(z) = 0.
\end{equation}
Indeed, if \( \liminf_{z\to\zeta}h_\RS(z) = l>0 \), take \(0< 2\epsilon :=\min\{l,2/\cp_\D(\K)\} \). Let \( S \subset\D\) be the radial segment \( \{z:z=rp(\zeta),~r\in[1-\delta,1)\} \) and \( S_\zeta \) (resp.  \( S_\xi \)) be the connected component of \( p^{-1}(S) \) accumulating on $\partial\RS$ to \( \zeta \) (resp.  \( \xi \)). If \( \delta>0 \) is small enough, then
\begin{equation}
\label{continuum1}
h_\RS(z) \geq \epsilon, \quad z\in S_\zeta\cup S_\xi,
\end{equation}
by Lemma~\ref{lem:hr}. Furthermore, Lemma~\ref{lemB} yields that either \( ler(z_1)\geq0 \) or \( ler(z_2)\geq0 \) if \( p(z_1)=p(z_2) \in S\setminus p(A_0\cup N_0) \). In particular, we get from \eqref{ler0} and \eqref{continuum1} that
\begin{equation}
\label{continuum2}
g(\mu,\D;z) \geq \epsilon, \quad z \in S \setminus p(A_0\cup N_0),
\end{equation}
where we notice that $A_0\cup N_0$ as well as $p(A_0\cup N_0)$ are polar. This contradicts Lemma~\ref{lem:thin-sets}, applied with \( g(\sigma,D;\cdot)=g(\mu,\D;\cdot) \) and \( \xi \) being \( p(\zeta) \), since \( R_{\epsilon/2} \) from that lemma would necessarily be a subset of $p(A_0\cup N_0)$. \emph{This proves our claim} \eqref{continuum0}.

Next, assume for a contradiction that \( \limsup_{z\to\zeta}h_\RS(z) = l^\prime> 0 \), and pick  \( 0< 2\epsilon \leq \min\{l^\prime,2/\cp_\D(\K)\} \) such that the level line \( L_\epsilon:=\{ z:h_{\RS}(z)=\epsilon\} \) is a smooth  1-dimensional  submanifold of $\RS$ (this can be achieved according to Sard's theorem). Notice that \( \zeta \) must be a limit point of \( L_\epsilon \) because any neighborhood of \( \zeta \) in \( \RS_* \) contains a connected open set $U\ni\zeta$ with $U\cap\RS$ connected ($p$ is a local homeomorphism at $\zeta$ and we may take $p(U)$ to be a disk) in which $h_{\RS}$ assumes values arbitrary close to $0$ and $l^\prime$ by definition of $\liminf$ and $\limsup$; hence, as $h_{\RS}(U\cap\RS)$ is connected, it contains the value $\epsilon$.

Let \( D_0 \) be a disk centered at \( p(\zeta) \) of small enough radius so that \( D_\zeta \) and \( D_\xi \), the components of \( p^{-1}(D_0) \) in \( \RS_* \) that contain respectively \( \zeta \) and \( \xi \), are in one-to-one correspondence with \( D_0 \) under \( p \). Decreasing the radius of \( D_0 \) if necessary, we can assume that  \( h_{\RS}(z) \geq \epsilon/2 \) for \( z\in D_\xi \) by Lemma~\ref{lem:hr}. Let us redefine
\begin{equation}
  \label{nds}
S_\zeta := L_\epsilon\cap \overline D_\zeta, \quad S:=p(S_\zeta), \quad \text{and} \quad S_\xi:=p^{-1}(S)\cap \overline D_\xi.
\end{equation}
Observe that \( h_{\RS} \) cannot be constant in view of \eqref{continuum0} and \eqref{hr-gamma}. Therefore, no connected component of \( S \) can be a closed curve in \( D_0 \) by the maximum principle for harmonic functions. In addition, if \( S \) has  a connected component, say \( S_* \), accumulating at \( z_*\in \T \), then
\[
g(\mu,\D;z) \geq \epsilon/2, \quad z \in S_* \setminus p(A_0\cup N_0),
\]
exactly as in \eqref{continuum2}. Let \( R_{\epsilon/3} \) be as in Lemma~\ref{lem:thin-sets}, applied with \( g(\sigma,D;\cdot)=g(\mu,\D;\cdot) \) and \( \xi \) being \( z_* \). Then \( S_* \) must intersect every circle \( \{z\in\D:|z-z_*|=1-r\} \), \( r\in R_{\epsilon/3} \), by connectedness. Necessarily, the intersection must be a subset of \( p(A_0\cup N_0) \) and therefore polar. Since contractive maps do not increase the logarithmic capacity \cite[Theorem 5.3.1]{Ransford}, this means that \( R_{\epsilon/3} \) is polar, which contradicts Lemma~\ref{lem:thin-sets} (as we explain in Section~\ref{ssec_cap}, polar subsets of \( \D \) have Greenian and logarithmic outer capacity zero). Thus, \( S \) is a system of smooth curves, each connected component of which has  at least one limit point on \( \partial D_0\cap \D \). Consequently, if \( T_0\subset D_0 \) is a circle centered at \( p(\zeta) \), then any connected component of \( S \) intersecting the interior of \( T_0 \) must intersect \( T_0 \) as well since it accumulates at a point of $\partial D_0$. That is, \( S \) must intersect any such circle and we arrive at a contradiction exactly as above.
\end{proof}

The exceptional set where inequality is strict in Lemmas~\ref{lem:pdlet-disk} and~\ref{lem:pdlet-surface} a priori depends on the subsequence $\{n_m\}$ under consideration. The next lemma shows that there exists a polar set outside of which equality holds, both in \eqref{lef-2} and \eqref{lef-3}, for one and the same subsequence, see \eqref{lem8a-1}.

\begin{lemma}
\label{lem8a}
For quasi every \( z\in\RS \), there is a sequence \( \mathcal N_z=\{n_m^z\}_{m=1}^\infty \subseteq \mathcal N \) such that 
\[
\left\{
\begin{array}{lcl}
\displaystyle\lim_{m\to\infty}g(\nu_{n_m^z},\Omega_m;z) &=& u^{\prime\prime}(z), \smallskip\\
\displaystyle\lim_{m\to\infty}g(\mu_{n_m^z},\D;p(z)) &=& g(\mu,\D;p(z)).
\end{array}
\right.
\]
\end{lemma}
\begin{proof}
Our goal is to show that there exists a  subsequence \( \{ n_m^*\} \subseteq \mathcal N \) such that
\begin{equation}
\label{lem8a-1}
\liminf_{m\to\infty}\big( g(\nu_{n_m^*},\Omega_m;z) + g(\mu_{n_m^*},\D;p(z)) \big) =  g(\mu,\D;p(z)) + u^{\prime\prime}(z) 
\end{equation}
for quasi every \( z\in \RS \). Since the inequalities in \eqref{lef-2} and \eqref{lef-3} hold for every \( z\in \RS \), this will indeed imply the claim of the lemma.

To prove \eqref{lem8a-1}, we shall rewrite the sum of two potentials in the left-hand side as a single potential on \( \RS \). To this end, we lift \( \mu_n \) and \( \mu \) to \( \RS \) via the construction described in \eqref{lifted-measureA}. Specifically, with the notation introduced there, it follows from \eqref{remontp} that
\begin{equation}
\label{lem8a-2}
g({\widehat\mu}_n,\mathcal R;z) = g(\mu_n,\D;p(z)) \qandq g({\widehat\mu},\RS;z) = g(\mu,\D;p(z)), \quad z\in \RS.
\end{equation}
Now, we can write \( g(\nu_n,\Omega_m;z) + g(\mu_n,\D;p(z)) \) as a sum of three terms:
\begin{equation}
\label{lem8a-3}
g(\nu_n+{\widehat\mu}_n,\Omega_m;z) + \big( g({\widehat\mu}_{n\mathcal b\Omega_m},\RS,z) - g({\widehat\mu}_n,\Omega_m;z)\big) + g({\widehat\mu}_{n\mathcal b\RS\setminus \Omega_m},\RS;z),
\end{equation}
and we shall study their behavior separately.

To start, recall that \( \mu=\mu_{\mathcal b\D}^\prime \), where \( \mu^\prime \) is the weak$^*$ limit of \( \mu_n \) along \( \mathcal N \) in $\overline{\D}$. Thus, by the discussion after \eqref{remontp}, an analogous relation holds between \( {\widehat\mu} \) and \( {\widehat\mu}_n \). Namely,  since \( \mu_n \) has total mass \( 1 \),  the total mass of \( {\widehat\mu}_n \) is equal to \( M \), the number of sheets of \( \overline\RS \). In particular, the sequence ${\widehat\mu}_n$ converges weak$^*$ on $\overline\RS$ to ${\widehat\mu}^\prime $, and on $\mathcal{R}$ to ${\widehat\mu}^\prime_{\mathcal{b}\RS}={\widehat\mu}$. 

Since the sets \( \Omega_m \) exhaust \( \RS\setminus E_f \), it holds that \( {\widehat\mu}(\RS\setminus \Omega_m) - {\widehat\mu}(E_f) \to 0 \) as \( m \to \infty \). Moreover, as each \( \RS\setminus \Omega_m \) is compact with boundary of \({\widehat\mu}\)-measure zero, it follows from Lemma~\ref{lem:vague} that the measures ${\widehat\mu}_{n\mathcal b\RS\setminus \Omega_m}$ converge weak$^*$ to ${\widehat\mu}_{\mathcal b\RS\setminus \Omega_m}$ along $\mathcal{N}$. In particular, \( {\widehat\mu}_n(\RS\setminus \Omega_m)-{\widehat\mu}(\RS\setminus \Omega_m) \to 0  \) as \( n\to\infty \). Hence, for each \( m \) there exists \( n_m^\prime \in \mathcal N \) such that 
\[
|{\widehat\mu}_{n}(\mathcal R\setminus \Omega_m) - {\widehat\mu}(E_f)| \leq 2|{\widehat\mu}(\RS\setminus \Omega_m) - {\widehat\mu}(E_f)| \quad  \text{as soon as} \quad n\geq n_m^\prime.
\]
Let \( \sigma \) be a weak$^*$ limit point on $\RS$ of  the family \( \{{\widehat\mu}_{n_m^\prime\mathcal b\RS\setminus \Omega_m}\}_{m\in\N} \). Clearly \( \supp(\sigma)\subseteq E_f \) and $\sigma$ is also a weak$^*$ limit point of this family on the fixed compact set $\RS\setminus \Omega_1$. Thus, integrating against any function which is identically $1$ on $\RS\setminus \Omega_1$ and passing to the limit gives us $\sigma(E_f)={\widehat\mu}(E_f)$ by our very choice of $n^\prime_m$. In another connection, if $\varphi:\RS\to[0,1]$ is a continuous function with compact support which is $1$ on a compact set $K\subseteq E_f$, then
\[
\sigma(K)\leq \int\varphi d\sigma \leq\limsup_m\int\varphi d{\widehat\mu}_{n_m^\prime\mathcal b\RS\setminus \Omega_m}\leq\lim_m \int\varphi d{\widehat\mu}_{n_m^\prime}=\int\varphi d{\widehat\mu}.
\]
As the infimum over $\varphi$ of the rightmost term above is ${\widehat\mu}(K)$ by Urysohn's lemma and the outer regularity of Radon measures, we get that  \( \sigma \leq {\widehat\mu}_{\mathcal b E_f} \) by  the inner regularity of Radon measures on $\RS$, see \cite[Theorem~2.18]{Rudin}. Altogether, we deduce  that \( \sigma = {\widehat\mu}_{\mathcal b E_f} \), and consequently the measures \( {\widehat\mu}_{n_m^\prime\mathcal b\RS\setminus \Omega_m} \) converge  weak$^*$ to \( {\widehat\mu}_{\mathcal b E_f} \) along \( \mathcal N^\prime := \{n_m^\prime\} \). In particular,
\begin{equation}
  \label{lem8a-4}
\lim_{m\to\infty}g({\widehat\mu}_{n_m^\prime\mathcal b\RS\setminus \Omega_m},\RS;z) = g({\widehat\mu}_{\mathcal b E_f},\RS;z), \quad z\in \RS\setminus E_f,
\end{equation}
which settles the asymptotic behavior of the last term in \eqref{lem8a-3} along the sequence of indices $\{n_m^\prime\}_m$. Notice that by the definition of $n^\prime_m$, if $\mathcal{N}^\prime$ is replaced by an arbitrary subsequence $\{n^{\prime\prime}_m\}_m$ thereof, then  \eqref{lem8a-4} continues to hold along this new subsequence of indices. This will be used in the forthcoming steps.

Next, let \( f_m(z,w) := g_\RS(z,w) - g_{\Omega_m}(z,w) \) for  $z,w\in\Omega_m$. Clearly, \( f_m(z,\cdot)\) is harmonic in \( \Omega_m \) and continuous on $\overline{\Omega}_m$, by the regularity of $\Omega_m$. Its boundary values are equal to \( g_\RS(z,\cdot) \) on \( \partial \Omega_m \), in particular they are identically zero on \( \partial \RS \). Fix \( z\in\RS\setminus E_f \), and let \( k \) be an integer such that \( z\in\Omega_k \). Then, we get for all \( m > k \) that
\begin{equation}
\label{lem8a-5}
0 \leq f_m(z,w) \leq \max_{w\in\partial\Omega_m\setminus  \partial \RS} g_\mathcal R(z,w) \leq \max_{w\in\partial\Omega_{k+1}} g_\RS(z,w) =: C_k, \quad w\in\overline\Omega_m,
\end{equation}
where the constant \( C_k \) is finite, independent of \( m \), and we used the maximum principle for harmonic functions twice (once for \( f_m(z,\cdot) \) on \(\Omega_m \) and once for \( g_\RS(z,\cdot) \) on \( \RS\setminus\Omega_{k+1} \)). Observe further that the functions \( f_m(z,\cdot) \), \(m>k\), are not only positive harmonic in each \( \Omega_l \) for fixed $l$ satisfying $k<l\leq m$, but form a decreasing sequence there. Therefore they converge in \( \Omega_l \) to a non-negative harmonic function, say \( f^{\{l\}}(z,\cdot) \), by \hyperref[hst]{Harnack's theorem}. As this claim is true for all large \( l \), the \( f^{\{l\}}(z,\cdot) \) inductively define a harmonic function \( f(z,\cdot) \) on \( \RS\setminus E_f \) that satisfies \( 0\leq f(z,\cdot)\leq C_k \) when \( z\in\Omega_k \). Thus, it extends harmonically to the entire surface \( \RS \) by the \hyperref[rt]{Removability theorem} for harmonic functions, and as its trace on \( \partial\RS \) is zero we conclude that \( f(z,\cdot)\equiv 0 \).

Observe  now that \( f_m(z,w) = f_m(w,z) \) for $m$ large enough so that $z,w\in\Omega_m$. Thus, it is jointly harmonic in both variables \cite[p. 561]{Lelong}. Hence,  by \hyperref[hst]{Harnack's theorem} and the diagonal argument, any subsequence of $\{f_m(\cdot,\cdot)\}$ has a further subsequence converging locally uniformly in $\RS\setminus E_f\times \RS\setminus E_f$, and we know from what precedes that the limit function can only be zero. In particular, it holds that
\begin{equation}
  \label{etalm}
  \eta_{l,m} := \max_{z,w\in\overline\Omega_l} f_m(z,w) \to 0 \quad \text{as} \quad m\to\infty,
\end{equation}
where we used that each \( f_m(\cdot,\cdot) \) extends continuously by zero to \( \partial\RS\times\partial\RS \), and therefore the maximum principle for harmonic functions can be applied to show that the convergence is indeed uniform on \( \overline{\Omega}_l \). Given \( m \), define
\[
l_{m}:=\max\big\{l\,:\,1\leq l\leq m,\,\eta_{l,m}\leq1/l\big\}.
\]
It follows from \eqref{etalm} that $l_m\to\infty$ and $\eta_{l_m,m}\to0$ when $m\to\infty$, whence for each fixed \( z\in \Omega_k \) and all \( m>k \) it holds that
\begin{eqnarray}
0 \leq \int_{\Omega_m} f_m(z,w) d{\widehat\mu}_n(w) &=& \left(\int_{\Omega_m\setminus\Omega_{l_m}} + \int_{\Omega_{l_m}}\right) f_m(z,w) d{\widehat\mu}_n(w) \nonumber \\
\label{lem8a-6}
&\leq& C_k{\widehat\mu}_n(\overline\Omega_m\setminus\Omega_{l_m}) + M \eta_{l_m,m},
\end{eqnarray}
where we used \eqref{lem8a-5} while recalling that \( M \) is the number of sheets of \( \RS \), which is a  bound for the total mass of each \( {\widehat\mu}_n \). Now, we obtain from Lemma~\ref{lem:vague} and our choice of $\Omega_m$ that the measures ${\widehat\mu}_n$ converge weak$^*$ to ${\widehat\mu}$ along $\mathcal{N}$,  not only on \( \RS \), but also on every compact set of the form $\overline{\Omega}_m\setminus \Omega_{l_m}$. Hence, taking into account that \( {\widehat\mu}\big( (\RS\setminus E_f)\setminus\Omega_{l_m} \big) \to 0 \) as \( m\to\infty \), we can associate to each \( m \) an integer \( n_m^{\prime\prime}\in\mathcal N^\prime \) such that \( n_m^{\prime\prime}\geq  n_m^{\prime} \) and
\[
{\widehat\mu}_{n_m^{\prime\prime}} \big(\overline\Omega_m\setminus\Omega_{l_m}\big) \to 0 \quad \text{as} \quad m\to\infty.
\]
Clearly, the choice of \( \mathcal N^{\prime\prime} :=\{ n_m^{\prime\prime} \}  \) is independent of \( z \). Thus, we get from \eqref{lem8a-6} together with the choices of \( \{l_m\} \) and \( \mathcal N^{\prime\prime} \) that for every $z\in\RS\setminus E_f$ it holds
\begin{equation}
\label{lem8a-7}
\lim_{m\to\infty} g({\widehat\mu}_{n_m^{\prime\prime}\mathcal b\Omega_m},\RS;z) - g({\widehat\mu}_{n_m^{\prime\prime}},\Omega_m;z) = \lim_{m\to\infty} \int_{\Omega_m} f_m(z,w)d{\widehat\mu}_{n_m^{\prime\prime}}(w) = 0,
\end{equation}
which settles the asymptotic behavior of the middle term in \eqref{lem8a-3} along \( \mathcal N^{\prime\prime} \).

Lastly, to describe asymptotics of the first term in \eqref{lem8a-3}, we need to repeat some steps of the proof of Lemma~\ref{lem:pdlet-surface} with \( \nu_n \) replaced by \( \nu_n + {\widehat\mu}_n \). Recall the definition of the sets \( \Omega_{m,l} \) given just before \eqref{pdlet-0}. We may adjust it so that \( (\nu^*+{\widehat\mu})(\partial\Omega_{m,l}) = 0 \). Indeed, $\partial\Omega_{m,l}=(\partial\Omega_m\setminus\partial\RS)\cup p^{-1}(\T_{r_l})$. We already know that $(\nu^*+{\widehat\mu})(\partial\Omega_m\setminus\partial\RS)=0$, so we only need to ensure that $(\nu^*+{\widehat\mu})(p^{-1}(\T_{r_l}))=0$ for each $l$. This can be achieved as before since ${\widehat\mu}$ is finite and therefore $\nu^*+{\widehat\mu} $ is still a Radon measure. Since ${\widehat\mu}^\prime(\partial\Omega_{m,l})={\widehat\mu}(\partial\Omega_{m,l})=0$ by construction, it follows from Lemma \ref{lem:vague} that
\[ 
{\widehat\mu}_{n\mathcal{b}\overline\Omega_m\setminus\Omega_{m,l}} \overset{w^*}{\to }{\widehat\mu}^\prime_{\mathcal{b}\overline\Omega_m\setminus\Omega_{m,l}} \qandq {\widehat\mu}_n(p^{-1}(\T_{r_l}))\to0
\]
for all $m,l$. So, when $z\in\Omega_{m,l}$, we get that 
\begin{multline}
\label{lem8a-8}
\lim_{\mathcal N^{\prime\prime}\ni n\to\infty} g\big({\widehat\mu}_{n\mathcal{b}\Omega_m\setminus\overline\Omega_{m,l}},\Omega_m;z\big) = \lim_{\mathcal N^{\prime\prime}\ni n\to\infty} g\big({\widehat\mu}_{n\mathcal{b}\overline\Omega_m\setminus\Omega_{m,l}},\Omega_m;z\big) \\ = g\big({\widehat\mu}^\prime_{\mathcal{b}\overline\Omega_m\setminus\Omega_{m,l}},\Omega_m;z\big) = g\big({\widehat\mu}_{\mathcal{b}\Omega_m\setminus\overline\Omega_{m,l}},\Omega_m;z\big)
\end{multline}
since $g_{\Omega_m}(\cdot,z)$ is continuous on $\overline\Omega_m\setminus\Omega_{m,l}$ and vanishes on $\partial\RS$ by regularity of $\Omega_m$.  Moreover, the convergence is locally uniform in \( z\in\Omega_{m,l} \) by the continuity of Green functions with respect to both variables off the diagonal. From \eqref{lem8a-8} and the reasoning used  after \eqref{pdlet-0}, it follows that
\[
  \lim_{\mathcal N^{\prime\prime}\ni n\to\infty} g\big((\nu_n+{\widehat\mu}_n)_{\mathcal{b}\Omega_m\setminus\overline\Omega_{m,l}},\Omega_m;z\big) = h_{m,l}(z) + g\big({\widehat\mu}_{\mathcal{b}\Omega_m\setminus\overline\Omega_{m,l}},\Omega_m;z\big),
\]
locally uniformly in \( \Omega_{m,l} \) for each $l$ and some subsequence $\mathcal N^{\prime\prime}_m$ of $\mathcal N^{\prime\prime}$. Arguing as we did to obtain \eqref{pdlet-1}, but applying this time the \hyperref[PDLEG]{Lower Envelope theorem} to $(\nu_n+{\widehat\mu}_n)_{\mathcal{b}\overline{\Omega}_{m,l}}$, we get in view of the above limit that
\begin{align*}
\liminf_{\widetilde{\mathcal N}^{\prime\prime}\ni n\to\infty} g(\nu_n+{\widehat\mu}_n,\Omega_m;z) &=  g((\nu^*+{\widehat\mu})_{\mathcal b \overline \Omega_{m,l}},\Omega_m;z) + h_{m,l}(z) + g\big({\widehat\mu}_{\mathcal{b}\Omega_m\setminus\overline\Omega_{m,l}},\Omega_m;z\big) \\ &= g(\nu^*+{\widehat\mu},\Omega_m;z) + h_m(z)
\end{align*}
for quasi every \( z\in\Omega_m \), where \( \widetilde{\mathcal N}^{\prime\prime} \) is the diagonal of the table \( \{ {\mathcal N}_m^{\prime\prime} \}_{m=1}^\infty \) and to get the second equality we used \eqref{pdlet-2} along with the explanation preceding it on the inductive definition of  $u_m$ by the right-hand side of \eqref{pdlet-1}. The previous equation stands analog to \eqref{pdlet-3},  and continues to hold if $\widetilde{\mathcal N}^{\prime\prime}$ is replaced by any subsequence thereof. Now, the last part of the proof in Lemma~\ref{lem:pdlet-surface} was predicated on the limit
\[
u_m(z) = g(\nu^*,\Omega_m;z) + h_m(z) \to u^{\prime\prime}(z) \quad\mathrm{as}\quad m\to\infty,\quad z\in \RS\setminus E_f,
\]
where $u_m$ were initially  defined inductively by the right-hand side of \eqref{pdlet-1} and their limit  $u^{\prime\prime}$ assumed the form \eqref{pdlet-5}. In the present case, \(u_m\) is replaced by \( u_m + g({\widehat\mu},\Omega_m;\cdot)\) and the monotone convergence theorem together with the polar character of $E_f$ imply that
\[
\lim_mg({\widehat\mu},\Omega_m;z)=g({\widehat\mu},\RS\setminus E_f;z)=g({\widehat\mu}_{\mathcal b \RS\setminus E_f},\RS;z), \quad z\in \RS\setminus E_f.
\]
Hence, arguing as we did after \eqref{pdlet-9}, we obtain similarly to \eqref{pdlet-4bis} that
\begin{equation}
\label{lem8a-9}
\liminf_{m\to\infty} g(\nu_{n_m^*}+{\widehat\mu}_{n_m^*},\Omega_m;z) = u^{\prime\prime}(z) + g({\widehat\mu}_{\mathcal b \RS\setminus E_f},\RS;z),
\end{equation}
for quasi every \( z\in\RS\setminus E_f \) and some subsequence \( \{n_m^*\}\subseteq \widetilde{\mathcal N}^{\prime\prime} \). The desired limit \eqref{lem8a-1} now follows from \eqref{lem8a-2}, \eqref{lem8a-3}, \eqref{lem8a-4}, \eqref{lem8a-7}, and \eqref{lem8a-9}.
\end{proof}

Hereafter we shall deal with the \emph{fine topology}, which is the coarsest topology for which superharmonic and therefore \( \delta \)-subharmonic functions are continuous. In this connection, the reader may want to consult  the definitions and properties collected in Section~\ref{ssec_fine}. 

Recall the definition of $N_0\subset\RS$ before \eqref{ler}  as being the $+\infty$-set of $g(\mu,\D;p(\cdot))$, i.e., $N_0=p^{-1}(\widetilde{N}_0)$ where $\widetilde{N}_0\subset\D$ is the $+\infty$-set of $g(\mu,\D;\cdot)$. Clearly, $N_0$ is a finely closed and polar $G_\delta$-set. Let us define
\begin{equation}
\label{f10a}
G_\pm := \big\{  z\in\RS\setminus N_0:\pm ler(z)>0\big\}.
\end{equation}
It is easily seen from \eqref{ler} that $ler:\mathcal{R}\setminus N_0\to[-\infty,+\infty)$ is finely continuous and that $G_-$ and $G_+$ are finely open in $\mathcal{R}\setminus N_0$. Since \( N_0 \) is polar, if the complement of either \( G_+ \) or \( G_- \) in $\RS\setminus N_0$ is thin at a point \( z \), then the respective complement in $\RS$ is also thin at \( z \). Hence, \( G_+ \) and \( G_- \) are in fact finely open in \( \RS \). Hereafter, we put $D_\pm:=p(G_\pm)$.

\begin{lemma}
\label{lem11a} 
For $z\in G_+$ and any $\zeta\in \RS$ with  $p(\zeta)=p(z)$, it holds that $\zeta\in G_+$ and $ler(z)=ler(\zeta)$. In particular $G_+=p^{-1}(D_+)$. Moreover, $D_+$ is finely open.
\end{lemma}
\begin{proof}
By the definition of $ler$ given in \eqref{ler0}, we get from equation \eqref{hnm} together with Lemmas~\ref{lem:hr-e},~\ref{lem:finite}, and  \ref{lem8a} (recall $\mathcal{N}^{\prime\prime}$ was renamed as $\mathcal{N}$  at the top of Section~\ref{ssec:ler}) that there exists a polar set $B\subset \RS$ with the following property: for each $z\in G_+\setminus B$  there is a  subsequence $\mathcal N_z\subset \mathcal{N}$ such that
  \begin{equation}
    \label{convbsGp}
    \left\{
\begin{array}{lcl}
  \displaystyle
  \lim_{\mathcal N_z\ni n\rightarrow\infty}\frac1n\log\left\vert\left(  f-N(M_n)\circ p\right)(z)\right\vert & = & ler(z), \\
\displaystyle\lim_{\mathcal N_z\ni n\rightarrow\infty}g(\mu_{n},\D;p(z)) &=& g(\mu,\D;p(z)),
\end{array}
\right.
\end{equation}
where we note that $g(\mu,\D;p(z))<+\infty$ since $z\notin N_0$. Without loss of generality, we may assume that $B$ contains \(p^{-1}(p( E_f))\) and \( A_0 \), since both sets are polar, see Lemma~\ref{lemB} for the definition of $A_0$. As $ler(z)>0$ and \( f(z) \) is finite on $G_+\setminus B$, the above limit implies that
\[
\lim_{\mathcal N_z\ni n\rightarrow\infty}\frac1n\log\big|N(M_n)(p(z))\big| = ler(z).
\]
Now, if $\zeta\in\RS\setminus B$ is such that $p(\zeta)=p(z)$, obviously $\zeta\notin N_0$  and in addition $\zeta\notin E_f$, by definition of $B$. Thus, on account of the finiteness of $f(\zeta)$, we get that
\begin{equation}
  \label{echf}
  \lim_{\mathcal N_z\ni n\rightarrow\infty}\frac1n\log\big|\left(f-N(M_n)\circ p\right)(\zeta)\big| = \lim_{\mathcal N_z\ni n\rightarrow\infty}\frac1n\log\big|N(M_n)(p(z))\big| = ler(z).
\end{equation}
On the other hand, from the second equation in \eqref{convbsGp}, we deduce as in \eqref{ineq} that the first limit in \eqref{echf} is at most $ler(\zeta)$, whence $0<ler(z) \leq ler(\zeta)$. In particular, $\zeta\in G_+$, and reversing the roles of $z$ and $\zeta$ gives $ler(z) = ler(\zeta)$. This proves the first assertion of the lemma when $z\in G_+\setminus B$.  

To prove it on all of $G_+$, pick \( z,\zeta \in\RS\) such that \( p(z) = p(\zeta) \) and let \( D_0 \subset\D\) be a disk centered at \( p(z) \). Denote by \( D_z \), \( D_\zeta \) the connected components of \( p^{-1}(D_0) \) that contain \( z \), \( \zeta \) respectively, and make \( D_0 \)  small enough so that \( D_z\cap{\bf rp}(\RS) \subseteq \{ z \} \) and \( D_\zeta\cap{\bf rp}(\RS) \subseteq \{ \zeta \} \). Let as before \( m(\xi) \) be the ramification order of \( \xi \), so that \( D_z \) (resp. \( D_\zeta \)) is (isomorphic to) an \( m(z) \)-sheeted cyclic covering of \( D_0 \). For \( x\in D_0\setminus\widetilde N_0 \), define
\[
g(x) := m(\zeta)\sum_{\xi\in p^{-1}(x)\cap D_z} m(\xi) ler(\xi) - m(z)\sum_{\xi\in p^{-1}(x)\cap D_\zeta} m(\xi) ler(\xi).
\]
It follows from  \eqref{ler0} and \eqref{push-down} that
\begin{multline*}
\sum_{\xi\in p^{-1}(x)\cap D_z} m(\xi) ler(\xi) = m(z)g(\mu,\D;x)- g\big(p_*(\nu_{\mathcal b D_z}),\D;x\big) \\-\sum_{\xi\in p^{-1}(x)\cap D_z} m(\xi) \left( g\big(\nu_{\mathcal b\RS\setminus D_z},\RS;\xi\big) + h_\RS(\xi) \right).
\end{multline*}
Notice that the last summand above is a continuous, even harmonic function in \( D_0 \).  Hence, the sum itself is finely continuous in \( D_0\setminus\widetilde N_0 \). Likewise, the second sum in the definition of \( g \) is finely continuous in \( D_0\setminus\widetilde N_0 \) and so is \( g \). In particular, \( U:=\{ x\in D_0\setminus\widetilde N_0:g(x)\neq0 \} \) is finely open, and since \( p(G_+) \) is finely open by Lemma~\ref{lem:fineopen} the set \( U\cap p(G_+) \) is in turn finely open. From the first part of the proof, it follows that if \( x\in (D_0\cap p(G_+))\setminus p(B) \) then \( g(x) =0 \), whence \( U\cap p(G_+)\subseteq p(B) \). Thus,  \( U\cap p(G_+) \) must be empty  as \( p(B) \) is polar, that is, $g\equiv0$ on $D_0\cap p(G_+)$. Now, it can be readily checked that
\[
g(p(z)) = m(z)m(\zeta)\big( ler(z) - ler(\zeta) \big),
\]
and therefore \( ler(z) = ler(\zeta) \) if \( z \in G_+ \), thereby  proving the first assertion of the lemma. The second is then obvious, and the third follows from Lemma~\ref{lem:fineopen}.
\end{proof}

\begin{lemma}
\label{lem12a}
The set $G_-$ lies schlicht over $\D$. That is, \( p: G_-\to D_- \) is a bijection. Moreover, $G_-\cap {\bf rp}(\RS)=\varnothing$, and for each $z\in G_{-}$ we have that
\begin{equation}
\label{f12a}
ler(\zeta) =0 \text{ \ \ for all \ }\zeta\in p^{-1}(p(z))\setminus\{z\}.
\end{equation}
\end{lemma}
\begin{proof}
Pick \( z\in G_-\setminus{\bf rp}(\RS) \) and let $D_0\subset\RS\setminus{\bf rp}(\RS)$ be a  conformal disk centered at $z$, homeomorphic under $p$ to a Euclidean disk $p(D_0)$; note that \( p^{-1}(D_0)\setminus D_0 \) is open. Define \( U := G_-\cap D_0 \) and \( V := \big( p^{-1}(D_0)\setminus D_0 \big) \cap p^{-1}(p(U)) \). Clearly \( U \) and \( V \) are disjoint  finely open subsets of $\RS$, by Lemma~\ref{lem:fineopen}. In view of Lemma~\ref{lemB}, there is a polar set $A_0\subset\RS\setminus N_0$  such that $V\cap G_-\subseteq A_0$.  Thus, $V\cap G_-$ must be empty as otherwise it is finely open. This shows that $G_-\setminus {\bf rp}(\RS)$ lies schlicht over $\D$. However, the complement of a schlicht set is always non-thin at any ramification point by Lemma~\ref{lem:compthicks}. Hence,  $G_-$ cannot be a fine neighborhood of  a ramification point, and since it is finely open $G_-\cap {\bf rp}(\RS)=\varnothing$. Altogether, $G_-$ lies schlicht over $\D$ and  \eqref{f12a} now follows  from this and Lemma~\ref{lem11a}.
\end{proof}

\subsection{Modified Logarithmic Error Function}

In this subsection we modify the function \( ler(z) \) by clearing out parts of \( \D \) and \( \RS \) from the support of \( \mu \) and \( \nu \), respectively. We shall accomplish this via the technique of balayage, described in Section~\ref{ssec_bal}. Let us start with some preliminary geometric considerations. Recall that any connected (topological) $1$-manifold embedded in $\RS$ is a Jordan curve.

\begin{lemma}
\label{lem:boundary}
For each \( \epsilon>0 \) there exists a Jordan curve \( J\subset G_- \) such that \( p(J)\) is a Jordan curve included in \( \{z:1-\epsilon<|z|<1\} \) and \( p({ \bf rp}(\RS)) \) belongs to the interior domain of \( p(J) \). Moreover, there exists a finely connected component of \( G_- \), say \( G_J \), such that \( J\subset G_J \).
\end{lemma}
\begin{proof}
  We may assume that $\epsilon$ is small enough that $p({\bf rp}(\RS))\subset \D_{1-\epsilon}$. In particular, $p$ is injective  on every  conformal disk centered at a point of  $\partial\RS$ with radius smaller than or equal to $\epsilon$.  

Recall that \( \RS \) is a subset of a Riemann surface \( \RS_* \) lying over \( \overline\C \). Define \( G_-^* := G_-\cup\B\cup\mathcal D \), where $\mathcal{D}$ is the connected component of \( p^{-1}(\overline\C\setminus\overline\D) \) that borders \( \B \). Let us show that for each \( \eta\in\B \) there is a disk \( D_\eta \subset\{z:1-\epsilon<|z|<1+\epsilon\}\), centered at \( p(\eta) \) with  radius $r_\eta$, such that the circle $\partial D_\eta$ is included in  $p(G_-^*)$. In fact, we can pick $r_\eta$ so that there exist radii $r^\prime_\eta$ arbitrarily close but not equal to $r_\eta$ for which each disk $D^\prime_\eta$ centered at \( p(\eta) \) of radius $r^\prime_\eta$ also satisfies $\partial D^\prime_\eta\subset p(G_-^*)$. Indeed, like we did to establish Lemma~\ref{lem8a}, let \( \widehat\mu \) be the  lift \( \mu \) to \( \RS \) , see \eqref{lifted-measureA}.  By definition, \( ler(z) \geq 0 \) for \( z\in \RS\setminus (G_- \cup N_0)\). Therefore, if \( \eta\in\B \) is a limit point of \( \RS\setminus G_-  \), then we get from \eqref{hr-gamma},  \eqref{ler0}, the definition of \( h_\RS(z) \), and  the identity $g(\widehat\mu,\RS;\zeta)=+\infty$ when $\zeta\in N_0$ that
\[
\liminf_{\RS\setminus G_- \ni \zeta\to \eta} g(\widehat\mu,\RS;\zeta) \geq 2/\cp_\D(\K).
\]
The claim now follows from Lemma~\ref{lem:thin-sets}  by taking \( 1-r_\eta \) to be an accumulation point of \( R_{1/\cp_\D(\K)} \) in \( (1-\epsilon,1) \).

Let \( U_\eta \) be the connected component of \( p^{-1}(D_\eta) \) containing \( \eta \), which is an open subset of $\RS_*$ satisfying \( \partial U_\eta \subset G_-^* \). Since the collection \( \{U_\eta\}_\eta \) covers \( \B \), which is compact, it contains a finite subcover, say \( \{U_{\eta_i}\} \). Replacing  $D_{\eta_i}$ by some $D^\prime_{\eta_i}$ as above if needed, we can ensure by (finite) induction on $i$ that $\partial U_{\eta_i}$ and $\partial U_{\eta_j}$ may intersect only transversally for $i\neq j$, and that $\partial U_{\eta_i}\cap\partial U_{\eta_j}\cap\partial U_{\eta_k}=\varnothing$ if $i,j,k$ are all distinct. Then \( V:=\cup_i U_{\eta_i} \) is an open neighborhood of \( \B \) with boundary $\partial V$ included in $G_-^*$  that consists of  a finite union of disjoint Lipschitz-smooth  Jordan curves. In fact, there are exactly two such curves, one in \(\RS \) and  \( \mathcal D \), because each connected component of \( \partial V \) can be continuously deformed into \( \B \) via radial retraction within \( V \). We can choose the component within \( \RS \) to be \( J \) since Lipschitz curves are finely connected in $\RS$, see Section~\ref{ssec_fine}.
\end{proof}

Recall now the finely open subsets \( D_+ \) and \( D_- \) of \( \D \) introduced before Lemma~\ref{lem11a}. Denote by \( D_J \) the union of all finely connected components of \( D_+\cup D_- \) that lie entirely within the interior domain of \( p(J) \), where $J$ was introduced in Lemma \ref{lem:boundary}. Note that \( D_J \) is finely open because so are the fine components of finely open sets, see Section \ref{ssec_fine}. Define
\begin{equation}
  \label{VpD}
D^\prime := D_J \cup i(\D\setminus D_J) \qandq V^\prime :=p^{-1}(D^\prime),
\end{equation}
where \( i(\cdot) \) is the subset of finely isolated points. Thus, it follows from \eqref{Vprime} and  Lemma~\ref{lem:fineopen} that $D^\prime$ and $V^\prime$ are finely open while
\begin{equation}
\label{DprimeVprime}
b\big(\D\setminus D^\prime\big) = \D\setminus D^\prime \quad \text{and} \quad b\big(\RS\setminus V^\prime\big) = \RS \setminus V^\prime,
\end{equation}
where \( b(\cdot) \) stands for the base of a set (in particular, $\D\setminus D^\prime$ has no finely isolated points).  In other words, $D^\prime$ and $V^\prime$ are regular finely open sets, see Section \ref{ssec_reg}. Recall from Section~\ref{ssec_bal} the notation $\sigma^{E}$ for the balayage of the measure $\sigma$ onto the set $E$.

\begin{lemma}
\label{lem:ler1}
Let \( N_1 := p^{-1}(\widetilde N_1) \), where \( \widetilde N_1 \) is the \( +\infty\)-set of \( g\big(\mu^{(1)},\D;\cdot\big)\) and we set
\begin{equation}
  \label{def_bal1}
\mu^{(1)}:=\mu^{\D\setminus D^\prime} \qandq \nu^{(1)}:=\nu^{\RS \setminus V^\prime}.
\end{equation}
 Then \( N_1 \subseteq N_0 \), \( N_1\setminus V^\prime = N_0\setminus V^\prime \), and for every $z\in\RS\setminus N_1$ we can define
\[
ler^{(1)}(z) := g\big(\mu^{(1)},\D;p(z)\big) - g\big(\nu^{(1)},\RS;z\big) - h_\RS(z)
\]
with values in $[-\infty,+\infty)$. This function satisfies
\begin{equation}
\label{f11i}
ler^{(1)}(z) = 
\begin{cases}
ler(z), & z\in \RS\setminus (V^\prime\cup N_1), \smallskip \\ 0, &  z\in V^\prime\setminus N_1.
\end{cases} 
\end{equation}
\end{lemma}
\begin{proof}
Since the Green potential of a measure dominates  the Green potential of any balayage of that measure, as explained at the beginning of Section~\ref{ssec_bal}, it holds that
\[
g\big(\mu^{(1)},\D;z\big)\leq g\big(\mu,\D;z\big)<+\infty, \quad z\in \RS\setminus N_0,
\]
by the very definition of \( N_0 \). Thus, \( N_1\subseteq N_0 \) and the upper equality in \eqref{f11i} as well as equality \( N_1\setminus V^\prime = N_0\setminus V^\prime \) are consequences of \eqref{DprimeVprime} and \eqref{char_balayage}. Since  \( h_\RS(z) \) is a harmonic function on \( \RS \) and \( \overline{ V^\prime} \cap \partial \RS = \varnothing \) by construction (recall that $p(V^\prime)$ lies interior to $p(J)$),  equation \eqref{reproduction} yields that
\begin{equation}
\label{ler1-2}
\int h_\RS(x)d\delta_z^{\RS\setminus V^\prime}(x) = h_\RS(z), \quad z\in V^\prime.
\end{equation}
Since $V^\prime$ is a regular finely open set, Lemma \ref{supfin} implies that  \( \delta_z^{\RS\setminus V^\prime} \) is carried by \( \partial_\mathsf{f}V^\prime \) and  does not charge polar sets. As \( ler(z) = 0 \) for quasi every \( z\in\partial_\mathsf{f}V^\prime \) by the definition of $V^\prime$, we get from the definition of \( ler^{(1)} \), \eqref{ler1-2}, \eqref{GreenHarmonicMeasure}, Lemma~\ref{saturated}, and \eqref{ler0}  that
\begin{eqnarray*}
ler^{(1)}(z) & = & \int \big( g(\mu,\D;p(x)) - g(\nu,\RS;x) - h_\RS(x) \big) d\delta_z^{\RS\setminus V^\prime}(x)  \\
& = & \int ler(x) d\delta_z^{\RS\setminus V^\prime}(x) = 0
\end{eqnarray*}
for \( z\in V^\prime \setminus N_1\), which proves the lower equality in \eqref{f11i}. 
\end{proof}

Let \( G_J \) be as in Lemma~\ref{lem:boundary} for some 
small \( \epsilon>0 \). Denote by \( G^*_J \) the union of \( G_J \) and the annular region delimited by \( J \) and \( \B \); the latter is diffeomorphic under $p$ to
\( \{z:1-\epsilon<|z|<1\} \) if  we fix $\epsilon$ small enough.
Clearly,  \( G^*_J \) is a fine domain that lies schlicht over \( \D \). Let  \( \mathcal G \) be the collection of all fine domains \( G\subset\RS \) lying schlicht over \( \D \) and containing \( G_J^* \). The set \( \mathcal G \) is partially ordered by inclusion and every chain in it is bounded above by the union of its elements. Therefore,   by Zorn's lemma, \( \mathcal G \) possesses a maximal element, say \( G_\mathsf{max}\). Note that if $\RS\setminus G_\mathsf{max}$ is thin at two points $\zeta_1,\zeta_2 \in \RS$, then $p(\zeta_1)\neq p(\zeta_2)$ as otherwise $G_\mathsf{max}$ could not be schlicht over $\D$ by Lemma~\ref{lem:fineopen}. Hence, it follows from \eqref{Vprime} and the maximality of $ G_\mathsf{max}$ that \( \RS\setminus G_\mathsf{max}\) is its own base.

\begin{lemma}
\label{lem:maximality}
Any maximal domain \( G_\mathsf{max}\in\mathcal G \) is a Euclidean domain. Moreover, no connected component of \( \RS\setminus G_\mathsf{max} \) (resp. \(\D\setminus p( G_\mathsf{max})\)) consists of a single point.
\end{lemma}
\begin{proof}
Observe that \( G_\mathsf{max} \) cannot contain a ramification point of \( \RS \) as it is finely open and lies schlicht over \( \D \), see Lemma~\ref{lem:compthicks}. Thus, for every $\zeta\in G_\mathsf{max}$ there exists $r_0>0$ such that each component of
\[
p^{-1}\big(\big\{z\in\D:~|z-p(\zeta)|<r_0\big\}\big)
\]
is in one-to-one correspondence with $\{z\in\D:~|z-p(\zeta)|<r_0\}$ under \( p \). Let $V$ be the component containing $\zeta$. Since the intersection $V\cap G_\mathsf{max}$ is finely open, $V\setminus G_\mathsf{max}$ is thin at $\zeta$ and therefore there is  $r_1\in(0,r_0)$ such that
\[
V\cap p^{-1}\big(\big\{z\in\D:~|z-p(\zeta)|=r_1\big\}\big)\subset G_\mathsf{max},
\]
see Section~\ref{ssec_thin}. Now, as $G_\mathsf{max}$ lies schlicht over $\D$, it holds that
\begin{equation}
\label{pdinterr1}
G_\mathsf{max}\cap \left(p^{-1}\big(\big\{z\in\D:~|z-p(\zeta)|=r_1\big\}\big)\setminus V\right) = \varnothing.
\end{equation}
Further, since $G_\mathsf{max}$ is finely connected, we necessarily have that
\[
G_\mathsf{max}\cap \left(p^{-1}\big(\big\{z\in\D:~|z-p(\zeta)|<r_1\big\}\big)\setminus V\right) = \varnothing,
\]
otherwise $G_\mathsf{max}$ would intersect the fine boundary of  \(p^{-1}\big(\big\{z\in\D:~|z-p(\zeta)|<r_1\big\}\big)\setminus V\)  which is contained in  \(p^{-1}\big(\big\{z\in\D:~|z-p(\zeta)|=r_1\big\}\big)\setminus V\) (in fact, equal to it by regularity), thereby contradicting \eqref{pdinterr1}. The maximality of $G_\mathsf{max}$ now yields that
\[
p^{-1}\big(\big\{z\in\D:~|z-\zeta|<r_1\big\}\big) \cap V \subset G_\mathsf{max},
\]
hence  $\zeta$ belongs to the Euclidean interior of $G_\mathsf{max}$ and so $G_\mathsf{max}$ is Euclidean open. Thus, its connected components are open and therefore finely open. Hence, $G_\mathsf{max}$ is a Euclidean domain since it is finely connected by definition.

To prove the second assertion, assume to the contrary that a point $\zeta\in \RS\setminus G_\mathsf{max}$ is a connected component of the latter. \emph{We claim} that any open neighborhood  $W$ of $\zeta$ contains an open set $O\ni\zeta$ whose boundary $\partial O$ is a smooth Jordan curve $\mathcal{C}$ contained in $G_\mathsf{max}$. To see this, assume with no loss of generality that $\overline{W}\subset \RS$, and put $F:=\RS\setminus G_\mathsf{max}\cap \overline{W}$. The latter is closed in $\RS$, so there is a smooth function $h:\RS\to\R^+$ of which $F$ is the zero set, see Section~\ref{ssec:notation}. Given a sequence $\{c_n\}$ of regular values of $h$ tending to $0$, let $O_n$ be the connected component containing $\zeta$ of the open set $\{z:\,h(z)<c_n\}$. The sets $O_n$ form a (strictly) nested sequence of connected open sets containing $\zeta$, whose intersection is connected and contained in $F$ whence reduces to $\zeta$. If $O_{n_k}\cap \partial W\neq\varnothing$ for some increasing subsequence of indices $n_k$, take $z_{n_k}\in O_{n_k}\cap \partial W$ and extract from $\{z_{n_k}\}$ a subsequence converging to $z\in\partial W$, which is possible due to compactness of the latter. On the other hand, since $\overline O_{n_{k+1}}\subset O_{n_k}$,  $z\in\cap_k \overline O_{n_k}=\{\zeta\}\in \overline W\setminus \partial W$, which is a contradiction. Therefore, $O_n\subset W$ for $n$ large enough and $\partial O_n\cap F=\varnothing$ by construction, so we may set $O=O_n$ and $\mathcal{C}:=\partial O_n$ for any $n$ large enough, because $\partial O_n$ is a connected component of the level set $h^{-1}(c_n)$, and thus it is a smooth Jordan curve. \emph{This proves the claim}.

If $\zeta\notin{\bf rp}(\RS)$, let us pick $W$ so small that $p$ is injective on each connected component of $p^{-1}(p(W))$. We now argue as before, observing that ${\mathcal C}\subset G_\mathsf{max}$ and $(p^{-1}(p({\mathcal C}))\setminus {\mathcal C})\cap G_\mathsf{max}=\varnothing$ (by schlichtness), so that $ G_\mathsf{max}\cap (p^{-1}(p(O))\setminus O)=\varnothing$ because $G_\mathsf{max}$  is connected and $(p^{-1}(p({\mathcal C}))\setminus {\mathcal C})$ is the boundary of \(p^{-1}(p(O))\setminus O\). Thus, by maximality, $G_\mathsf{max}$  should contain $O$ which contradicts the fact that $\zeta\in O$. Finally, if $\zeta\in{\bf rp}(\RS)$ then Lemma~\ref{tdG} contradicts the existence of ${\mathcal C}$.

This proves that no connected component of \( \RS\setminus G_\mathsf{max} \) consists of a single point. To show the same is true of  \(\D\setminus p( G_\mathsf{max})\), observe that such a component would consist, by the same reasoning as before, of a $z\in  \partial p( G_\mathsf{max})$ lying interior to Jordan curves of arbitrary small diameter contained in $p( G_\mathsf{max})$. For $\gamma:[0,2\pi]\to  p( G_\mathsf{max})$ such a parametrized Jordan curve, let $\xi\in G_\mathsf{max}$ satisfy $p(\xi)=\gamma(0)$ ($=\gamma(2\pi)$).  Let further $\ell:[0,2\pi]\to \RS$ be a continuous lift of $\gamma$ starting at $\xi$,  i.e., $\ell(0)=\xi$ and $p(\ell(t))=\gamma(t)$ for $t\in[0,2\pi]$. Of necessity $\ell([0,2\pi])\subset G_\mathsf{max}$, because $p(\partial G_\mathsf{max})\cap p(G_\mathsf{max})=\varnothing$ since $G_\mathsf{max}$ is open and lies schlicht over $\D$, while $p$ is an open map; then schlichtness again  implies that $\ell$ is a parametrized Jordan curve in \(G_\mathsf{max}\), and it is the unique lift of $\gamma$ to \(G_\mathsf{max}\). 

Set $p^{-1}(z)=\{\zeta_1,\ldots,\zeta_\ell\}$, and let $D_z$ be an open disk centered at $z$  such that each connected component of $p^{-1}(D_z)$ is (isomorphic to) a  $m(\zeta_j)$-sheeted cyclic covering $D_{\zeta_j}$ of $D_z$; in addition, we require that $D_z$ is so small that $D_{\zeta_j}\subset U_{\zeta_j}$ for all $j$, where $U_{\zeta_j}$ is as in Lemma~\ref{tdG}. Let $\gamma:[0,2\pi]\to D_z$ be a parametrized Jordan curve containing $z$ in its interior, and $\ell:[0,2\pi]\to G_\mathsf{max}$ the associated lift.  Necessarily $\ell([0,2\pi])$ is contained in a single component of $p^{-1}(D_z)$, say  $D_{\zeta_j}$, and  it must contain $\zeta_j$  in its interior (otherwise $\ell$ would be a unit in $\pi_1(D_{\zeta_j}\setminus\zeta_j)$ and so would be $\gamma=p\circ\ell$ in $\pi_1(D_z\setminus z)$). Now, if  $m(\zeta_j)>1$, then we contradict Lemma~\ref{tdG}. Hence,  $\ell$ is valued in a $D_{\zeta_j}$ such that $m(\zeta_j)=1$, which is  thus homeomorphic to $D_z$ under $p$.

 Let $\gamma_n$ be a sequence of Jordan curves in $D_z$, containing $z$ in their interior and shrinking to $z$ when $n\to\infty$. Let further  $\ell_n$ be the corresponding sequence of lifts to \(G_\mathsf{max}\). By what precedes, some subsequence $\ell_{m_n}$ shrinks to \(\zeta_j\) in $D_{\zeta_j}$, for some $j$ such that $D_{\zeta_j}$ is homeomorphic to $D_z$ under $p$. Moreover, $\ell_{m_n}$ contains $\zeta_j$ in its interior. We can now argue as we did to show that no connected component of \( \RS\setminus G_\mathsf{max} \) consists of a single point, and contradict  the maximal character  of \(G_\mathsf{max}\). This completes the proof of the lemma.
 \end{proof}
 
Define \( D_\mathsf{max} := p(G_\mathsf{max}) \) and \( V_\mathsf{max}:=p^{-1}(D_\mathsf{max}) \). Notice that both sets are  open by Lemma~\ref{lem:maximality} along with openness and continuity of $p$.

\begin{lemma}
\label{lem:ler2}
Let \( N_2 := p^{-1}(\widetilde N_2) \), where \( \widetilde N_2 \) is the \( +\infty\)-set of \( g\big(\mu^{(2)},\D;\cdot\big)\) and we set
\begin{equation}
  \label{def_bal2}
\mu^{(2)}:=\big(\mu^{(1)}\big)^{\D\setminus D_\mathsf{max}} \qandq  \nu^{(2)}:=\big(\nu^{(1)}\big)^{\RS\setminus V_\mathsf{max}}.
\end{equation}
Then  \( N_2 = N_1\setminus V_\mathsf{max} \), and for \( z\in\RS\setminus N_2\) we can define
\begin{equation}
\label{ler2}
ler^{(2)}(z) := g\big(\mu^{(2)},\D;p(z)\big) - g\big(\nu^{(2)},\RS;z\big) - h_\RS(z).
\end{equation}
In this case it holds that
\begin{equation}
\label{f12m}
\begin{cases}
\displaystyle \limsup_{\zeta\to z} ler^{(2)}(\zeta) \leq -2/\cp_\D(\K), & z\in \B,  \\
ler^{(2)}(\zeta) = 0, & z\in \RS\setminus ( G_\mathsf{max}\cup N_2 ). 
\end{cases}
\end{equation}
Moreover, we have that
\begin{equation}
\label{f12n}
\big\|\mu^{(2)}\big\|
\begin{cases}
= \|\mu\| & \text{if } \supp(\mu^{(1)})\cap D_\mathsf{max} = \varnothing, \smallskip \\
< \|\mu\| & \text{otherwise}.
\end{cases}
\end{equation}
\end{lemma}
\begin{proof}
By Lemma~\ref{lem:maximality} the set \( \D\setminus D_\mathsf{max} \) is closed  in $\D$ and none of its connected components reduces to a point. Hence, it has no finely isolated points  (remember that a connected set cannot be thin at an accumulation point, see discussion after \eqref{thin}). Thus, \( \D\setminus D_\mathsf{max} \) is its own base. Consequently, \( \RS\setminus V_\mathsf{max}\) is also its own base by Lemma~\ref{lem:fineopen}.  Therefore, we get  from \eqref{char_balayage} that
\begin{equation}
\label{bal3}
\mu^{(2)}(D_\mathsf{max}) = \nu^{(2)}(V_\mathsf{max}) = 0.
\end{equation}
This  implies that the potentials of \( \mu^{(2)} \) and \( \nu^{(2)} \) are harmonic in \( D_\mathsf{max} \) and \( V_\mathsf{max}\), respectively. Moreover, since \( \partial D_\mathsf{max}\setminus\T \) is separated from \( \T \) by the very definition of \( G_\mathsf{max} \), \( \partial V_\mathsf{max}\setminus\partial\RS \) is necessarily separated from \( \partial\RS \) and these potentials extend continuously by zero to \( \T \) and \( \partial \RS \) respectively by the regularity of the latter, see the discussion after \eqref{regGB}. Therefore, \( ler^{(2)} \) extends harmonically to the whole of \( V_\mathsf{max} \) and continuously to $\partial\RS\setminus\mathcal T$ by Lemma~\ref{lem:hprime2}. So, as in Lemma~\ref{lem:ler1}, the inclusion \( N_2 \subseteq N_1\setminus V_\mathsf{max} \) follows directly from \eqref{defbalf} and the equality \( N_2 = N_1\setminus V_\mathsf{max} \) is then deduced from \eqref{char_balayage}.  In addition, the first inequality in \eqref{f12m} holds in view of \eqref{hr-gamma}, the definition of \( h_\RS \) given after \eqref{ler0}, and the fact that \( h^{\prime\prime} \) is non-negative, see Lemma~\ref{lem:pdlet-surface}. 

Since \( N_1\setminus V^\prime = N_0\setminus V^\prime \) by Lemma~\ref{lem:ler1} we have that \( \RS\setminus( V^\prime\cup V_\mathsf{max}\cup N_0) = \RS\setminus( V^\prime\cup V_\mathsf{max}\cup N_1) \). Moreover, \( ler(z) =0 \) on this set by the very definitions of \( V^\prime \) and \( V_\mathsf{max} \). Further,
\[
\RS\setminus(V_\mathsf{max}\cup N_2) = \RS\setminus(V_\mathsf{max}\cup N_1) =  (\RS\setminus( V^\prime\cup V_\mathsf{max}\cup N_1) ) \cup (V^\prime\setminus N_1).
\]
 Since  \( \D\setminus D_\mathsf{max} \) and \( \RS\setminus V_\mathsf{max}\) are their own bases, it therefore follows from \eqref{char_balayage} and \eqref{f11i} that
\begin{equation}
\label{ler-1-2-pr}
ler^{(2)}(z) = ler^{(1)}(z) =0, \quad z\in \RS\setminus(V_\mathsf{max}\cup N_2).
\end{equation}

To study the values of \( ler^{(2)} \) on $V_\mathsf{max} \setminus G_\mathsf{max}$, let us show that this is an open set. Indeed, since \( p(G_\mathsf{max}) = p(V_\mathsf{max}) = D_\mathsf{max} \), for each $\zeta\in V_\mathsf{max}$ there exists a disk $D_{p(\zeta)}\subset D_\mathsf{max}$, centered at $p(\zeta)$, and a point $z\in G_\mathsf{max}$ with $p(z)=p(\zeta)$ such that \( D_\zeta \subset V_\mathsf{max} \) and \( D_z\subset G_\mathsf{max} \), where $D_\zeta$ and $D_z$ are the connected components of $p^{-1}(D_{p(\zeta)})$ that contain $\zeta$ and $z$, respectively. If $\zeta\neq z$, then $D_\zeta\cap G_\mathsf{max}=\varnothing$ since $ G_\mathsf{max}$ lies schlicht over $\D$ and therefore \( D_z\subset V_\mathsf{max} \setminus G_\mathsf{max} \) as claimed. Moreover, since $\RS\setminus V_\mathsf{max}$ is its own base, we get from Lemma~\ref{distregt} that $\RS\setminus (V_\mathsf{max}\setminus G_\mathsf{max}) $ is also its own base.

Pick $z\in V_\mathsf{max} \setminus G_\mathsf{max}$ and notice that $V_\mathsf{max} \setminus G_\mathsf{max}$ consists of at most finitely many connected components. Let $V_z$ be the component containing $z$. As $ G_\mathsf{max}$ contains the annular region delimited by \( J \) and \( \B \) and lies schlicht over $\D$, it follows that $\partial (V_\mathsf{max} \setminus G_\mathsf{max})\cap\partial\RS=\partial\RS\setminus\B$. Since \( h_\RS\equiv 0 \) on \( \partial \RS\setminus \B\) by Lemma~\ref{lem:hprime2}, we get from Lemmas~\ref{supfin}  and~\ref{casparta} that
\begin{equation}
\label{ler1-2-p}
\int h_\RS(x)d\delta_z^{\RS\setminus V_z}(x) = \int h_\RS(x)d\delta_z^{\RS\setminus (V_\mathsf{max}\setminus G_\mathsf{max})}(x)=h_\RS(z).
\end{equation} 
Since \( V_z \),  \( V_\mathsf{max}\setminus G_\mathsf{max} \), and \( G_\mathsf{max} \) are Euclidean open sets, it holds that \( \partial V_z \subseteq \partial (V_\mathsf{max}\setminus G_\mathsf{max}) \subseteq \partial V_\mathsf{max} \), a set on which \( ler^{(1)}(z) \) is zero except possibly on \( N_2 \) by \eqref{ler-1-2-pr}. Consequently, in view of \eqref{ler2}, we see upon using  Lemma~\ref{saturated}, \eqref{GreenHarmonicMeasure}, and \eqref{ler1-2-p} that
\[
ler^{(2)}(z) =  \int ler^{(1)}(x) d\delta_z^{\RS\setminus V_z}(x) = 0,
\]
where the second equality follows from Lemma \ref{supfin} as \( \partial_\mathsf{f} V_z \subseteq \partial V_z \) and $ler^{(1)}(x)=0$ quasi and therefore  $\delta_z^{\RS\setminus V_z}$-almost everywhere on \(  \partial_\mathsf{f} V_z \). 

Finally, we get from \eqref{preservemass} and \eqref{mass-balayage} that \( \|\mu^{(2)}\|=\|\mu^{(1)}\| \) when \( \supp(\mu^{(1)})\cap D_\mathsf{max}=\varnothing \) and \( \|\mu^{(2)}\|<\|\mu^{(1)}\| \) otherwise as well as that \( \|\mu^{(1)}\|=\|\mu\| \), which proves \eqref{f12n}.
\end{proof}

\subsection{Projected Logarithmic Error Function}

Recall that \( z\in \partial G_\mathsf{max} \) is called accessible if there exists a continuous map \( \psi \) on \( [0,1] \) such that \( \psi(t)\in G_\mathsf{max} \) for \( t\in[0,1) \) and \( \psi(1)=z \).

\begin{lemma}
\label{lem:prime-ends}
It holds that \( \operatorname*{card}\left(  p^{-1}(z)\cap (\partial G_\mathsf{max}\setminus E^*) \right) \leq 2 \) for all except countably many \( z\in \partial D_\mathsf{max} \), where \( E^* \subset \partial G_\mathsf{max}\) is the subset of non-accessible points.
\end{lemma}
\begin{proof}
Since \( \T\subset \partial D_\mathsf{max} \) and \( p^{-1}(\T) \cap \partial G_\mathsf{max} = \B \), we only need to consider \( z\in\partial D_\mathsf{max}\cap \D \). The set   \( (\partial D_\mathsf{max}\cap \D)\setminus p({\bf rp}(\RS))\) can be covered by countably many open sets of the form $D_x$, where $x\in (\partial D_\mathsf{max}\cap \D )\setminus p({\bf rp}(\RS))$ and $D_x\subset \D$ is a disk centered at $x$, small enough that each component of $p^{-1}(D_x)$ is homeomorphic to  $D_x$ under $p$. Hence, it is enough to show that for each such $D_x$ and all \( z\in \partial D_\mathsf{max} \cap D_x\) but countably many,  one has \( \operatorname*{card}\left(  p^{-1}(z)\cap (\partial G_\mathsf{max}\setminus E^*) \right) \leq 2 \).  Fix $D_x$ and let $x_1,\ldots,x_N$ denote the preimages of $x$ under $p$. We write $V_{x_\ell}$ for the connected component of $p^{-1}(D_x)$ containing $x_\ell$, $1\leq \ell\leq N$. Then, each $z\in D_x$ has preimages  $z_1,\ldots,z_N$ under $p$, with $z_\ell\in V_{x_\ell}$. When \( z\in \partial D_\mathsf{max} \cap D_x\), it follows from the definition of accessibility that if $z_{\ell}\in \partial G_\mathsf{max}\setminus E^*$, then there is a continuous  arc \( \psi_{z_\ell}: [0,1] \to \overline{G}_{\mathsf{max}}\) such that \( \psi_{z_\ell}(t)\in G_\mathsf{max} \) for \( t\in[0,1) \) and \( \psi_{z_\ell}(1)=z_\ell \). Moreover, on shortening and reparametrizing the arc if necessary, we may assume that $\psi_{z_\ell}$ is valued in $V_{x_\ell}$. Now, if \( p^{-1}(z)\cap (\partial G_\mathsf{max}\setminus E^*) \) contains 3  distinct points, say $z_{\ell_j}$ for $j=\{1,2,3\}$, then $p(\psi_{z_{\ell_j}})$ are Jordan arcs  having only the point $z$ in common, because $G_\mathsf{max}$ lies schlicht over $\D$ and the $V_{x_{\ell_j}}$ are disjoint. Thus, $T_z:=\cup_j p(\psi_{z_{\ell_j}})$ is a triod, and if we had uncountably many such $z$ then some triple $(\ell_1,\ell_2,\ell_3)$ would occur uncountably many times. Assigning different colors to $\psi_{z_{\ell_1}}$,  $\psi_{z_{\ell_2}}$ and $\psi_{z_{\ell_3}}$, we get and uncountable collection of colored triods $T_z$ whose arcs of different colors never meet, again because $G_\mathsf{max}$ lies schlicht over $\D$ and the $V_{x_{\ell_j}}$ are disjoint. However, this contradicts the Moore triod theorem \cite[Proposition 2.18]{Pommerenke}, thereby finishing the proof.
\end{proof}

Put \( K_\mathsf{max}:=\D\setminus D_\mathsf{max} \). It follows from Lemma~\ref{lem:maximality} that $f$ has a single-valued meromorphic continuation throughout \( D_\mathsf{max}\setminus p(E_f) \) and \( K_\mathsf{max} \) is a compact subset of \( \D \). Hence, \( K_\mathsf{max}\in\mathcal{K}_f\), where \( \mathcal K_f \) was defined just before \eqref{mincapset}. Clearly, the measure \( \mu^{(2)} \) is supported on \( K_\mathsf{max} \).

\begin{lemma}
\label{lem13b}
Let \( m_\zeta \) be the ramification order of \( \zeta\in\RS \). Define
\begin{equation}
\label{f13a0}
ler_\mathsf{pr}(z) := \sum_{\zeta\in p^{-1}(z)} m_\zeta ler^{(2)}(\zeta), \quad z\in\D\setminus\widetilde N_2.
\end{equation}
Then \(  ler_\mathsf{pr} \) is a \( \delta \)-subharmonic function in \( \D \) such that \(  ler_\mathsf{pr}(z)=0 \) when \( z\in K_\mathsf{max}\setminus\widetilde N_2 \) and
\[
\limsup_{z\to \zeta} ler_\mathsf{pr}(z) \leq -2/\cp_\D(\K), \quad \zeta\in\T.
\]
Moreover, there exist non-negative measures \( \mu_\mathsf{pr} \) and \( \nu_\mathsf{pr} \), supported on \( \partial D_\mathsf{max} \), and a non-negative function  \( h_\mathsf{pr} \), harmonic in \( \D \),  such that \( \mu_\mathsf{pr} \leq 2\mu^{(2)} \) and for \( z\in\D\setminus\widetilde N_2 \)
\begin{equation}
ler_\mathsf{pr}(z) = g(\mu_\mathsf{pr},\D;z) - g(\nu_\mathsf{pr},\D,z) - h_\mathsf{pr}(z) - 2/\cp_\D(\K).
\label{f13g}
\end{equation}
\end{lemma}
\begin{proof}
The first two claims of the lemma follow readily from \eqref{f12m} and the computation in \eqref{defetagev}. Let \( \widehat\mu^{(2)} \) be the pullback of \( \mu^{(2)} \) onto \( \RS \), see \eqref{lifted-measureA} and \eqref{remontp}. Then we can equivalently write
\begin{align*}
ler^{(2)}(z) &= g\big(\widehat\mu^{(2)},\RS;z \big) - g\big(\nu^{(2)},\RS;z \big) - h_\RS(z) \\
& =  g(\mu_\mathsf{R},\RS;z) - g(\nu_\mathsf{R},\RS;z)  - h_\RS(z),
\end{align*}
where \( \mu_\mathsf{R} - \nu_\mathsf{R} = \widehat\mu^{(2)} - \nu^{(2)} \) and \( \mu_\mathsf{R} \), \( \nu_\mathsf{R} \) are positive mutually singular measures on \( p^{-1}(K_\mathsf{max}) \), i.e., \( \mu_\mathsf{R} - \nu_\mathsf{R} \) is the Riesz charge of a \( \delta \)-subharmonic function \( ler^{(2)} \). Clearly, \( \nu_\mathsf{R}\leq \nu^{(2)} \), \( \mu_\mathsf{R} \leq \widehat\mu^{(2)} \), and for any Borel set \( B\subset p^{-1}(K_\mathsf{max})\setminus {\bf rp}(\RS) \) that lies schlicht over \( \D \) it holds that
\begin{equation}
\label{muRmupr}
\mu_\mathsf{R}(B) \leq \mu^{(2)}(p(B)). 
\end{equation}

Similarly to the computation in \eqref{defetagev}, one can show that \( \sum_{\zeta\in p^{-1}(z)} m_\zeta h(\zeta) \) is harmonic in \( \D \) when \( h \) is harmonic on \( \RS \). Hence, it follows from the maximum principle for harmonic functions and Lemma~\ref{lem:hr} that
\[
\sum_{\zeta\in p^{-1}(z)} m_\zeta h_\RS(\zeta) = 2/\cp_\D(\K) + h_\mathsf{pr}(z), \quad h_\mathsf{pr}(z) := \sum_{\zeta\in p^{-1}(z)} m_\zeta h^{\prime\prime}(\zeta).
\]
Thus, if we set \( \mu_\mathsf{pr} = p_*(\mu_\mathsf{R}) \) and \( \nu_\mathsf{pr} = p_*(\nu_\mathsf{R}) \), see \eqref{push-down}, we get \eqref{f13g}.

As explained in Lemma~\ref{lem:ler2}, \( \RS\setminus V_\mathsf{max} \) is its own base and \( V_\mathsf{max} \setminus G_\mathsf{max} \), \( G_\mathsf{max} \) are disjoint open sets. Hence, \( \RS \setminus G_\mathsf{max} \) is its own base by Lemma~\ref{distregt}. Moreover,  the function \( ler^{(2)} \) is equal to zero  on this set by \eqref{f12m}. Hence, we get from the proof of \cite[Theorem~2]{EremFugSod92} (that carries over {\it mutatis mutandis} to any hyperbolic surface) the implication:
\begin{equation}
  \label{EFS}
\delta_z^{\RS\setminus G_\mathsf{max}}(F) =0 \quad \Rightarrow \quad (\mu_\mathsf{R}+\nu_\mathsf{R})(F)=0
\end{equation}
for any \( F\subset \RS\setminus G_\mathsf{max} \) and  some (therefore any) \( z\in G_\mathsf{max} \). In particular, it follows from Lemma~\ref{supfin} that \( \supp(\mu_\mathsf{R}+\nu_\mathsf{R})\subseteq \partial G_\mathsf{max} \setminus\mathcal{T}\) and that this measure does not charge polar sets. Hence, since \( G_\mathsf{max} \) lies schlicht over \( D_\mathsf{max} \), it holds that \( \supp(\mu_\mathsf{pr}+\nu_\mathsf{pr})\subseteq \partial D_\mathsf{max} \setminus\T\). Moreover, let \( \tilde E \) be the set of points \( z\in \partial D_\mathsf{max}  \) such that \( \operatorname*{card}\left(  p^{-1}(z)\cap (\partial G_\mathsf{max}\setminus E^*) \right) > 2 \), where \( E^* \) is the set of non-accessible points of \( \partial G_\mathsf{max} \). It follows from Lemma~\ref{lem:prime-ends} that \( \tilde E \) is at most countable and therefore is not charged by \( \mu_\mathsf{pr} \). Of course, the same is true of \( p({\bf rp}(\RS)) \), as it is a finite set. For any Borel set \( B \subseteq \partial D_\mathsf{max} \), let \( \tilde B := B\setminus(\tilde E \cup p({\bf rp}(\RS))) \). Then, we obtain
\begin{align*}
\mu_\mathsf{pr}(B) &= \mu_\mathsf{pr}(\tilde B) = \mu_\mathsf{R} \big( p^{-1}(\tilde B)\cap \partial G_\mathsf{max} \big) \\
&= \mu_\mathsf{R} \big( p^{-1}(\tilde B)\cap (\partial G_\mathsf{max}\setminus E^*) \big)  \leq 2  \mu^{(2)}(B),
\end{align*}
where the third equality follows from \cite[Corollary~2]{EremFugSod92}, which says that the Riesz charge of a \( \delta \)-subharmonic function cannot charge points that are non-accessible from the complement of the base of its zero set, and the last inequality is due to Lemma~\ref{lem:prime-ends} and \eqref{muRmupr}.
\end{proof}

\subsection{Computation of the Logarithmic Error Function}
\label{ssec_concl}

Lemma~\ref{lem13b} implies that \( ler_\mathsf{pr}(z)=0 \) for \( z\in K_\mathsf{max}\setminus \widetilde N_2 \). Hence, it follows from \eqref{f13g} and properties of the Green equilibrium potential that
\begin{equation}
\label{proof_lb1}
g\big( \mu_\mathsf{pr},\D;z \big) - g(\nu_\mathsf{pr},\D;z) - h_\mathsf{pr}(z) = 2\frac{\cp_\D(K_\mathsf{max})}{\cp_\D(\K)} g \big(\mu_{\D,K_\mathsf{max}},\D;z \big)
\end{equation}
for \( z\in K_\mathsf{max}\setminus \widetilde N_2 \) (since \( K_\mathsf{max} \) is equal to its own base, \( g \big(\mu_{\D,K_\mathsf{max}},\D;\cdot \big) \equiv 1/\cp_\D(K_\mathsf{max}) \) on \( K_\mathsf{max} \)). Assume that either \( \nu_\mathsf{pr} \) is a non-zero measure or \( h_\mathsf{pr} \) is strictly positive (harmonic) function in  \( \D \). Then, since \( K_\mathsf{max} \) is separated from \( \T \), it would hold that 
\[
g\big( \mu_\mathsf{pr},\D;z \big) >  2\frac{\cp_\D(K_\mathsf{max})}{\cp_\D(\K)} g \big(\mu_{\D,K_\mathsf{max}},\D;z \big)
\]
for \( z\in K_\mathsf{max}\setminus \widetilde N_2 \). Since \( \supp(\mu_{\D,K_\mathsf{max}} )\subseteq K_\mathsf{max} \) and   \( K_\mathsf{max} \) is its own base, the \hyperref[sdp]{Strong Domination Principle}, see Section~\ref{ssec_thin}, implies that the above inequality holds everywhere in \( \D \). Thus,  integrating both sides of this inequality against \( \mu_{\D,K_\mathsf{max}} \) which is a probability measure, we get on using Tonelli's theorem on the left-hand side and multiplying by \( \cp_\D(K_\mathsf{max}) \)  that
\begin{equation}
\label{proof_lb5}
\|\mu_\mathsf{pr} \| > 2\frac{\cp_\D(K_\mathsf{max})}{\cp_\D(\K)}.
\end{equation}
On the one hand, by Lemma~\ref{lem13b} and equation \eqref{f12n} together with the very construction of \( \mu \), we have that \( \|\mu_\mathsf{pr} \| \leq 2\|\mu^{(2)}\| \leq 2\|\mu\|\leq 2 \). On the other hand holds  \( \cp_\D(K_\mathsf{max})\geq\cp_\D(\K) \), see \eqref{mincapset}. These observations clearly show that \eqref{proof_lb5} is impossible. Hence, it is necessarily the case that
\begin{equation}
\label{proof_lb4}
 \nu_\mathsf{pr} = 0 \quad \text{and} \quad h_\mathsf{pr}(z) \equiv 0, \;\; z\in\D.
 \end{equation}
Then, one can rewrite \eqref{proof_lb1} as
\begin{equation}
\label{proof_lb2}
g\big( \mu_\mathsf{pr},\D;z \big) = 2\frac{\cp_\D(K_\mathsf{max})}{\cp_\D(\K)} g \big(\mu_{\D,K_\mathsf{max}},\D;z \big)
\end{equation}
for \( z\in K_\mathsf{max}\setminus \widetilde N_2 \). Using now the \hyperref[sdp]{Strong Domination Principle} in both directions, we get that \eqref{proof_lb2} holds for every \( z\in \D \). Therefore,
\begin{equation}
  \label{inegpfi}
\mu^{(2)} \geq \frac12 \mu_\mathsf{pr} = \frac{\cp_\D(K_\mathsf{max})}{\cp_\D(\K)} \mu_{\D,K_\mathsf{max}}
\end{equation}
which, upon recalling once again \eqref{mincapset} and the fact that \( \mu_{\D,K_\mathsf{max}} \) is a probability measure, gives us
\begin{equation}
\label{proof_lb3}
\cp_\D(K_\mathsf{max})=\cp_\D(\K) \qandq \mu^{(2)} = \mu_{\D,K_\mathsf{max}}.
\end{equation}
In particular $\widetilde{N}_2=\varnothing$ and the first equality in \eqref{proof_lb3}, combined with the minimality and uniqueness of \( \K \), yields that \( K_\mathsf{max} = \K \). In addition, as $\|\mu^{(2)}\|=1$ by \eqref{proof_lb3}, we get from \eqref{f12n} that $\supp(\mu^{(1)})\cap D_\mathsf{max} =\varnothing$ and therefore $\mu^{(1)}=\mu^{(2)}$, by \eqref{char_balayage}, \eqref{def_bal2}, and the fact that $\D\setminus D_{\mathsf{max}}$  is its own base (because so is \( \RS \setminus G_\mathsf{max} \)  and we can use Lemma \ref{lem:fineopen}). Moreover, we also get from \eqref{f12n} that $\supp_{\mathsf{f}}(\mu^{(1)})\cap D_\mathsf{max} = \varnothing$ (note that $\supp_{\mathsf{f}}(\mu^{(1)})$ exists because $\mu^{(1)}$ is admissible, see Section \ref{ssec_bal}).  Remembering that the fine open set $D^\prime$ is regular by \eqref{DprimeVprime}, we get from \eqref{preservemass} and Lemma~\ref{supfin} that $\supp_{\mathsf{f}}(\mu^{(1)})\supset \cup_{t}\partial_{\mathsf{f}}D^\prime_t$ where $D^\prime_t$ are the finely connected components of $D^\prime$ such that $\mu(D_t)\neq0$. However, since $K_\mathsf{max}$(= \( \K \)) has no fine interior by inspection of \eqref{Kf}, each $\partial_{\mathsf{f}}D_t^\prime$ must intersect  $D_\mathsf{max}$ whenever $D^\prime_t$ is  nonempty. Hence, \(\mu\) cannot charge $D^\prime$ as otherwise it would contradict that $\supp_{\mathsf{f}}(\mu^{(1)})\cap D_{\mathsf{max}}=\varnothing$. Consequently, since  $\D\setminus D^\prime$ is its own base, we conclude in view of  \eqref{char_balayage} and \eqref{def_bal1} that $\mu^{(1)} = \mu $. Thus, we obtain altogether that
\begin{equation}
\label{concl1}
\mu_{\D,\K} = \mu^{(2)} = \mu^{(1)} = \mu = \mu^\prime,
\end{equation}
where the last equality comes from the inequalities \( \mu\leq\mu^\prime \) and \( \|\mu^\prime\|\leq1 \), see \eqref{mun-muprime}. In addition, since $\mu_{\D,\K}$ has finite potential everywhere, we get that $\widetilde{N}_0=N_0=\varnothing$ whence also $\widetilde{N}_1=N_1=\varnothing$.

From \eqref{concl1} one sees that $\mu$ does not charge $D_{\mathsf{max}}$, implying in view of \eqref{ler} that $ler$ is subharmonic on $V_{\mathsf{max}}=p^{-1}(D_{\mathsf{max}})$. In particular, since $g (\mu_{\D,K_\mathsf{max}},\D;\cdot)$ extends continuously by zero on $\T$ and  $h^\prime$ by $2/\cp_\D(K_f)$ on $\mathcal{T}$, see Lemma~\ref{lem:hr}), while $\mathrm{fine}\lim_{z\to \zeta} ler(z)=ler(\zeta)<0$ when $\zeta\in J$ by the fine continuity of $ler$ and the fact that $J\subset G_-$, it follows from the relative fine boundary maximum principle \cite[Theorem 10.8]{Fuglede} that $ler<0$ in the annular region $\mathcal{A}(\mathcal{T},J)$ bounded by $\mathcal{T}$ and $J$, since it is bounded above by the fine potential $g (\widehat{\mu}_{\D,K_\mathsf{max}},\RS;\cdot)$ there, see \eqref{remontp}. Let now $G^\prime$ be a finely connected component of $G_+$.  Since \( G^+ = p^{-1}(p(G^+)) \) by Lemma~\ref{lem11a} and $\mathcal{A}(\mathcal{T},J)\subset G_-$, $p(G^\prime)$ lies in the interior of $p(J)$. Hence, $D^\prime = p(G^\prime)\in D_J$ and by what precedes $D^\prime \cap \K = \varnothing$ as otherwise $\mu=\mu_{\D,\K}$ would charge $D^\prime$ because it is carried by the whole set \( \K \) according to \eqref{fsuppeqm}, i.e., it cannot be carried by $K_f\setminus D^\prime$ (which is finely closed). Consequently, $G^\prime\cap p^{-1}(K_f)=\varnothing$ which implies that $ler\leq0$ on $p^{-1}(K_f)$, and the relative fine boundary maximum principle in turn implies that $ler\leq0$ on  $V_\mathsf{max}$. Moreover, $ler<0$ on $G_\mathsf{max}$ as it is strictly negative on  $\mathcal{A}(\mathcal{T},J)$. Immediately we deduce that $G_+=\varnothing$ and \( G_\mathsf{max}\subseteq G_- \). Furthermore, since \( G_- \) lies schlicht over \( \D \) by  Lemma~\ref{lem12a} while \( K_f \) has no fine interior, maximality of \( G_\mathsf{max} \) implies that \( G_\mathsf{max}= G_- \). Altogether, we obtain that \( D^\prime = V^\prime = \varnothing \) and $G_-=G_\mathsf{max}$. In particular, the step of Lemma~\ref{lem:ler1} is vacuous and \( \nu^{(1)}=\nu \) as well as $ler=ler^{(1)}$. Moreover, \( ler(z)=0 \) holds for \( z\in \RS\setminus G_\mathsf{max} \).

Next, by \eqref{inegpfi}, \eqref{concl1} and  the construction of the measures \( \mu_\mathsf{pr} \), \( \mu_\mathsf{R} \) in Lemma~\ref{lem13b}, along with the discussion after \eqref{EFS}, one has
\[
2\mu_{\D,\K}  = 2\mu^{(2)} = \mu_\mathsf{pr} = p_*(\mu_\mathsf{R}) \qandq \mu_\mathsf{R} \leq (\widehat\mu_{\D,\K})_{\mathcal b\partial G_\mathsf{max}}.
\]
Lemma~\ref{lem:prime-ends} now yields that this last inequality is in fact an equality. As \(\mu_\mathsf{R}-\nu_\mathsf{R}\) is the Riesz charge of $ler^{(2)}$ that vanishes on $\RS\setminus G_\mathsf{max}$ by \eqref{f12m}, the discussion after \eqref{EFS} implies \( \nu_\mathsf{R} = \nu^{(2)}_{\mathcal b \partial G_\mathsf{max}} \), and since \( \nu_\mathsf{pr} = p_*(\nu_\mathsf{R}) \) we get from \eqref{proof_lb4} and \eqref{def_bal2} that
\begin{equation}
\label{concl2}
0 = \nu^{(2)}(\overline G_\mathsf{max}) = \nu^{(1)}(\overline G_\mathsf{max}) = \nu(\overline G_\mathsf{max}),
\end{equation}
where the middle equality holds by \eqref{preservemass} and Lemma \ref{supfin} (because $\delta_z^{\RS\setminus G_\mathsf{max}}$ is a strictly positive function of $z$ in the regular open set $ G_\mathsf{max}$), while the last equality comes from $ler=ler^{(1)}$. Recall that \( \nu= \nu^{\prime\prime} + \nu^\prime = \nu^* + \tilde \nu + \nu^\prime \), see \eqref{ler0} and \eqref{pdlet-8}, where \( \nu^* \) the vague limit point of \( \{\nu_n\} \) in \( \RS \), see Lemma~\ref{lem:pdlet-surface}. As \( ler(z)=0 \) for \( z\in \RS\setminus G_\mathsf{max} \) and $ler$ is a \( \delta \)-subharmonic function, we get from \cite[Theorem~2]{EremFugSod92} that its Riesz charge is supported on \( \partial G_\mathsf{max} \). In view of \eqref{concl2}, it entails that
\begin{equation}
\label{concl3}
(\widehat\mu_{\D,\K})_{\mathcal b \overline{p^{-1}(K_f)\setminus\partial G_\mathsf{max}}} = \nu = \nu^{\prime\prime }= \nu^*,
\end{equation}
where \( \tilde \nu + \nu^\prime = 0\) since it is a measure supported on \( E_f \), see Lemma~\ref{lem:hr-e} and \eqref{pdlet-6}, while \( \nu \) does not charge polar sets by the first equality above. Since \( \nu=\nu^{(1)} \) and \( \nu^{(1)} \) has no mass in \( V_\mathsf{max} \) by \eqref{concl3}, the step of Lemma~\ref{lem:ler2} was also vacuous. We thus get that \( ler=ler^{(2)} \), and it follows from \eqref{proof_lb4} together with the construction of \( h_\mathsf{pr} \) in Lemma~\ref{lem13b} that \( h_\RS=h^\prime \). Therefore, the inequality in \eqref{f12m} is in fact an equality by Lemma~\ref{lem:hr},
and consequently
\begin{equation}
\label{concl4}
ler(z) = 
\begin{cases}
2 g(\mu_{\D,\K,},\D;p(z)) - 2/\cp_\D(\K), &  z\in G_\mathsf{max}, \smallskip \\
0, & z\in\RS\setminus G_\mathsf{max},
\end{cases}
\end{equation}
as both sides of \eqref{concl4} are continuous on \( \overline{\RS} \), equal to zero on \( \RS\setminus G_\mathsf{max} \), equal to \( -2/\cp_\D(\K) \) on \( \B \), and  harmonic in \( G_\mathsf{max} \) so that the equality in \( G_\mathsf{max} \) is consequence of the maximum principle for harmonic functions. 

Notice that we started our proof with the limit \eqref{f6a} taking place along the full sequence \( \N \) of integers, that was later refined into a subsequence in Section~\ref{sec_sripped_error}, and refined still further in Section~\ref{ssec:ler} to a subsequence \( \mathcal N \) along which all the above results hold. However, we could have initiated our argument using any subsequence  \( \N_0\subset\N \) in \eqref{f6a} with the same conclusions holding along some \( \mathcal N\subseteq\N_0 \). Hence, if there existed a subsequence along which either \( \mu^\prime\neq\mu_{\D,\K} \), \( \nu^\prime\neq 0 \), \( \tilde\nu\neq 0 \),  or \( \nu^* \) are not as in the left-hand side of \eqref{concl3}, or if \( h^{\prime\prime} \) were not equal to zero, then we could use it as \( \N_0 \) in \eqref{f6a} to arrive at a contradiction. Hence, all lemmas of this section hold along the full sequence \( \N \).

This proves \eqref{wcopg} and the analogous assertion in Theorem \ref{thm:main2}.

\subsection{Convergence in Capacity}
\label{ssec:conv_cap}

In view of Sections~\ref{ssec_reduc}-\ref{ssec_nehari}, to complete the proof of  Theorems~\ref{thm:main1} and~\ref{thm:main2}, we only need to establish that
\begin{equation}
\label{concl6}
\frac1{2n}\log | (f-N(M_n))(z) | \, \cic \, g\big(\mu_{\D,\K},\D;z\big) - \frac1{\cp_\D(\K)} \qasq n\to\infty, \;\; z\in \D\setminus\K.
\end{equation}

Let \( K\subset\D\setminus\K \) be compact. Further, let \( V \) be an open neighborhood of \( p^{-1}(\K) \) on \( \RS \) whose closure is disjoint from \( K \), and \( U  \)  an open subset of \( \RS \) containing \( K \) whose closure is disjoint from \( \overline V\cup \B \).  For convenience, we also assume that \( V=p^{-1}(p(V)) \). Recall from \eqref{gen-bl}--\eqref{hnm} the relation
\[
\frac1{2n}\log | (f-N(M_n)\circ p)(z) | = \frac12\big(h_{n,n}(z) + g(\mu_n,\D;p(z)) - g(\nu_n,\Omega_n;z) \big), \quad z\in\Omega_n.
\]
We have established in the previous section that the functions \( h_{n,n} \) converge to \( - h_{\RS} = -h^\prime \) locally uniformly in \( \RS\setminus E_f \) and therefore on \( K \), see Lemma~\ref{lem:hr-e} (note that $\mathcal{N}^\prime$ may now be replaced by $\N$ by the last remark in Section~\ref{ssec_concl}). Since
\[
\frac12\big(-h_{\RS}(z) + g(\mu_{\D,\K},\D;p(z)) - g(\nu,\RS;z) \big) = \frac12 ler(z) = g(\mu_{\D,\K},\D;p(z)) - \frac1{\cp_\D(\K)}
\]
for \( z\in G_\mathsf{max} \) by \eqref{ler0} and \eqref{concl4}, to obtain \eqref{concl6} we only need to establish that for any \( a>0 \)
\begin{equation}
\label{bsy}
\begin{cases}
\displaystyle\lim_{n\to\infty} \cp_\D\big( \{z\in p(K):|g(\mu_n,\D;z) - g(\mu_{\D,\K},\D;z)|> a\} \big) = 0, \medskip \\
\displaystyle\lim_{n\to\infty} \cp_\D\big( p\{z\in K:| g(\nu_n,\Omega_n;z) - g(\nu,\RS;z)|> a\} \big) = 0.
\end{cases}
\end{equation}

Write \( \mu_n=\mu_{n,1}+\mu_{n,2} \), where \( \mu_{n,1}:= \mu_{n\mathcal b p(V)} \). Notice that \( \mu_{n,1} \cws \mu_{\D,\K} \) and therefore the differences \( g(\mu_{n,1},\D;\cdot) - g(\mu_{\D,\K},\D;\cdot) \) converge uniformly to zero on \( K \).  Recall that the Green potential of a measure supported in \( \D \) can be written as a difference of the logarithmic potentials of that measure and of its balayage onto \( \T \), see \eqref{balogh}. In \cite[Lemma~21]{BStYa}, it was shown  that when compactly supported measures converge weak$^*$ to the zero measure then their logarithmic potentials converge to zero in logarithmic capacity on \( \C \),  see \eqref{deflogcap} for the definition of logarithmic capacity. Thus, the first relation in \eqref{bsy} holds but in logarithmic capacity.  However, since $t\mapsto g_D(t,z)+\log|z-t|$  is uniformly  bounded for $z$ on compact subsets of $\D$, the convergence in logarithmic and Greenian capacities in \( \D \) are equivalent. This finishes the proof of the first limit in \eqref{bsy}.

The proof of the second limit in \eqref{bsy} proceeds similarly, but requires a more detailed analysis. As we have shown in the previous subsection, the measures $\nu_{n\mathcal b\Omega_n}$ converge vaguely to $\nu^*$ on \( \RS \), where \( \nu^* \) stands for the left-hand side of \eqref{concl3}. Hence, one has the limiting relation
\begin{equation}
\label{wsc-1}
\nu_{n,1}:=\nu_{n\mathcal b \overline V\cap\Omega_n} \overset{w*}{\to} \nu^*
\end{equation}
on \( \RS \). The functions \( g_{\mathcal R}(z,w) - g_{\Omega_n}(z,w) \) uniformly converge to zero for \( w\in V \) and \( z\in K \), see \eqref{etalm} (we extend \( g_{\Omega_n}(z,w) \) by zero to \( \RS\setminus \Omega_n \)). Thus, as the measures \( \nu_{n,1} \) have uniformly bounded masses (for they converge weak$^*$ to a finite measure), we get that
\begin{equation}
\label{wsc-2}
\lim_{n\to\infty} \big(g(\nu_{n,1},\RS;z) - g(\nu_{n,1},\Omega_n;z)\big) = 0
\end{equation}
uniformly on \( K \).  Moreover, since the \( g(\nu_{n,1},\RS;\cdot) \) are positive harmonic functions on \( U \) for all \( n \) large, they converge uniformly to \( g(\nu^*,\RS;\cdot) \) on \( K \) (they converge pointwise on \( U \) by \eqref{wsc-1} and hence uniformly on \( K \) by \hyperref[hst]{Harnack's theorem}). That is, we get from \eqref{wsc-2} that
\begin{equation}
\label{wsc-3}
\lim_{n\to\infty} g(\nu_{n,1},\Omega_n;z)  = g(\nu^*,\RS,z)
\end{equation}
uniformly on \( K \).  Next, let \( \nu_{n,2} := \nu_{n\mathcal b U} \). As the
\( \nu_n \) converge vaguely to \( \nu^* \) and $\overline{U}\cap(\overline{V}\cup\mathcal{T})=\varnothing$, the measures \( \nu_{n,2} \) converge strongly to the zero measure. Subsequently, the measures \( p_*(\nu_{n,2}) \)  strongly converge to zero. As before, it entails that the potentials \( g(p_*(\nu_{n,2}),\D;p(\cdot)) \) as well as \( g(\nu_{n,2},\RS;\cdot) \) converge to zero in capacity on \( K \), by \eqref{push-down}.  As the latter potentials majorize \( g(\nu_{n,2},\Omega_n;\cdot) \), these in turn converge to zero in capacity on \( K \). Finally, define \( \nu_{n,3} := \nu_n - \nu_{n,1} - \nu_{n,2} \). The potentials \( g(\nu_{n,3},\Omega_n;\cdot) \) form a sequence of positive harmonic functions in \( U \). By \hyperref[hst]{Harnack's theorem} there exists a subsequence of indices, say \( \mathcal N_0 \), along which these potentials converge locally uniformly  on \( U \) to some \( h_0 \) which is a non-negative harmonic function or $+\infty$. Now, we can initiate the proof of Theorem~\ref{thm:main2} in Section~\ref{sec_sripped_error} with \( \mathcal N_0 \) instead of \( \N \). Then, it would follow from \eqref{concl3} and \eqref{wsc-3} that the function \( h^{\prime\prime} \) from Lemma~\ref{lem:pdlet-surface} must coincide with \( h_0 \) in \( U \). Proceeding with the remainder of the proof we again inevitably arrive at the conclusion that \( h^{\prime\prime}\equiv0 \). Hence, the potentials \( g(\nu_{n,3},\Omega_n;\cdot) \) converge to \( 0 \) locally uniformly on \( U \) and thus uniformly on \( K \). This finishes the proof of the second limit in \eqref{bsy} and, respectively, of \eqref{concl6}; that is to say, of Theorem~\ref{thm:main2}.

Finally, we need to account for Remark~\ref{convlogcapv}, in which we claimed that for any sequence $\{M_n\}$ of $n$-th root optimal meromorphic approximants  to $f\in \mathcal F(\overline\C\setminus D)$ and any $\epsilon>0$, it holds that
\begin{equation}
\label{lower_bound_log_cap}
\lim_{n\to\infty} \cp\left\{z\in K:\frac1{2n}\log | (f-M_n)(z) |- g\big(\mu_{D,\K},D;z\big)+ \frac1{\cp_D(\K)}  >\epsilon\right\} = 0
\end{equation}
for every compact set \( K\subset \overline D\setminus \K \); moreover, if $|f-M_n|^{1/n}$ is asymptotically constant on $T$, i.e., if
\begin{equation}
\label{asce}
\lim_{n\to\infty}\inf_{z\in\mathcal{T}}|(f-M_n)(z)|^{1/n} =  \lim_{n\to\infty}\sup_{z\in\mathcal{T}}|(f-M_n)(z)|^{1/n}=\exp\left\{-\frac{2}{\cp_D(K_f)}\right\},
\end{equation}
then the limit in \eqref{lower_bound_log_cap} is supplemented by 
\begin{equation}
\label{upper_bound_log_cap}
\lim_{n\to\infty} \cp\left\{z\in K: \frac1{2n}\log | (f-M_n)(z) |- g\big(\mu_{D,\K},D;z\big)+ \frac1{\cp_D(\K)}  < - \epsilon\right\} = 0.
\end{equation}
Together, \eqref{lower_bound_log_cap} and \eqref{upper_bound_log_cap} establish a full analog of \eqref{concl6} but in logarithmic rather than Greenian capacity, and not just in \( D\setminus K_f \) but in \( \overline D\setminus \K \). Notice also that we work directly on the domain \( D \) here, because logarithmic capacity is not preserved by conformal maps. 

Since convergence in logarithmic and Greenian capacities are equivalent on compact subsets of \( D\setminus \K \), we only need to establish  \eqref{lower_bound_log_cap} and \eqref{upper_bound_log_cap} for a compact set \( K_\eta \) of the form
\[
K_\eta:=\big\{ z\in \overline D:1-\eta \leq \big|\phi^{-1}(z) \big| \leq1 \big\}
\]
for some \( \eta>0 \) sufficiently small, where \( \phi:\D \to D \) is a conformal map. Let \( K_\eta^\prime = \overline{K_{2\eta}\setminus K_\eta} \). We claim that for all \( n \) large enough there exists \( r_n \in (1-4\eta,1-2\eta) \) such that
\begin{equation}
\label{small_circle}
\left|\frac1{2n}\log | (f-M_n)(z) |- g\big(\mu_{D,\K},D;z\big)+ \frac1{\cp_D(\K)}\right|  \leq \epsilon,
\end{equation}
for all \( z\in T_{r_n} \), where \( T_r := \phi(\T_r) \). Indeed, let \( E_{n,2\eta} \) be the set of all \( z\in K_{2\eta}^\prime \) for which the above inequality is false. As we have already shown, \eqref{concl6} holds for the $n$-th root optimal approximants \( M_n \), both in Greenian and logarithmic capacities on compact subsets of \( D\setminus \K \), in particular, on \( K_{2\eta}^\prime \). Hence, \( \cp(E_{n,2\eta}) \to 0 \) as \( n\to\infty \). On the other hand, if our claim were false, then \( E_{n,2\eta} \cap T_r \neq \varnothing \) for every \( r\in[1-4\eta,1-2\eta] \). In this case it must hold that
\[
\cp(E_{n,\eta})\geq \frac{\cp([1-4\eta,1-2\eta])}{\|(\phi^{-1})^\prime\|_{K_\eta^\prime}} = \frac{\eta}{2\|(\phi^{-1})^\prime\|_{K_\eta^\prime}}
\]
because contractive maps, in particular $\phi^{-1}_{\mathcal{b}K^\prime_\eta}/\|(\phi^{-1})^\prime\|_{K_\eta^\prime}$ followed by the circular projection, do not increase logarithmic capacity \cite[Theorem 5.3.1]{Ransford}. The above contradiction proves our claim. Clearly, \( T_{r_n} \) is a Jordan curve homotopic to \( T \) in $D\setminus K_f$ for all \( \eta \) small enough. By adding the same constant to \( f \) and \( M_n \) we can assume that \( f \) is not vanishing on \( K_{2\eta}^\prime \). Hence, we deduce from \eqref{small_circle} that $|f-M_n|<|f|$ on $T_{r_n}$ for all $n$ large enough. Besides, the same inequality certainly holds on $T$ and therefore, by Rouch\'e's theorem, the number of zeros of $f-M_n$ in $K_{2\eta}$ is less that the number of its poles plus the degree of $f$ on $T\cup T_{r_n}$ (which is bounded by a constant independently of $n$). It now follows from Theorem~\ref{thm:main2} that \( f-M_n \) has little \( o(1) \) poles and zeros in \( K_{2\eta} \) as \( n\to\infty \). Denote by \( b_{n,\eta}^{pole} \) and \( b_{n,\eta}^{zero} \) the Blaschke products in \( D \) vanishing at the poles and zeros of \( f-M_n \) in \( K_{2\eta} \), respectively. As Green potentials of measures converging
strongly to zero converge to zero in capacity in \( D \),  the functions
\[
e_n(z) := \frac{1}{2n}\log|(f-M_n)(z)b_{n,\eta}^{pole}(z)/b_{n,\eta}^{zero}(z)|-g(\mu_{D,K_f},D;z)+\frac1{\cp_D(\K)}
\] 
still converge to zero in capacity in \( D\setminus K_f \). Hence, similarly to \eqref{small_circle}, we can find \( r_n^*\in(1-2\eta,1-\eta) \) such that \( |e_n(z)| < \epsilon \) for all \( z\in T_{r_n^*} \). As $\limsup_n \sup_{z\in T}e_n(z)\leq0$  by the $n$-th root optimality of $e_n$ and $\lim_n \|e_n\|_T=0$ if in addition \eqref{asce} holds (asymptotic circularity of the error), the desired conclusion \eqref{lower_bound_log_cap} and \eqref{upper_bound_log_cap} come out by applying the maximum principle for the harmonic function to $e_n$ on \( K_\eta \), along with the minimum principle if \eqref{asce} holds.

\section{Proof of Theorem~\ref{thm:main2p}}

Since the considerations of Sections~\ref{ssec_reduc} and~\ref{ssec_nehari} still apply, we can assume that \( D=\D \), i.e., that  $f$ is analytic in $\overline{\C}\setminus E,$ where $E\subset\D$ is closed polar, and we may replace \(\{M_n\} \) by  the sequence \(\{N(M_n)\} \) of its Nehari modifications. Write \( N(M_n) = h_n/b_n \), where $h_n\in \mathcal{A}(\D)$ and \( b_n \) is a Blaschke product vanishing at the poles of \( N(M_n) \) according to their multiplicities. That is, \( b_n=q_n/\widetilde q_n \), where  $q_n\in\mathcal{M}_n(\D)$ is a polynomial such that \( q_nM_n \in \mathcal{A}(\D) \) and $\widetilde{q}_n(z):=z^n\overline{q_n(1/\bar z)}$ is the reciprocal polynomial. We have that
\[
|f(z) - (h_n/b_n)(z)| = |(b_nf)(z)-h_n(z)|, \quad z\in \T,
\]
since Blaschke products are unimodular on the unit circle. Thus, \( h_n \) is in fact the best Nehari approximant of $b_nf$ in \( \mathcal{A}(\D) \). Recall that the above expressions converge to zero faster than geometrically to zero by the very choice of \( \{M_n\} \).

Our goal is to show that $h_n/b_n$  converge in logarithmic capacity to $f$ in \( \overline{\D}\setminus E \) at faster than geometric rate. That is, we fix a compact set $K\subset \overline{\D}$ disjoint from $E$ and we will prove that 
\begin{equation}
\label{convcapc}
\lim_{n\to\infty} \cp\big( \{z\in K:|e_n(z)|> a^n\} \big) = 0\quad \text{for any} \quad a>0,
\end{equation}
where \( e_n(z) := f(z)-(h_n/b_n)(z) \) is the approximation error. This will establish convergence in logarithmic capacity on compact subsets of \( \overline\D\setminus E \). Subsequently, as we pointed out in Section~\ref{ssec:conv_cap}, \eqref{convcapc} yields an analogous claim for the Greenian capacity on any compact subset of \( \D\setminus E \).

Since $\cp(E)=0$, it follows from \cite[Theorems 5.5.2 \& 5.5.4]{Ransford} that for each $\eta>0$ there is $k\in\N$ and $p_k\in\mathcal M_k(E)$ (we can take \( p_k \) to be the \( k \)-th Fekete polynomial for \( E \)) such that
\begin{equation}
\label{lemniscate}E\subset L_\eta:=\big\{\zeta\in\C: |p_k(\zeta)|<\eta^k\big\}.
\end{equation}
Pick $\eta<\min\{\mathrm{dist}(K,E)\,,\,1\}$ to be adjusted later. Of necessity $K\cap L_\eta=\varnothing$, because  $|p_k(z)|\geq\dist(K,E)^k$  for $z\in K$. Let  $\gamma\subset L_\eta$ be  a system of closed curves encompassing each point of $E$ exactly once, and such that $\dist(\gamma,E)<\dist(\T,E)/4$. Then, one has that
\begin{equation}
\label{Cauchyexp}
f(z) = \int_\gamma \frac{f(\xi)}{z-\xi}\frac{d\xi}{2\pi\ic}, \quad z\in\overline\C\setminus\overline{\mathrm{int}\,\gamma},
\end{equation}
by the Cauchy formula, where \( \mathrm{int}\,\gamma \) is the union of the bounded components of the complement of \( \gamma \).  Since $\mathbb{P}_-$ evaluated at $|z|>1$ coincides with the Cauchy projection having  kernel $(2\pi\ic(z-\zeta))^{-1}d\zeta$ on $\T$, the Hankel operator $\Gamma_{b_n f}$ acts on $v\in H^2$ by
\begin{equation}
\label{intCprojm}
\Gamma_{b_nf}(v)(z)=\int_\T \frac{b_n(\zeta)f(\zeta)v(\zeta)}{z-\zeta}\frac{d\zeta}{2\pi\ic}, \quad  |z|>1.
\end{equation}
Inserting  \eqref{Cauchyexp} into  \eqref{intCprojm} yields by Fubini's theorem and the Cauchy formula for $H^2$-functions that
\begin{equation}
\label{Hankint}
\Gamma_{b_nf}(v)(z) = \int_\gamma\frac{b_n(\xi)f(\xi)v(\xi)}{z-\xi}\frac{d\xi}{2\pi\ic}, \quad z\in \overline\C\setminus\overline{\mathrm{int}\,\gamma}.
\end{equation}
Observe that \eqref{Hankint}, initially proven for $|z|>1$, actually defines  an analytic extension to the exterior of $\gamma$ of \(\Gamma_{b_nf}(v)=\mathbb{P}_-(b_nfv) \in H_-^2\) (and therefore an extension to $\overline{\C}\setminus E$ since $\gamma$ could be taken arbitrary close to $E$).  Next, recall that a first singular vector of the Hankel operator $\Gamma_{b_nf}$ is an element $v_0\in H^2$ of unit norm that maximizes  \(\|\Gamma_{b_nf}(v)\|\) over all $v\in H^2$ with \(\|v\|_2=1\), and that it  always can be chosen to be outer\footnote{An outer function $w\in H^2$ is of the form \( w(z)=\alpha\exp\Big \{ \int_\T \frac{\xi+z}{\xi-z}  \log|w(\xi)| \frac{|d\xi|}{2\pi}\Big \}\), with $w_{\mathcal{b}\T}\in L^2(\T)$ and \( |\alpha|=1 \).}, see for instance \cite[p. 62]{BaSey}. Then, we get from \eqref{Hankint} and \eqref{forerreur} (applied with $n=0$ and $f$ replaced by $fb_n$)  that
\begin{equation}
\label{lerreurpf}
(e_nb_n)(z) = (b_nf-h_n)(z)=\frac{1}{v_0(z)}\int_\gamma \frac{b_n(\xi)f(\xi)v_0(\xi)}{z-\xi}\frac{d\xi}{2\pi\ic},\quad z\in\D\setminus\overline{\mathrm{int}\,\gamma}.
\end{equation}
Note that the right-hand side of \eqref{lerreurpf} is analytic in $\D\setminus\overline{\mathrm{int}\,\gamma}$, since $v_0$  is outer and thus has no zeros in $\D$.  Let \( B_k := p_k/\widetilde p_k \). Similarly to \eqref{Cauchyexp}--\eqref{lerreurpf}, it holds that
\begin{equation}
\label{Pminusalot}
\mathbb{P}_-\left(e_nb_nB_k^\ell \right)(z) =  \frac{1}{2\pi\ic}\int_\gamma \frac{b_n(\xi)f(\xi)B_k^\ell(\xi)}{z-\xi}d\xi,\quad z\in\overline{\D}\setminus\overline{\mathrm{int}\,\gamma},
\end{equation}
where $\ell\in\N$ is such that $\ell k\leq n<(\ell+1)k$ and the right-hand side again defines an analytic extension of the left-hand side into the exterior of $\gamma$. Indeed, we can express the left-hand side of \eqref{Pminusalot} for $|z|>1$ as the Cauchy integral of \( e_nb_nB_k^\ell \) on $\T$ like we did in \eqref{intCprojm} for $\mathbb{P}_-(b_nfv)$.  Since \(e_nb_n\) is analytic across $\T$ and $B_k\in \mathcal{H}(\overline{\D})$, we then deform the contour of integration into a circle  of radius slightly smaller than $1$, which can be done without changing the value of the integral by Cauchy's theorem. Subsequently, we insert \eqref{lerreurpf} in this integral and use Fubini's theorem and the residue formula as before to get \eqref{Pminusalot}.

Recall that  $\mathrm{dist}(\T,E)\geq\mathrm{dist}(K,E)>\eta$ by construction. Observe also that \( |p_k|\leq \eta^k \) on \( \gamma \) by \eqref{lemniscate} and that $|\widetilde{p}_k|\geq \mathrm{dist}(\T,E)^k$ in $\D$. Since $|b_n|\leq1$, we get from our choice of \( \ell \) that
\begin{equation}
\label{expmoinserr}
\left|\mathbb{P}_-\left( e_nb_nB_k^\ell \right)(z)\right| \leq \left( \frac\eta{\dist(\T,E)}\right)^{n-k} \frac{|\gamma| \|f\|_\gamma}{\dist(z,\gamma)}, \quad z\in\overline{\D}\setminus\overline{\mathrm{int}\,\gamma},
\end{equation}
where \(|\gamma|\) stands for the arclength of $\gamma$.  In another connection, it follows from \cite[Lemma]{PommerenkeNP}
      that for any $\varepsilon\in(0,1/3)$ there exists \( W_n\subset \D \) such that  $\cp(W_n)\leq 3\varepsilon$ and
\[
|p_k^\ell(\zeta)q_n(\zeta)|> \varepsilon^{n+\ell k} \|p_k^\ell q_n\|_\T, \quad \zeta\in\D\setminus W_n.
\]
As $|\widetilde{p}_k^\ell\widetilde{q}_n|(z)\leq \|\widetilde{p}_k^\ell\widetilde{q}_n\|_\T=\|p_k^\ell q_n\|_\T$  for $z\in\D$ by the maximum principle and the definition of the reciprocal polynomial, we get  that $|(b_nB_k^\ell)(z)|\geq \varepsilon^{n+k\ell}$ for $z\in\D\setminus W_n$.  Since \( \mathbb P_+ + \mathbb P_- \) is the identity operator, we can use the analytic continuation provided by \eqref{Pminusalot} to write
\begin{equation}
\label{decerrb}
e_n(z)=\frac{\mathbb{P}_+(e_nb_nB_k^\ell)(z)+\mathbb{P}_-(e_nb_nB_k^\ell)(z)}{b_n(z)B_k^\ell(z)}, \quad z\in  \overline{\D}\setminus\overline{\mathrm{int}\,\gamma}. 
 \end{equation}
The estimate $|(b_nB_k^\ell)(z)|\geq \varepsilon^{n+k\ell}$ for $z\in\D\setminus W_n$, \eqref{expmoinserr}, and the definition of $\ell$ give us that
\begin{equation}
\label{estPmerr}
 \left|\frac{\mathbb{P}_-(e_nb_nB_k^\ell)(z)}{b_n(z)B_k^\ell(z)}  \right| \leq \frac1{\varepsilon^{2n}} \left( \frac\eta{\dist(\T,E)}\right)^{n-k} \frac{|\gamma| \|f\|_\gamma}{\dist(z,\gamma)}, \quad z\in \overline{\D}\setminus (\overline{\mathrm{int}\,\gamma}\cup W_n).
 \end{equation}
Now, given \(0<\varepsilon<1\) and $0<a<1$, choose \( \eta \) in \eqref{lemniscate} so that $0<\eta<a\varepsilon^2\mathrm{dist}(\T,E)$. Choice of \( \eta \) of course fixes \( k \) in \eqref{lemniscate}. Then, since $K$ lies exterior to $\gamma$ because $K\cap L_\eta=\varnothing$, we get from \eqref{estPmerr} that there exists a natural number $n_0=n_0(f,K,\varepsilon,\eta,\gamma)$ for which
\begin{equation}
\label{estPmerrfg}
\left|\frac{\mathbb{P}_-(e_nb_nB_k^\ell)(z)}{b_n(z)B_k^\ell(z)}  \right| < a^n, \quad n\geq n_0, \quad z\in K\setminus W_n.
\end{equation}
Next, as $\|e_nb_nB_k^\ell\|_\T=\|e_n\|_\T\to0$ faster than geometrically with $n$ by hypothesis, \eqref{expmoinserr} and the triangle inequality yield that
\begin{equation}
\label{errPplus}
\left\|\mathbb{P}_+(e_nb_nB_k^\ell)\right\|_\T=   \left\|e_nb_nB_k^\ell-\mathbb{P}_-(e_nb_nB_k^\ell)\right\|_\T\leq C\left(\frac{\eta}{\mathrm{dist}(\T,E)}\right)^n, \quad n\geq n^\prime_0,
\end{equation}
for some  constant $C=C(f,E,\gamma)$ and some $n^\prime_0$ depending on $C$ and the speed of approximation of $f$ by $h_n/b_n$. Subsequently, as in \eqref{estPmerrfg}, we get from the estimate $|(b_nB_k^\ell)(z)|\geq \varepsilon^{n+k\ell}$ for $z\in\D\setminus W_n$, \eqref{errPplus}, and the maximum modulus principle that 
\begin{equation}
\label{estPperrfg}
 \left|\frac{\mathbb{P}_+(e_nb_nB_k^\ell)(z)}{b_n(z)B_k^\ell(z)}  \right| < a^n, \quad n\geq n_0^{\prime\prime}, \quad z\in K\setminus W_n,
\end{equation}
for some natural number \( n_0^{\prime\prime} = n_0^{\prime\prime}(f,K,\varepsilon,\eta,\gamma) \).  Because $\varepsilon$ and $a$ can be arbitrarily small and \( \cp(W_n) \leq 3\varepsilon \), \eqref{convcapc} now follows from \eqref{decerrb}, \eqref{estPmerrfg} and \eqref{estPperrfg}.

 Having proven  that \(M_n \overset{\mathrm{cap}}{\to} f \) in \( D \setminus E\) at faster than geometric rate whenever it is a sequence of $n$-th root optimal meromorphic approximants to $f$, we turn to the construction of rational functions $R_{k_n}\in\mathcal{R}_{k_n}(D)$ such that the poles of $R_{k_n}$ are among the poles of $M_n$ lying in $V$ and  \eqref{optimalEs} holds, where $V$ any open set such that $E\subset V\subset \overline{V}\subset D$.  Let $B\supset E$ be a closed set contained in $V$ which is regular for the Dirichlet problem, see Section~\ref{ssec_reg}. Such a $B$ is easily constructed as a sublevel set, for some small regular value,  of a smooth non-negative function whose zero set is $E$, see, for example, discussion after \eqref{exhaustion}. Then the Green equilibrium potential $G(z):=g(\mu_{D,B},D;z)$ is harmonic in $D\setminus B$, continuous on $D$, strictly less than the constant $1/\cp_D(B)$ on $D\setminus B$ and equal to that constant on $B$, see \eqref{GreenEqPot} or Section~\ref{ssec_cap} for a more detailed discussion. Since $\partial_zG(z)$ is holomorphic in $D\setminus B$, the critical points of $G$ are isolated and cannot accumulate in $D\setminus B$, so we can find an interval $[t_1, t_2]\subset(0,1/\cp_D(B))$ that is free of critical values and such that $G^{-1}([t_1,t_2])\subset V$. For any $t\in[t_1,t_2]$, $\gamma(t) := G^{-1}(t)$ is a 1-dimensional compact manifold, i.e., a finite union of disjoint real analytic closed curves $\gamma_{1,t},\ldots,\gamma_{N,t}$, none of which lies interior to another (by the maximum principle), and such that \( B\subset\mathrm{int}\,\gamma(t)\subset\overline{\mathrm{int}\,\gamma(t)}\subset V \). Note that $N$ is independent of $t\in[t_1,t_2]$ since any such $t$ is a regular value; note also that the total length $|\gamma(t)| = \sum_{j=1}^N|\gamma_{j,t}|$ is bounded above independently of \( t \), say by a constant $L$, because the gradient $\nabla G$ is normal to $\gamma_{j,t}$ at its every point and therefore the divergence formula implies for any $t\in[t_1,t_2]$:
\[
t|\gamma(t)|\min_{z\in G^{-1}(t)}\|\nabla G(z)\| \leq t_2|\gamma(t_2)|\max_{z\in G^{-1}(t_2)}\|\nabla G(z)\| -\int_{G^{-1}([t,t_2])}\|\nabla G\|^2dxdy.
\]
Pick $a>0$, set  $K:=\max\{\|\nabla G(z)\|:z\in G^{-1}([t_1,t_2])\}$ and let $n_a\in\N$ be so large that
\begin{equation}
\label{convcapct}
\cp\bigl( \{z\in G^{-1}([t_1,t_2]):|f(z)-M_n(z)|> a^n\} \bigr) <\frac{t_2-t_1}{4K},\quad n\geq n_a.
\end{equation}
Such a $n_a$ exists by the first part of the proof. Let \( A_n \) be the set whose capacity is estimated in \eqref{convcapct}. Assume for the moment that for each $t\in[t_1,t_2]$ there exists $z\in G^{-1}(t) \cap A_n$. Then the image of \( A_n \) under $G/K:G^{-1}([t_1,t_2])\to\R$ is equal to the interval \( [t_1/K,t_2/K] \) whose capacity is \( (t_2-t_1)/(4K) \). However, since contractive maps do not increase the logarithmic capacity \cite[Theorem 5.3.1]{Ransford}, the capacity of \( G(A_n)/K\) should be strictly smaller than \( (t_2-t_1)/(4K) \) by \eqref{convcapct}. Hence, for each $n\geq n_a$ there is $t_n\in[t_1,t_2]$ for which
\begin{equation}
\label{inegfgpart}
\left| \int_{\gamma(t_n)}\frac{f(\xi)-M_n(\xi)}{z-\xi} \frac{d\xi}{2\pi\ic} \right|\leq \frac{L a^n}{2\pi\mathrm{dist}(T,\overline{V})}, \quad z\in T.
\end{equation}
Pick a positive sequence $\{a_k\}$ converging to $0$ and, without loss of generality,  arrange things so that $n_{a_k}<n_{a_{k+1}}$.  Define
\[
J_n(z):= \int_{\gamma(t_{n_{a_k}})}\frac{M_n(\xi)}{z-\xi}\frac{d\xi}{2\pi\ic}, \quad z\in D\setminus \overline V, \quad n_{a_k} \leq n< n_{a_{k+1}}.
\]
Clearly, $J_n$ is a rational function retaining the singular part of $M_n$ inside the system of arcs \( \gamma(t_{n_{a_k}}) \), and it is of type $(k_n-1,k_n)$ where $k_n\leq n$ is the number of the poles of $M_n$ inside this system of arcs, counting multiplicities. If we put $R_{k_n}:=J_n$, then since  $a_k\to0$ we get from  \eqref{inegfgpart} and the Cauchy formula that \eqref{optimalEs} holds, as desired.

\appendix

\section{Potential Theory on a Riemann Surface}
\label{sec_app}

Even though the proof of Theorem~\ref{thm:main2} in Sections~\ref{ssec_nehari}--\ref{ssec_concl} was carried out for \( D=\D \), this appendix is written for a general Jordan domain \( D \) since specializing \( D \) to the unit disk would only shorten the proofs of Lemmas~\ref{lem:thin-sets} and~\ref{casparta} by a couple of paragraphs but otherwise would not lead to any further simplifications.

\subsection{Subharmonic Functions}
\label{ssec_sub}

Let $d$ be the differential and ${}^*$ the conjugation operators on a connected Riemann surface. The Laplacian $\Delta:=d{}^*d$ takes smooth functions  to $2$-forms.  If $U$ is an open subset of the surface and  $u:U\to\R$ a locally integrable function (against the area-form ${}^*{\bf 1}$, where ${\bf 1}$ is the constant unit function, or equivalently against the Lebesgue measure $(\ic/2)dz\wedge d\bar{z}$ in any system of local coordinates $z$, $\overline{z}$), the distributional Laplacian $\Delta u$ is the $0$-current acting on a smooth compactly supported functions $\varphi$ on $U$ by $\int u\Delta\varphi$. When $\Delta u=0$, one says that $u$ is harmonic on $U$, and such functions are in fact smooth (even real analytic) by Weyl's lemma \cite[Theorem~24.9]{Forster}. Subharmonic functions on $U$ are defined as upper-semicontinuous functions $u:U\to[-\infty,\infty)$ such that, if $V$ is open in $U$ and $h:V\to\R$ is harmonic, then $u-h$ is either  constant or fails to have a maximum in $V$. On open subsets of $\C$, this definition coincides with the usual one;  see \cite[Definition~2.2.1 \& Theorem~2.4.1]{Ransford}. A superharmonic function is the negative of a subharmonic function. A difference of two subharmonic functions is sometimes called a \( \delta \)-subharmonic function.

Harmonicity and subharmonicity are local properties:  $u$ is harmonic (resp. subharmonic) on $U$ if and only if its restriction to every open subset is, or equivalently if and only if $u\circ\varphi^{-1}$ is harmonic (resp. subharmonic) on the open set $\varphi(V\cap U)\subset\C$ whenever $(V,\varphi)$ is a local chart. Thus, standard facts regarding such functions on open subsets of a Riemann surface follow from their planar counterparts,  using charts. In particular,  the \emph{integrability theorem} \cite[Theorem~2.5.1]{Ransford} states that a subharmonic function which is not identically $-\infty$ is locally integrable, and therefore it has a  distributional Laplacian. Hence, two subharmonic functions that coincide almost everywhere (with respect to area measure) are in fact equal, for either they are both identically $-\infty$ or they have the same distributional Laplacian, and so their difference is harmonic; this is the \emph{weak identity principle}. The following is a variant of Harnack's theorem \cite[Theorem~1.3.10]{Ransford} and of \cite[Theorem~2.4.6]{Ransford}.

\begin{hst}
\label{hst}
A sequence of harmonic functions on $U$ that is bounded below has a subsequence that converges locally uniformly on $U$, either to $+\infty$ or to a harmonic function. For an increasing sequence, convergence holds along the full sequence. A decreasing sequence of subharmonic functions converges pointwise to a subharmonic function.
\end{hst}

A locally integrable function $u$ is subharmonic if and only if $\Delta u$ is a Radon measure on the surface; that is, $\Delta u$ is a positive linear form on continuous functions with compact  support. Indeed, as this statement is local, it reduces to its planar analog. The ``only if'' part  follows from \cite[Section~3.7]{Ransford}. As to the ``if'' part, let  $W\subset\C$ be open  and $\nu$ be a finite positive Borel measure carried by $W$. The logarithmic potential  of \( \nu \), i.e., $V^\nu(z):=\int\log|z-t|^{-1}d\nu(t)$, is superharmonic on $\C$ with distributional Laplacian  $-\nu$, so if $u$ is locally  integrable on $W$ with $\Delta u=\nu$ there, then $h:=u+V^\nu_{\mathcal bW}$ is harmonic and therefore $u=-V^\nu_{\mathcal bW}+h$ is subharmonic in $W$ (as in the main text, for a set $E$ (that may require further qualification), a subscript $\mathcal bE$ indicates ``restriction to $E$'').

\subsection{Green Functions}
\label{ssec_green}

Throughout, \( \Omega \) will be a subdomain of some ambient algebraic Riemann surface $\RS_*$  such that \( p(\Omega) \) is a bounded domain in \( \C \), where $p$ stands for the canonical projection; in particular the results apply to $\Omega=\RS$ with $p(\Omega)=D$, see beginning of Section~\ref{laclasse}. Since the lift to $\Omega$ of a positive non-constant superharmonic function  on \( p(\Omega) \) is again positive, non-constant, and superharmonic, \( \Omega \) is hyperbolic and as such possesses Green functions \cite[Theorem~IV.3.7]{FarkasKra}. Notice also that there exists a subdomain $\Omega^\prime\subset\RS_*$ with \(p(\Omega^\prime)\) bounded such that $\overline{\Omega}\subset\Omega^\prime$, where an overline (as in \( \overline\Omega \)) always denotes the closure in $\RS_*$.

Recall that the Green function for $\Omega$ with pole at $w$, denoted by $g_\Omega(\cdot,w)$, is the unique function that is harmonic and positive in $\Omega\setminus\{w\}$ with a logarithmic singularity at $w$ and whose largest harmonic minorant is identically zero. By a logarithmic singularity at $w$, it is meant that if $(V,\varphi)$ is a coordinate chart on $\Omega$ such that $w\in V$ and $\varphi(w)=0$, then $g_\Omega(\varphi^{-1}(\cdot),w)+\log|\cdot|$ is harmonic on $\varphi(V)$. Obviously $g_\Omega(\cdot,w)>0$ everywhere on $\Omega$, by the minimum principle for harmonic functions. Recall also that if a superharmonic function on $\Omega$ is not identically $+\infty$ and has a harmonic minorant, then it has the largest one whose construction can be carried out as in the Euclidean case \cite[Theorem~4.3.5]{Helms}, because Poisson modifications can be performed locally.

Clearly, $g_\Omega(\cdot,w)$ is superharmonic and $\Delta g_\Omega(\cdot,w)=-\delta_w$, where $\delta_w$ is the Dirac mass at $w$. Moreover, $g_\Omega$ is symmetric in that $g_\Omega(z,w)=g_\Omega(w,z)$ \cite[Theorem~IV.3.10]{FarkasKra}.  Symmetry entails that $g_\Omega(z,w)$ is separately harmonic in $z$ and $w$,  and therefore jointly harmonic on $\{(z,w)\in\Omega\times\Omega:\,z\neq w\}$ \cite[p. 561]{Le61}; in particular, $(z,w)\mapsto g_\Omega(z,w)$ is continuous off the diagonal. Note that
\begin{equation}
\label{GreenDom}
g_{\Omega}(z,w) \leq g_{\Omega^\prime}(z,w), \quad z,w\in\Omega\subseteq\Omega^\prime,
\end{equation}
because $g_{\Omega}(\cdot,w)- g_{\Omega^\prime}(\cdot,w)$  is a harmonic minorant of $g_{\Omega}(\cdot,w)$ and therefore must be non-positive. Thus, if $F_1$, $F_2$ are relatively closed subsets of $\Omega$ with $\overline{F}_1\cap\overline{F}_2=\varnothing$, we deduce that $(z,w)\mapsto g_\Omega(z,w)$ is  bounded on $F_1\times F_2$ because $g_{\Omega^\prime}(z,w)$ is continuous on the compact set $\overline{F}_1\times \overline{F}_2$ whenever  $\overline{\Omega}\subset\Omega^\prime$. We also remark that to each $w\in\Omega$ there is an open set $V\ni w$ and a constant $C=C(V)$ such that
\begin{equation}
\label{feL}
\int_{V} g_\Omega(z,w')\,{}^*{\bf 1}(z)<C,\qquad w'\in V,
\end{equation}
a result that follows by uniformization from the corresponding fact on the disk \cite[Theorem~4.4.12]{Helms}. Moreover, if a sequence of open sets $\Omega_n$ increases to $\Omega$ as $n\to\infty$, it follows from \eqref{GreenDom} that \( g_{\Omega}(\cdot,w)_{\mathcal b\Omega_n} - g_{\Omega_n}(\cdot,w) \) is a decreasing sequence of positive harmonic functions that must converge locally uniformly in \( \Omega \), by Harnack's theorem; as the limit is necessarily a non-negative harmonic minorant of \( g_{\Omega}(\cdot,w) \), it must be identically zero.

\subsection{Green Potentials}
\label{ssec_pot}

A Green potential in \(\Omega\) is a non-negative superharmonic function whose largest harmonic minorant is identically zero. Given \( \sigma \), a Radon measure in \( \Omega \), let us put
\begin{equation}
\label{GreenPot}
g(\sigma,\Omega;z) := \int g_\Omega(z,w)d\sigma(w).
\end{equation}
This is a superharmonic function of  $z\in\Omega$ which is either identically $+\infty$, or locally integrable with distributional Laplacian $-\sigma$ by Fubini's theorem. If \( g(\sigma,\Omega;\cdot) \) is not identically $+\infty$, using the monotone convergence and Fubini's theorem, the proof of \cite[Lemma~4.3.6]{Helms} carries over to integrals instead of sums to show  that the largest harmonic minorant of  $g(\sigma,\Omega;\cdot)$ is the integral against $d\sigma(w)$ of the largest harmonic minorants of the $g_\Omega(\cdot,w)$, namely zero. Thus,  \( g(\sigma,\Omega;\cdot) \) is a Green potential. Conversely, it follows from the \hyperref[rrt]{Riesz representation theorem} stated below that every Green potential has the form \eqref{GreenPot}. Notice that if $\sigma(\Omega)<\infty$, then $g(\sigma,\Omega;\cdot)\not\equiv+\infty$, for if $V\subset \Omega$ is as in \eqref{feL} and $W$ is a nonempty open set with  $\overline{W}\subset V$, then
\[
\int_W g(\sigma,\Omega;z){}^*{\bf 1}(z)<C\sigma(V) +C^\prime\sigma(\Omega\setminus V) \int_W{}^*{\bf 1}
\]
by Fubini's theorem, where $C^\prime$ is an upper bound for \(g_\Omega(z,w)\) on $(\Omega\setminus V)\times \overline{W}$. 

\begin{rrt}
\label{rrt}
Let $u\not\equiv+\infty$ be a superharmonic function on $\Omega$ that has a harmonic minorant. Then $u=g(\sigma,\Omega;\cdot)+h$\textcolor{blue}{,} where $h$ is the largest harmonic minorant of \( u \) and \( \sigma:=-\Delta u \).
\end{rrt}
\begin{proof}
Assume first that $\sigma(\Omega)<\infty$. Then, $g(\sigma,\Omega;\cdot)\not\equiv+\infty$ and therefore $h:=u-g(\sigma,\Omega;\cdot)$ is  harmonic on $\Omega$. Clearly, $h$ is a minorant of $u$. Since the largest harmonic minorant of \( g(\sigma,\Omega;\cdot) \) is zero,  $h$ is the largest harmonic minorant of $u$. If $\sigma(\Omega)=\infty$, pick $\Omega_m$  to be an increasing exhaustion of $\Omega$  by relatively compact open  sets. Put $\sigma_m:=\sigma_{\mathcal{b}\Omega_m}$, which are finite measures on $\Omega_m$ because $\sigma$ is a Radon measure. By what precedes, $u_{\mathcal b\Omega_m}=g(\sigma_m,\Omega_m;\cdot)+h_m$, where $h_m$ is the largest harmonic minorant of $u_{\mathcal b\Omega_m}$. As the functions $g_{\Omega_m}(\cdot,w)$ increase locally uniformly to $g_{\Omega}(\cdot,w)$ while $h_m$ decrease and are bounded below by any harmonic minorant of  $u$, we get by monotone convergence and \hyperref[hst]{Harnack's theorem} that $u=g(\sigma,\Omega;\cdot)+h$, where $h$ is harmonic and necessarily $g(\sigma,\Omega;\cdot)\not\equiv+\infty$. We now conclude the proof as in the first case.
\end{proof}

This version of the Riesz representation theorem featuring  the weak-Laplacian  may be compared to the more  abstract  formulation  for Green spaces (of which $\Omega$ is a special case) in \cite[Section~VI.7]{Brelot}, that  does not refer to the Laplacian; see also the planar statement of \cite[Theorem~4.5.4]{Ransford}. 

The previous considerations allow us to simplify in our case the notion of admissibility of a measure given in \cite[Section~1]{Fug71} and \cite[Section~I.3]{Fuglede}. According to that definition, a measure \( \sigma \) is \emph{admissible} if it is integrable against continuous Green potentials with compactly supported Laplacian. Fubini's theorem immediately implies that $g(\sigma,\Omega;\cdot)\not\equiv+\infty$ if \( \sigma \) is admissible. In the present Greenian context the condition $g(\sigma,\Omega;\cdot)\not\equiv+\infty$ is also sufficient for (and therefore equivalent to) admissibility of \( \sigma \). Indeed, if \( \nu \) is compactly supported in \( \Omega \) with continuous potential, let \( V \) be an open set such that \( \supp\,\nu\subset V\subset \overline V\subset \Omega \). By continuity, there exists \( C \) such that \( g(\nu,\Omega;z) \leq C \), \( z\in \Omega \). Since $g(\sigma,\Omega;\cdot)\not\equiv+\infty$, the same is true for the potential of $\sigma_{\mathcal b\Omega\setminus V}$. As this potential is harmonic in \( V \) and is not equal identically to \( +\infty \) there, it is finite and locally bounded in \( V \). Hence, \( g(\sigma_{\mathcal b\Omega\setminus V},\Omega;z) \leq C^\prime \), \( z\in \supp\,\nu \). Therefore, it follows from Fubini's theorem that 
\[
\int g(\nu,\Omega;z)d\sigma(z) \leq C\sigma(V) + C^\prime \nu(\Omega).
\]

\subsection{Capacities}
\label{ssec_cap}

Given two Radon measures $\sigma_1$ and $\sigma_2$ on $\Omega$, we put 
\[
(\sigma_1,\sigma_2)_\Omega:=\int g(\sigma_1,\Omega;z)d\sigma_2(z)=\int g(\sigma_2,\Omega;z)d\sigma_1(z),
\]
which is either a non-negative number or $+\infty$. The \emph{Green energy} of $\sigma$ is defined as $I_\Omega(\sigma):=(\sigma,\sigma)_\Omega$. The \emph{Greenian capacity} relative to $\Omega$  of a compact set $K\subset\Omega$ is a non-negative number
\begin{equation}
\label{defGcap}
\cp_\Omega(K):=\frac{1}{\inf_{\mu\in\mathcal{P}(K)} I_\Omega(\mu)},
\end{equation} 
where $\mathcal{P}(K)$ is the set of Borel probability measures on $K$. The \emph{Greenian capacity} of a Borel set $B$ is given by
\begin{equation}
\label{GreenCap}
\cp_\Omega(B) := \sup_{K\subset B}\cp_\Omega(K)=\inf_{U\supset B} \cp_\Omega(U) ,
\end{equation}
where the supremum is taken over all compact subsets of $B$, the infimum is taken over all open sets containing \( B \), and the equality is due to a theorem by Choquet \cite[Section~VIII.4]{Brelot}. When $K$ is compact and $\cp_\Omega(K)>0$, there exists a unique $\mu_{\Omega,K}\in\mathcal{P}(K)$, called the \emph{Green equilibrium measure} of $K$ in $\Omega$, to meet the infimum in \eqref{defGcap}. It is characterized by the fact that for some constant $C$($=1/\cp_\Omega(K)$), the \emph{Green equilibrium potential} $g(\mu_{\Omega,K},\Omega;z)$ satisfies $g(\mu_{\Omega,K},\Omega;z)\leq C$ for $z\in \Omega$ with $g(\mu_{\Omega,K},\Omega;z)=C$ for $z\in K\setminus E$,  where $E$ has Greenian capacity zero; this can be shown as in the Euclidean case \cite[Theorems~II.5.11 \&~II.5.12]{SaffTotik}. 

In the case of an arbitrary set \( B \), the infimum in \eqref{GreenCap} introduces the \emph{outer Greenian capacity} of $B$ and will serve as a definition of \(\cp_\Omega(B)\). However, it  may no longer match   the supremum (the latter defines the \emph{inner Greenian capacity} of $B$).

When $\Omega\subset\C$,  another notion of capacity is instrumental in this paper, namely the \emph{logarithmic capacity}  defined for a compact set $K\subset\C$ as
\begin{equation}
\label{deflogcap}
\cp(K):=\exp\left\{-\inf_{\mu\in\mathcal{P}(K)} \int V^\mu(z)d\mu(z)\right\},
\end{equation}
where $V^\mu(z)$ is the logarithmic potential of $\mu$ defined earlier in Section \ref{ssec_sub}. The logarithmic capacity of a Borel subset of $\C$ and the outer logarithmic capacity of an arbitrary subset are defined via the same process  as for Greenian capacity, based on the analog of \eqref{GreenCap}, see \cite{Ransford, SaffTotik} (note that in \cite{Ransford}, $\cp(E)$ denotes the inner logarithmic capacity and potentials carry a sign opposite to the current one). If $K$ is compact and $\cp(K)>0$, then  there is a unique $\mu_K\in\mathcal{P}(K)$, called the \emph{logarithmic equilibrium measure} of $K$, that realizes the infimum in \eqref{deflogcap}. It is characterized by the fact that for some constant $C$($=-\log\cp(K)$), the \emph{logarithmic equilibrium potential} $V^{\mu_K}(z)$ is at most $C$ for $z\in K$ and in fact equal to $C$ on $K$  except possibly for a subset of logarithmic capacity zero.  

Both the Greenian and logarithmic capacities are  right continuous on compact sets, meaning that
\[
\cp (\cap_{j=1}^\infty K_j) = \lim_n \cp (\cap_{j=1}^n K_j) \qandq \cp_\Omega (\cap_{j=1}^\infty K_j) = \lim_n \cp_\Omega (\cap_{j=1}^n K_j)
\]
if the $K_j$ are compact; see \cite[Theorem~5.1.3(a)]{Ransford} for the logarithmic case, the Greenian one being argued the same way with an obvious adaptation of \cite[Lemma~3.3.3]{Ransford}.  In addition, the (outer) Greenian and logarithmic capacities are  left continuous:
\[
\cp (\cup_{j=1}^\infty E_j) = \lim_n \cp (\cup_{j=1}^n E_j) \qandq \cp_\Omega (\cup_{j=1}^\infty E_j) = \lim_n \cp_\Omega (\cup_{j=1}^n E_j);
\]
for the logarithmic capacity this follows from \cite[Theorem 5.1.3(b)]{Ransford} combined with Choquet's theorem, and the Greenian case can be handled similarly, compare to \cite[Section~VIII.4]{Brelot}.

One form of the \emph{domination principle} for Green potentials says that if $g(\sigma,\Omega;\cdot)\leq v$ on $\supp\,\sigma$ (the support of $\sigma$) for some superharmonic function $v$, then $g(\sigma,\Omega;\cdot)\leq v$ everywhere on $\Omega$; in fact, we shall state a stronger version in Section~\ref{ssec_thin}. It implies the \emph{continuity theorem}, saying that if  the restriction of  $g(\sigma,\Omega;\cdot)$ to  $\supp\,\sigma$ is continuous at $z_0\in  \supp\,\sigma$ then $g(\sigma,\Omega;\cdot)$  is continuous at $z_0$. When $\sigma$ is a positive Borel measure  with compact support  such that $g(\sigma,\Omega;z)<+\infty$ for $\sigma$-a.e. $z$, there is an increasing  sequence  of measures $\sigma_k$ supported on  $\supp\,\sigma$, having continuous Green potentials and converging to $\sigma$ in the strong (total variation) sense, such that $g(\sigma_k,\Omega;\cdot)$ increases pointwise to $g(\sigma,\Omega;\cdot)$ on $\Omega$. The proof is {\it mutatis mutandis} the same as for logarithmic potentials \cite[Lemma~I.6.10]{SaffTotik}, using the continuity theorem for Green potentials. In particular, if $K$ is compact with $\cp_\Omega(K)>0$, we find upon letting  $\sigma$  be the Green equilibrium distribution that there exist nonzero positive  measures supported on $K$ whose Green potentials are continuous.

When $\Omega\subset\C$, a subset of $\Omega$ has (outer) Greenian capacity zero  if and only if it has (outer) logarithmic capacity zero. Indeed, it is enough to verify this claim on compact sets since capacity is left continuous and a set of outer (Greenian or logarithmic) capacity zero is contained in a Borel (even $G_\delta$) set of capacity zero, by definition. Moreover, by the increasing character of $g_\Omega(z,w)$ with $\Omega$, we may assume that $\Omega$ is simply connected. Then the result follows by comparing the logarithmic kernel $\log(1/|z-w|)$ with the Green kernel $g_\Omega(z,w)=\log|(1-\varphi(z)\overline{\varphi(w)})/(\varphi(z)-\varphi(w))|$, where $\varphi$ is a conformal map $\Omega\to\D$. A property holding pointwise  except on a set of outer Greenian capacity zero (equivalently: logarithmic capacity zero if $\Omega\subset\C$) is said to hold \emph{quasi everywhere}.

\subsection{Fine Topology}
\label{ssec_fine}

A  basis for the fine topology on $\Omega$ is given by all sets of the form
\begin{equation}
\label{defft}
\cap_{i=1}^m \big\{z\in B:v_i(z)<\alpha_i\big\},
\end{equation}
where $B\subseteq\Omega$ is open,  $v_i$ are superharmonic functions on $B$, and $\alpha_i$ are  constants. Consequently, all superharmonic functions $\Omega\to(-\infty,+\infty]$ are finely continuous (equivalently: all subharmonic functions $\Omega\to[-\infty,+\infty)$ are finely continuous), and the fine topology is the coarsest with this property because, by the  \hyperref[rrt]{Riesz representation theorem} and the monotonicity of Green functions with respect to the domain, each set of the form \eqref{defft} contains one for which $v_i$ are Green potentials. In particular, we may as well require in \eqref{defft} that $v_i$ be defined and superharmonic on the whole of $\Omega$. Hence, the present definition modeled after \cite[Definition~6.5.1]{Helms} (which deals with the Euclidean case) is equivalent to \cite[Definition~I.1]{Brelot}. It is known that the fine topology on $\Omega$ is locally connected \cite[Corollary to Theorem~9.11]{Fuglede},  and that the fine connected  components of a finely open set are finely open  \cite[Corollary~1]{Fug71}. Moreover,  Lipschitz curves are finely connected \cite[Theorem~7]{Fug71}, so that Euclidean domains are fine domains as well.

As a general convention, we use the prefix ``fine'' to signify that a notion is understood with respect to the fine topology. This way we distinguish the latter from the classical, Euclidean topology (more precisely: the one induced on $\Omega$ by the Euclidean topology of charts). The fine boundary of a set $S$ is denoted by $\partial_\mathsf{f}S$, and its fine closure by $\clos_\mathsf{f}$.

A set $E\subset \Omega$ is called \emph{polar} (in $\Omega$) if there is a superharmonic function $u\not\equiv+\infty$ on $\Omega$ such that $u(z)=+\infty$ for  $z\in E$. Superharmonic functions not identically $+\infty$ are locally integrable, therefore  polar sets have area  measure zero, see Section~\ref{ssec_sub}. By definition a polar set is contained in a $G_\delta$ polar set, and every $G_\delta$ polar set arises as the $+\infty$-set of a superharmonic function \cite[Section~VI.9]{Brelot}.  If $U$ is a fine domain and $E$ is polar, then $U\setminus E$ is again a fine domain \cite[Theorem~6]{Fug71}. In fact, polar sets are exactly the sets of zero outer Greenian  capacity (equivalently: of zero outer logarithmic capacity if $\Omega\subset\C$) defined in Section~\ref{ssec_cap}; this  is justified in Section \ref{ssec_bal}, but we take it presently for granted (we stress that \cite{Ransford} defines polar sets as having inner capacity zero, thereby making for  a larger class of non-Borel polar sets). One consequence is: if $\Omega^\prime\supset\Omega$ is hyperbolic and $E$ is polar in $\Omega$, then it is polar in $\Omega^\prime$ as well; indeed, since $g_{\Omega^\prime}(\cdot,w)\geq g_{\Omega}(\cdot,w)$, it is clear that $E$ has zero outer Greenian capacity  in $\Omega^\prime$ if it does in $\Omega$. Conversely, it is obvious from the definition  that $E$ is polar in $\Omega$ if it is polar in $\Omega^\prime \supset \Omega$, therefore we may speak of a polar set without specifying a hyperbolic subset of $\RS_*$ in which \( E \) is contained.

A countable union of polar sets is polar, for if $E_k$ is included in the $+\infty$-set of a superharmonic function $u_k\not\equiv+\infty$ while $K\subset\Omega$ is compact and of positive Lebesgue measure, then there are $t_k>0$ such that $u:=\sum_kt_k u_k$ is summable on $K$ (therefore $u\not\equiv+\infty$) and is superharmonic with value $+\infty$ at each point of $\cup_k E_k$. So, if $E\subset\Omega$ is polar and $p:\RS_*\to\C$ is the natural projection, then $p(E)$ is polar. Indeed, for $V\subset \Omega$ a domain such that $p:V\to p(V)$ is a homeomorphism, $\nu\circ p^{-1}$ is superharmonic on $p(V)$ when $\nu$ is superharmonic on $V$, and $\Omega\setminus \mathbf{rp}(\RS_*)$ can be covered with countably many such domains while $\mathbf{rp}(\RS_*)$ is finite. Conversely, if $V\subset\C$ is a bounded open set and $E\subset V$ is polar, then $p^{-1}(E)$ is polar because $v\circ p$ is superharmonic as soon as $v$ is superharmonic on $V$.

If $u\not\equiv+\infty$ is superharmonic and finite on a polar set $E$, then $E$ has outer $\Delta u$-measure zero \cite[Section~VI.9, item~$\beta)$]{Brelot}.  Consequently a Radon measure $\sigma$ of finite Green energy cannot charge a polar set $E$, for we may assume $\sigma$ has compact support (since $\Omega$ is $\sigma$-compact) and as in  Section~\ref{ssec_cap} there is an increasing sequence of measures $\sigma_k$ converging strongly to $\sigma$  with  $g(\sigma_k,\Omega;\cdot)$ continuous, whence $\sigma(E)=\lim_k\sigma_k(E)=0$.

\begin{rt}
\label{rt}
If $E\subset\Omega$ is a (relatively) closed polar set  while $u$ is superharmonic on $\Omega\setminus E$ and locally bounded below in a neighborhood of $E$, then $u$ extends in a unique manner to a superharmonic function on $\Omega$. Moreover, if $u$ is harmonic in $\Omega\setminus E$ and locally bounded in a neighborhood of $E$, then $u$ extends harmonically to $\Omega$.
\end{rt}
\begin{proof}
  The proof of the first statement carries over to hyperbolic Riemann surfaces from its planar version, see \cite[Theorem~3.6.1]{Ransford}. When $u$ is harmonic in $\Omega\setminus E$ and locally bounded in a neighborhood of $E$, it extends both to a subharmonic and a superharmonic function on $\Omega$ by the first part. Since $\Delta u$ does not depend on the extension because $E$ has Lebesgue measure zero, we deduce that $\Delta u=0$.
\end{proof}

The removability  theorem implies the following result.

\begin{gmp}
If $u$ is superharmonic and bounded below on some open set $U\subset\overline{U}\subset\Omega$, and if moreover $\liminf_{U\ni z\to\xi}u(z)\geq 0$ for quasi every $\xi\in\partial U$, then $u\geq 0$ in $U$. 
\end{gmp}
\begin{proof}
  Given $\varepsilon>0$, let $E_\varepsilon:=\{\zeta\in\partial U:\, \liminf_{z\to\xi} u(z)\leq-\varepsilon\}$. Then $E_\varepsilon$ is a closed polar set and the function $w:\Omega\to(-\infty,+\infty]$ given by $\min(u,-\varepsilon)$ on $U$ and $-\varepsilon$ on $\Omega\setminus(U\cup E_\varepsilon)$ is superharmonic on $\Omega\setminus E_\varepsilon$, by the glueing theorem, see \cite[Theorem~2.4.5]{Ransford} for a planar version of this local result. As $w$ is bounded below, it extends to a superharmonic function on $\Omega$. Because \( U\cup E_\varepsilon \subset \overline U \subset \Omega \), it holds that \( \liminf_{z\to\xi}w(z) = -\varepsilon \) for any \( \xi\in\partial\Omega \). Thus, $w\geq -\varepsilon$ in \( \Omega \) by the classical minimum principle contained in the very definition of superharmonic functions. 
\end{proof}

The hypothesis ``$u$ is bounded below'' can be relaxed somewhat: it is enough to assume that $u\geq -g$ where $g$ is a \emph{semi-bounded potential}, meaning that it is  the increasing pointwise limit of a sequence of locally bounded potentials, see \cite[Theorem~9.1]{Fuglede}. Note that $g(\sigma,\Omega;\cdot)$ is semi-bounded when it is finite $\sigma$-a.e.,  for we may assume $\sigma$ has compact support (as $\Omega$ is a countable union of compact sets) and then appeal to properties of the measures $\sigma_k$ in Section~\ref{ssec_cap}. In fact,  $g(\sigma,\Omega;\cdot)$ is semi-bounded if and only if it is finite $\sigma$-a.e. in \( \Omega \), which is also if and only if $\sigma$ does not charge polar sets, see  \cite[Section~I.2.6, Theorem]{Fuglede}.

\subsection{Thinness}
\label{ssec_thin}

Fine topology can also be introduced via the notion of thinness. A set $E\subset\Omega$ is said to be \emph{thin} at $\zeta\in\Omega$ if \( \zeta \) is not a fine limit point of \( E \). Equivalently, \( E \) is thin at \( \zeta \) if and only if either $\zeta\notin\overline{E}$ or there exists a function $v$, superharmonic in a neighborhood of $\zeta$, such that
\begin{equation}
\label{thin}
\liminf_{E\ni z\to \zeta,z\not=\zeta}v(z)> v(\zeta);
\end{equation}
see \cite[Theorem~6.6.3]{Helms} for a proof of this equivalence in the Euclidean setting, which applies to hyperbolic Riemann surfaces as well and also shows that $v$ in \eqref{thin} may be taken superharmonic on the whole of $\Omega$. Hence, the above definition of thinness (which is local) matches \cite[Definition~I.2]{Brelot} (whose local character is not immediate, see \cite[Theorem~VII.1]{Brelot}).  Setting $\liminf$ over the empty set to $+\infty$ by convention, \eqref{thin} may still be regarded as characterizing thinness at $\zeta\notin\overline{E}$, upon letting $v\equiv0$. Note that when the limit inferior in \eqref{thin} is taken over a full Euclidean neighborhood of $\zeta$, superharmonicity of $v$ implies that the inequality gets replaced by an equality. Clearly,  \( E \) is thin at \( \zeta \) if and only if for some (hence any) chart $(V,\varphi)$ with $\zeta\in V$, the planar set $\varphi(V\cap E)$ is thin at \(\varphi(\zeta)\), and a countable union of thin sets at $\zeta$ is again thin at $\zeta$. 
    
A set $V$ is a fine neighborhood of $\zeta\in V$ if and only if the complement of $V$ is thin at $\zeta$, see \cite[Theorem~I.3]{Brelot}. In particular, if \( V \) is finely open and \( Z \) is polar, then \( V\setminus Z \) is finely open. The points of \( E \subset\Omega\) at which $E$ is thin form a polar set, and $E$ is thin at each of its points if and only if it is polar \cite[Theorem~VII.7 \& Corollary]{Brelot}. One consequence of $E$ being thin at $\zeta$ is that, locally in a chart $(V,\varphi)$ with $\zeta\in V$, there are arbitrary small circles centered at $\varphi(\zeta)$ which do not meet $\varphi(E \cap V)$, see \cite[Theorem~6.7.9]{Helms}; in the same vein,  $\varphi(V\setminus E)$ contains a  segment  of the form \([\varphi(\zeta),\varphi(\zeta)+r e^{i\theta}]\) with $r=r(\theta)>0$ for quasi-every direction $e^{i\theta}$  in $\T$, see \cite[Theorem 5.4.3]{Ransford}. In particular, a connected set cannot be thin at an accumulation point and therefore polar sets are totally disconnected.

The \emph{base} $b(E)$ of a set $E$ is the set of points in $\Omega$ at which $E$ is non-thin, and $E$ is called a base if $b(E)=E$. It is known that $b(E)$ is a finely closed $G_\delta$ set, see \cite[Proposition~VII.8]{Brelot}. We record the following, elementary fact.

\begin{lemma}
\label{distregt}
Let $E,F\subset\Omega$ be disjoint  finely open sets such that $\Omega\setminus(E\cup F)$ is a base. Then $\Omega\setminus E$ and $\Omega\setminus F$ are bases as well.
\end{lemma}
\begin{proof}
  If $x\in (\Omega\setminus E)\cap(\Omega\setminus F)=\Omega\setminus(E\cup F)$, the latter set is non-thin at $x$ by assumption and therefore so is $\Omega\setminus E$. If now $x\in (\Omega\setminus E)\cap F=F$, then the fine openness of $F$ implies that $\Omega\setminus F$ is thin at $x$ and so is \( E\subset \Omega\setminus F \). However, if  $\Omega\setminus E$ were also thin at $x$ in this case,  then $\Omega= E\cup (\Omega\setminus E)$ would be  thin at $x$ which is impossible. Hence, $\Omega\setminus E$ is non-thin at any of its points, therefore it is a base.
\end{proof}

The notion of a base generates a strong form of the domination principle, see \cite[Theorem~VIII.4]{Brelot}.

\begin{sdp}
\label{sdp}
Let \( v \) be a non-negative superharmonic function in \( \Omega \) such that \( v \geq g(\sigma,\Omega;\cdot) \) quasi everywhere on a set \( E \) such that \( \sigma(\Omega\setminus b(E)) = 0 \).  Then \( v \geq g(\sigma,\Omega;\cdot) \) everywhere in \( \Omega \).
\end{sdp}

The fine closure of $E$ is equal to \( \clos_\mathsf{f}(E)=b(E)\cup i(E) \), where $i(E)$ is the set of finely isolated points of $E$, see \cite[Proposition~V.10]{Brelot}. The finely closed sets are precisely those for which $b(E)\subseteq E$. Note that $b(E)\cap i(E)=\varnothing$ and therefore, if $V$ is finely open, we have that
\begin{equation}
\label{Vprime}
b(\Omega\setminus V)=\Omega\setminus V^\prime, \quad V^\prime:=V\cup i(\Omega\setminus V),
\end{equation}
where we observe that $V^\prime$ is in turn  finely open. For any set $E$, the fine boundary $\partial_\mathsf{f}E$ is finely closed and, as shown in \cite[Lemma~12.3]{Fuglede}, it holds that
\begin{equation}
\label{expfb}
i(\partial_\mathsf{f}E)=i(E)\cup i(\Omega\setminus E) \qandq b(\partial_\mathsf{f}E)=b(E)\cap b(\Omega\setminus E).
\end{equation}

The next lemma connects fine topologies in \( D \) and \( \RS \) (defined at the beginning of Section \ref{laclasse}).

\begin{lemma}
\label{lem:fineopen}
The map $p:\RS\to D$ is finely open and finely continuous, that is, \( p(V) \) and \( p^{-1}(U) \) are finely open when \( V\subset\RS \) and \( U\subset D \) are finely open. Moreover, \( i(\RS\setminus p^{-1}(U)) = p^{-1}(i(D\setminus U)) \).
\end{lemma}
\begin{proof}
Let \( V\subset\RS \) be finely open and $\zeta\in V$. Denote by $\tilde O$ a Euclidean disk centered at $p(\zeta)$ of small enough radius so that \( O \), the connected component of $p^{-1}(\tilde O)$ containing $\zeta$, contains no ramification points of \( \RS \) except possibly \( \zeta \) itself (if \( m(\zeta)>1 \)) and \( O\cap p^{-1}(p(\zeta))=\{\zeta\} \). Since $E:=O\setminus V$ is thin at $\zeta$, there exists a superharmonic function \( v \) in \( O \) for which \eqref{thin} takes place. \emph{We claim} that
\begin{equation}
\label{defetagev}
\tilde v(z) := 
\begin{cases}
\sum_{w\in O\cap p^{-1}(z)}v(w), & z\in \tilde O\setminus\{p(\zeta)\}, \medskip \\
m(\zeta)v(\zeta), & z=p(\zeta),
\end{cases}
\end{equation}
is superharmonic on \( \tilde O \). Indeed, by shrinking $\tilde O$ if needed, we may assume that $v$ is the increasing limit of a sequence of continuous superharmonic functions $v_n$ on $O$, see \cite[Corollary~2.7.3]{Ransford} for a proof of this fact in the planar case that carries over to $\RS$ using local charts. Define \( \tilde v_n \) similarly to \eqref{defetagev}, only replacing $v$ with $v_n$. Clearly, \( \tilde v_n \) is superharmonic on $\{z\in\tilde{O}:z\neq p(\zeta)\}$. Since it is bounded around $p(\zeta)$ by the continuity of $\nu_n$, the restriction $\tilde v_{n\mathcal{b}\tilde{O}\setminus p(\zeta)}$ uniquely extends to a superharmonic function on $\tilde O$ by the \hyperref[rt]{Removability Theorem}. Of necessity, the value at $p(\zeta)$ of this extension is given by
\[
 \liminf_{z\to p(\zeta),z\neq p(\zeta)} \tilde v_n(z) = m(\zeta)v_n(\zeta)=\tilde v_n(p(\zeta)), 
\]
where the first equality comes from the continuity of \( v_n \) at $\zeta$. Hence, \(\tilde v_n(z) \) is superharmonic on $\tilde O$ and so is its increasing limit $\tilde v$. \emph{This proves the claim}. In another connection, the lower semicontinuity of $v$ shows that the analog of \eqref{thin} holds for \( \tilde v \) when the limit inferior is taken along $p(E)$. As $\tilde O\setminus p(V)\subseteq p(E)$, we get that $D\setminus p(V)$ is thin at $p(\zeta)$ so that $p(V)$ is finely open, as claimed. 

To show the second claim, observe that the lift of a function from \( D \) to \( \RS \) preserves superharmonicity and that \( \RS\setminus p^{-1}(U) = p^{-1}( D\setminus U ) \). The identity \( i(\RS\setminus p^{-1}(U)) = p^{-1}(i(D\setminus U)) \)  is now straightforward.
\end{proof}

In Lemma \ref{tdG} below, we single out for easy reference a basic geometric fact, used at several places in the paper. We say that a continuous injective map $\gamma:\T\to\RS$ is a \emph{parametrized Jordan curve} in $\RS$ and we simply call the image $\gamma(\T)$ a (non-parametrized) Jordan curve. On a hyperbolic Riemann surface $\RS$, any Jordan curve $\Gamma$ homotopic to a point is uniquely the boundary of a (topological) disk $O\subset\RS$ \cite[Theorem 2.4]{Strebel}; we say that $O$ is the \emph{interior} of $\Gamma$, and we write $O=\mathrm{int}\,\Gamma$. That $\Gamma$ is homotopic to a point in particular holds if it is included in a simply connected open set $U$ which is  the domain of a chart. Recall also that $E\subset\RS$ is called \emph{schlicht} over $U\subset\C$ if $U\supset p(E)$ and the restriction $p_{\mathcal bE}:E\to U$ is injective. 

\begin{lemma}
\label{tdG}
Let $E \subset\RS$ be schlicht over $D$ and $\xi\in\mathbf{rp}(\RS)$. There exists a neighborhood $U_\xi$ of $\xi$ such that no Jordan curve in $U_\xi$ contains $\xi$ in its interior and simultaneously is contained in $E$.
\end{lemma}
\begin{proof}
Let $U$ be a simply connected domain of a chart containing $\xi$, so that every Jordan curve $\T\to U$ is homotopic to a point. Let   \( O \subset D \) be a disk centered at \( p(\xi) \) and \( O_\xi \) the connected component of \( p^{-1}(O) \) containing \( \xi \), with \( O \)  small enough that $O_\xi\subset U$ and \( O_\xi\cap\mathbf{rp}(\RS) = \{ \xi \} \); this is possible since $\RS_*$ is compact. Then \( O_\xi \) is a neighborhood of \( \xi \) which is  (isomorphic via a biholomorphic map fixing $\xi$ to) an \( m(\xi) \)-sheeted cyclic covering of $O$. So, if we identify the homotopy groups $\pi_1(O_\xi\setminus\{\xi\})$ and $\pi_1(O\setminus\{p(\xi)\})$ with the infinite cyclic group generated by the symbol $a$, the induced morphism $p_*:\pi_1(O_\xi\setminus\{\xi\})\to \pi_1(O\setminus\{p(\xi)\})$ is the map $a\mapsto a^{m(\xi)}$. Now, a parametrized Jordan curve $\gamma:\T\to O_\xi$ containing $\xi$ in its interior is a generator of the fundamental group of $O_\xi\setminus\{\xi\}$. Hence, $p\circ \gamma$ is the $m(\xi)$-power of a generator of the fundamental group of $O\setminus\{p(\xi)\}$, in particular it  has winding number $\pm m(\xi)$ with respect to $p(\xi)$. However, if $\gamma$ is valued in $E$, then $p\circ\gamma:\T\to O$ is a parametrized Jordan curve because $p$ is injective on $E$, and therefore it has winding number $\pm1$ with respect to $p(\xi)$. This contradicts the assumption that $m(\xi)>1$, showing that $U_\xi:=O_\xi$ satisfies our requirements.
\end{proof}
 
The next lemma, used in the proof of  Lemma \ref{lem12a}, depends on Lemma \ref{tdG}. 
\begin{lemma}
\label{lem:compthicks}
If $E\subset\RS$ is schlicht over \(D\) and $\xi\in\mathbf{rp}(\RS)$, then $\RS\setminus E$ is non-thin at $\xi$.
\end{lemma}
\begin{proof}
  Let $(U_\xi,\varphi)$ be a chart around $\xi$, with $U_\xi$ as in Lemma~\ref{tdG} and $\varphi(\xi)=0$. If $\RS\setminus E$ is thin at $\xi$, then there is a  circle $\T_r:=\{|z|=r\}\subset \varphi(U_\xi)$ such that $\varphi^{-1}(\T_r)\subset E$. As $\varphi^{-1}:\T_r\to U_\xi$ is a Jordan curve in  $E$ that contains $\xi$ in its interior, this contradicts Lemma \ref{tdG}.
\end{proof}

\subsection{Regularity}
\label{ssec_reg}

Thinness is intimately  connected to the notion of a regular boundary point with respect to the Dirichlet problem. Given a Euclidean open set \( U \) with \( \overline U\subset \Omega \), a point $\zeta\in\partial U$ is called \emph{regular} if for any continuous function $\psi$ on $\partial U$ it holds that
\[
\lim_{U\ni z\to\zeta} H_\psi(z) = \psi(\zeta),
\]
where $H_\psi(z)$ is the Perron-Wiener-Brelot solution of the Dirichlet problem on $U$ with boundary data $\psi$, see \cite[Section~VI.6, item~$\gamma)$]{Brelot} for a description of the Perron-Wiener-Brelot process; other boundary points are called \emph{irregular}. When all its boundary points are regular, we say that $U$ itself is regular. It is known that $\zeta\in\partial U$ is regular if and only if
\begin{equation}
\label{regGB}
\lim_{U\ni z\to\zeta} g_U(z,w) = 0
\end{equation} 
for some (and then any) $w$ in each connected component of $U$. Moreover, $\zeta$ is  irregular if and only if the complement of $U$ is thin at $\zeta$  \cite[Theorem~VII.13]{Brelot}. This entails that regularity is a local notion, in particular  each point of $\partial U$ is  regular as soon as the latter is locally connected, as follows from the analogous property in a Euclidean space \cite[Theorem~4.2.2]{Ransford}.

Let \( \sigma \) be a finite measure, compactly supported in $U$. As $g_U(z,w)$ is bounded for $w\in\supp\,\sigma$ and $z$  outside of a neighborhood of the latter, see Section \ref{ssec_green}, we get from \eqref{regGB} and the dominated convergence theorem that $g(\sigma,U;\cdot)$ extends continuously by zero to the set of regular points of $\partial U$. When \( \sigma \) is not compactly supported, a weaker result is stated in Section~\ref{ssec_level} (Lemma~\ref{lem:thin-sets}).

Regular points of \emph{finely open} sets are defined analogously: when $U$ is finely open, $\zeta\in \partial_\mathsf{f}U$ is said to be regular if $\Omega\setminus U$ is non-thin at $\zeta$. By \eqref{expfb}, the set of regular points is then $b(\partial_\mathsf{f}U)$, and if $\partial_\mathsf{f}U$ is its own base one says that $U$ is regular, see \cite[Section~IV.12]{Fuglede}.

The following results are the natural analogs, for Green potentials on regular open sets of  hyperbolic surfaces, of their logarithmic counterparts in the plane, see \cite[Theorems~I.6.8 \&~I.6.9]{SaffTotik}.

\begin{pdleg}
\label{PDLEG}
  Let $U$ be a regular open set with compact closure $\overline{U}\subset\Omega$. If $\sigma_n$ are positive measures on $U$ with uniformly bounded masses that converge weak$^*$ to some measure $\sigma$ as $n\to\infty$, then
\begin{enumerate}
\item{{\rm [Principle of Descent]}} 
  \begin{equation}
    \label{principleGreen}
    \liminf_{n\to\infty}g(\sigma_n,U;z_n)\geq g(\sigma,U;z), \quad z_n\to z\in U.
    \end{equation}
\item{{\rm [Lower Envelope Theorem]}}
    \begin{equation}
    \label{envelopeGreen}
    \liminf_{n\to\infty}g(\sigma_n,U;z)=g(\sigma,U;z) \quad \text{ for quasi every} \quad z\in U.
    \end{equation}
  \end{enumerate}
\end{pdleg}
\begin{proof}
The arguments are a  minor variation of those used in \cite{SaffTotik}.   For  $M>0$ and $z\in U$, observe from \eqref{regGB} that $\varphi_{M,z}(w):=\min\{M,g_U(z,w)\}$ is continuous on $\overline{U}$ and zero on $\partial U$, locally uniformly with respect to $z$.   Thus, $\varphi_{M,z}$   lies in the closure of  the space $C_c(U)$ of continuous functions on   $U$ with compact support endowed with the $\sup$ norm. Moreover,   $|\varphi_{M,z_n}-\varphi_{M,z}|$ is arbitrary small on $U$ for $n$ large enough, by the minimum principle and the continuity  of Green functions off the diagonal. Hence, as  $\sigma_n\overset{w*}{\to}\sigma$, we get that   $\lim_n\int\varphi_{M,z_n}d\sigma_n=\int\varphi_{M,z}d\sigma$ and consequently, by monotone convergence, we deduce \eqref{principleGreen} from the relations
\[
  g(\sigma,U;z) = \lim_{M\to\infty} \int\varphi_{M,z}d\sigma = \lim_{M\to\infty} \lim_{n} \int\varphi_{M,z_n}d\sigma_n \leq \liminf_n g(\sigma_n,U;z_n).
\]
Next, assume to the contrary that $g(\sigma,U;z)<\liminf_n g(\sigma_n,U,z)$ for $z\in K$, where $K\subset U$ is such that $\cp_U(K)>0$. Clearly, we may suppose that $K$ is compact and so we can find a nonzero measure $\sigma_*$, supported on $K$, such that $g(\sigma_*,U;\cdot)$ is continuous on $U$, see Section~\ref{ssec_cap}. Then by Fatou's lemma it holds that
\begin{equation}
\label{Fatl}
\int g(\sigma,U;z)d\sigma_*(z) < \int \liminf_n g(\sigma_n,U;z)d\sigma_*(z) \leq \liminf_n \int g(\sigma_n,U;z) d\sigma_*(z).
\end{equation}
Moreover, as $g(\sigma_*,U;\cdot)$ extends continuously by zero on $\partial U$, see discussion after \eqref{regGB}, it lies in the closure of $C_c(U)$. Therefore, by Fubini's theorem,
\[
\lim_n \int g(\sigma_n,U;z)d\sigma_*(z) = \lim_n \int g(\sigma_*,U;z)d\sigma_n(z) = \int g(\sigma_*,U;z) d\sigma(z) =
 \int g(\sigma,U;z)d\sigma_*(z),
\]
thereby contradicting \eqref{Fatl}.
\end{proof}

\subsection{Superlevel Sets of Green Potentials}
\label{ssec_level}

In this section we restrict attention to a planar simply connected domain \( D \), which is the interior of a Jordan curve \( T \). In this case, any conformal map \( \phi:\D\to D \) extends to a homeomorphism from \( \overline \D \) onto \( \overline D \) \cite[Theorem 2.6]{Pommerenke}, that we continue to denote with $\phi$.  Clearly, such a domain \( D \) is  regular. As mentioned in the previous subsection, if \( \sigma \) is a finite Borel measure compactly supported in \( D \), then \( g(\sigma,D;\cdot) \) continuously extends by zero to \( T \). If \( \sigma \) is not compactly supported this may not hold, but when
\(F\) is a relatively closed subset of \( \D \) with a limit point \( \xi\in \T \), it was shown in \cite{Lueck86} that
\begin{equation}
\label{luecking}
\lim_{\epsilon\to 0} \cp_\D( F\cap \{|z-\xi|<\epsilon\})>0 \quad \Rightarrow \quad \liminf_{F\ni z\to\xi} g(\sigma,\D;z) = 0,
\end{equation}
and if the rightmost limit holds for every finite measure \( \sigma \), then the implication can be reversed. This result is in fact stated in \cite{Lueck86} with \( g(\sigma,\D;z) \) replaced by \( (1-|z|)g(\nu,\D;z) \) where \( \nu \) is any measure whose Green potential is not identically \( +\infty \), but the latter condition is equivalent to saying that the measure \(  d\sigma(z) := (1-|z|)d\nu(z) \) is finite and then convergence to zero along \( F \) of the limit inferior of \( (1-|z|)g(\nu,\D;z) \) and of \( g(\sigma,\D;z) \) are equivalent, see \cite[Section~3, Lemma]{Lueck86}. It is also pointed out in  \cite[Equation (2.5)]{Lueck86} that the leftmost limit in \eqref{luecking} can be equivalently replaced by \( \cp_\D( F\cap \{|z-\xi|<\epsilon\}) = \infty \) for every \( \epsilon>0 \) (note that instead of the Greenian capacity $\cp_\D(E)$ that we use,  \cite{Lueck86} employs the hyperbolic capacity  $\exp\{-1/\cp_\D(E)\}$). In Lemma~\ref{lem:thin-sets} below, we derive a useful consequence of \eqref{luecking}. With the notation of this lemma, we stress that  a stronger conclusion in fact holds  quasi everywhere, namely $U_\epsilon$ is thin at quasi every point of $T$ (this can be deduced from general properties of balayage covered in Section~\ref{ssec_bal}). The interest of Lemma~\ref{lem:thin-sets} lies with the fact  that its conclusion holds at \emph{every} point of $T$.

\begin{lemma}
\label{lem:thin-sets}
Let \( \sigma \) be a finite measure in \( D \), and for \( \epsilon>0 \) set  \( U_\epsilon := \{z\in D: g(\sigma,D;z) > \epsilon \} \). Let \( \xi \in \overline U_\epsilon \cap T \) and \( \phi:\D\to D \) be a conformal map such that \( \phi(1) = \xi \). Then there exists a closed set \( R_\epsilon \subset[0,1]\) such that \( R_\epsilon\cap(1-\delta,1) \) is non-polar for any \( \delta\in(0,1) \), and for each \( r\in R_\epsilon \) one has \( U_\epsilon \cap \phi(\{z\in\D:|1-z|=1-r\}) =\varnothing \).
\end{lemma}
\begin{proof}
  Suppose initially that \( D=\D \). Without loss of generality, we can assume that \( \xi=1 \). First, we shall show that
\begin{equation}
\label{lueck1}
\lim_{\delta\to0} \cp_\D ( U_\epsilon \cap D_\delta )  = 0,
\end{equation}
where \( D_\delta:= \{z\in\D:|z-1|<\delta\} \). For $\zeta\in\D$, let \( S(\zeta) := \{z\in\D:g_\D(z,\zeta)\geq \log 2\} \). It was shown in \cite[Section~3, Lemma]{Lueck86} that there exists \( \delta_0=\delta(\epsilon)>0 \) for which
\[
  \int_{\D\setminus S(\zeta)} g_\D(z,\zeta) d\sigma(z) \leq \epsilon/2, \quad |\zeta|>1-\delta_0.
\]
In particular, this inequality holds for \( \zeta\in D_{\delta_0} \). Hence, for any compact subset \( F\subset U_\epsilon \cap D_{\delta_0}\), it holds when \(\zeta\in F\) that
\begin{equation}
  \label{demajpL}
h(\zeta) := \int_{S(\zeta)} g_\D(z,\zeta) d\sigma(z) = g(\sigma,\D;\zeta) - \int_{\D\setminus S(\zeta)} g_\D(z,\zeta) d\sigma(z) \textcolor{blue}{>} \epsilon/2.
\end{equation}
Assume to the contrary that the limit in \eqref{lueck1} is larger that \( 2\eta>0 \) (the limit must exist as
  $\cp_\D ( U_\epsilon \cap D_\delta )$ decreases with \( \delta \)). Since \( U_\epsilon \) is an open set, we get from \eqref{GreenCap} that for any \( \delta>0 \) there exists a compact set \( F_\delta \subset U_\epsilon \cap D_\delta \) for which \( \cp_\D(F_\delta)\geq \eta \). This entails that there exist a sequence \( \delta_n\to0 \) and disjoint compact sets \( F_n \subset U_\epsilon \cap D_{\delta_n} \) with \( \cp_\D(F_n)\geq \eta \). Let \( \nu_n \) be the Green equilibrium distribution on \( F_n \) and \( F_n^* := \{z\in S(\zeta):\text{for some } \zeta\in F_n\} \). In view of \eqref{demajpL},
\[
\epsilon/2 \leq \int_{F_n} h(\zeta) d\nu_n(\zeta) \leq \frac1\eta \sigma(F_n^*) \to 0,
\]
where the second inequality and the fact that $\lim_n\sigma(F_n^*) \to 0$ can be established as in the proof of \cite[Theorem~1]{Lueck86} (compare to p. 486 of that reference). This contradiction proves \eqref{lueck1}.

Let \( \mathsf T:D_1 \to (0,1) \) be defined by \( z \mapsto \mathsf Tz:=1-|1-z| \) and put $V_\epsilon:=U_\epsilon\cap D_1$. Denoting by $ \mathsf TV_\epsilon$ the set $\{\mathsf T\zeta: \,\zeta\in V_\epsilon\}$, we claim that
\begin{equation}
\label{lueck2}
\lim_{\delta\to0} \cp_\D( \mathsf TV_\epsilon \cap D_\delta) = 0.
\end{equation}
Before proving \eqref{lueck2}, let us show why it implies the lemma. For this, consider \( R_\epsilon:=[0,1]\setminus \mathsf TV_\epsilon \), which is a closed set. If the conclusion of the lemma were not true, there  would exist \( \delta_0>0 \) such that \( R_\epsilon\cap(1-\delta_0,1) \) is polar. By definition of \( \mathsf T \) this would imply that \( \cp_\D(\mathsf TV_\epsilon\cap D_\delta) = \cp_\D((1-\delta,1)) = \infty \) for any \( \delta<\delta_0 \) (the last equality follows at once from the definition of the Greenian capacity), which contradicts \eqref{lueck2}. 

We are now left to demonstrate \eqref{lueck2}. Assume to the contrary that it does not hold, i.e., there exists \( \eta>0 \) such that for any \( \delta>0 \) there is a compact set \( F_\delta \subset \mathsf TV_\epsilon \cap D_\delta \) for which \( \cp_\D(F_\delta)\geq \eta \). The previous inequality means that there exists a probability measure \( \mu_\delta \) supported on \( F_\delta \) such that
\[
\int\int g_\D(x,y)d\mu_\delta(x)d\mu_\delta(y) \leq \frac1\eta.
\]
Since \( U_\epsilon \) is open, so is \( V_\epsilon \) and one easily sees that each \( z\in V_\epsilon\cap D_\delta \) has a neighborhood, say \( O_z \), whose closure is contained in \( V_\epsilon\cap D_\delta \) and whose circular projection \( \mathsf TO_z \) is an open subinterval of \( (0,1) \). These subintervals form an open cover of \( F_\delta \), which necessarily has a finite subcover, say \( \mathsf TO_{z_1},\ldots,\mathsf TO_{z_N} \). The closure \( K_\delta \) of \( O_{z_1} \cup \cdots \cup O_{z_N} \) is a compact subset of \( V_\epsilon\cap D_\delta \), and clearly \( F_\delta\subset \mathsf T K_\delta \). In particular, there exists a probability measure \( \nu_\delta \) on \( K_\delta \) such that \( \nu_\delta \mathsf T^{-1}=\mu_\delta \), see for example \cite[Theorem~A.4.4]{Ransford}. Then, by Fubini's theorem, it holds that
\[
\frac1\eta \geq \int\int g_\D(x,y)d\mu_\delta(x)\mu_\delta(y) = \int\int g_\D(\mathsf Tz,\mathsf Tw)d\nu_\delta(z)d\nu_\delta(w),
\]
and if we can show that \( g_\D(\mathsf Tz,\mathsf Tw) \geq g_\D(z,w) \) then we will deduce from the above estimate that
\[ 
\cp_\D( U_\epsilon \cap D_\delta)=\cp_\D( V_\epsilon \cap D_\delta) \geq \cp_\D(K_\delta) \geq \eta>0  \quad \text{for  any}\  \delta>0 ,
\]
which of course contradicts \eqref{lueck1}. Hence, the proof has been reduced to the verification of \( g_\D(\mathsf Tz,\mathsf Tw) \geq g_\D(z,w) \) for $z,w\in D_1$, that we now carry out. Since \( g_\D(z,w) = \log |(1-z\bar w)/(z-w)| \) and \(\mathsf T\) is real-valued, we need to show that
\begin{equation}
\label{lueck3}
E := |1-\mathsf Tz{\mathsf Tw} |^2 |z-w|^2 - |1-z\bar w|^2|\mathsf Tz-\mathsf Tw|^2\geq0.
\end{equation}
Set \( a\xi:=1-z \) and \( b\eta:=1-w \), where \( a,b\in(0,1) \) and \(|\xi|=|\eta|=1\) with \( \re\xi,\re\eta \in (0,1)\). Then
\begin{eqnarray*}
E &=& |a+b-ab|^2|a\xi-b\eta|^2 - |a\xi + b\bar\eta - ab\xi\bar \eta|^2 |a-b|^2 \\
&=& (S+2abU)(T-2abV) - (S+2ab W)(T-2ab)
\end{eqnarray*}
where \( S:=a^2+b^2+(ab)^2 \), \( T:= a^2+b^2 \), \( U := 1-a-b \), \( V:= \re(\xi\bar\eta) \), and \( W := \re( \xi \eta - a \eta - b \xi)\). Therefore,
\begin{eqnarray*}
E &=& 2ab\big( S(1-V) + T(U-W) + 2ab(W-UV) \big) \\
&=& 2ab \big( (S+2abU)(1-V) + (T-2ab)(U-W) \big) \\
&=& 2ab \big( (a+b-ab)^2(1-V) + (a-b)^2(U-W) \big).
\end{eqnarray*}
Because $V\leq1$, the above expression can  be estimated from below as
\begin{eqnarray*}
E &\geq& 2ab (a-b)^2 ( 1 - V + U - W) \\
&=& 2ab(a-b)^2 \big( 2 - \re(\xi\bar\eta + \xi\eta) - a(1-\re\eta) - b(1-\re\xi) \big),
\end{eqnarray*}
and since  \( 1-\re\eta\), \(1-\re\xi\), \(1-a\) as well as \(1-b\) are all positive, it therefore holds that
\begin{eqnarray*}
E &\geq& 2ab (a-b)^2 \big( \re\eta + \re\xi  - \re(\xi\bar\eta + \xi\eta) \big) \\
&=& 2ab(a-b)^2 \big( \re\eta + \re\xi  - 2\re\eta\re\xi \big) \\
& = & 2ab(a-b)^2 \big( \re\eta(1-\re\xi) + \re\xi(1-\re\eta) \big).
\end{eqnarray*}
As \( \re\xi,\re\eta \in [0,1]\), this establishes \eqref{lueck3} and completes the proof of the lemma when \( D=\D \).

Finally, it remains to reduce the case of a general domain \( D \) to the one of the unit disk. Using \cite[Theorem~A.4.4]{Ransford} once more, let \( \nu \) be a finite measure in \( \overline\D \) such that \( \nu\phi^{-1} = \sigma \). Then
\begin{align*}
g(\sigma,D;\phi(z)) & = \int g_D(\phi(z),\zeta) d\sigma(\zeta) = \int g_D(\phi(z),\phi(w)) d\nu(w) \\
& = \int g_\D(z,w) d\nu(w) = g(\nu,\D;z), \quad z\in\D,
\end{align*}
by conformal equivalence of Green functions. Since \( \phi:\overline\D\to\overline D \) is a bijection, the superlevel set
\( \{z\in\D:g(\nu,\D;z)>\epsilon\} \)  is equal to \( \phi^{-1}(U_\epsilon) \), from which the desired result follows.
\end{proof}

\subsection{Balayage}
\label{ssec_bal}

Let $v$ be a non-negative superharmonic function on $\Omega$ and $E$ be a subset of $\Omega$. The \emph{balayage function} (or \emph{regularized reduction}) of $v$ relative to $E$, denoted by $\mathcal{B}_v^E$, is the lower semi-continuous regularization  of
\begin{equation}
\label{defbalf}
\inf\big\{u|~u\text{ is superharmonic and positive in }\Omega,~u\geq v\text{ on }E\big\},
\end{equation}
see \cite[Section~5.3]{Helms} for an account on $\R^n$ that carries over to $\Omega$ without change; in fact,  $\mathcal{B}_v^E$ coincides with the infimum in \eqref{defbalf} except perhaps on a polar set where lower semi-continuous regularization may modify the value. The balayage function $\mathcal{B}_v^E$ is superharmonic in $\Omega$, harmonic in $\Omega\setminus\overline E$, and equal to $v$ on $b(E)$ \cite[Section~VIII.1]{Brelot}. Clearly,  $\mathcal{B}_v^E\leq v$ everywhere, since $v$ qualifies as one of the functions $u$ in \eqref{defbalf}. The balayage function $\mathcal{B}_v^E$ does not change if \( E \) gets replaced by \( b(E) \) or by \( \clos_\mathsf{f}(E) \); in fact, it remains invariant if $E$ is altered by a polar set.

When $E$ is compact, the \hyperref[sdp]{Strong Domination Principle} and  properties of the Green equilibrium potential $g(\mu_{\Omega,E},\Omega;\cdot)$ imply that $\mathcal{B}_1^E=\cp_{\Omega}(E)g(\mu_{\Omega,E},\Omega;\cdot)$. Thus, it follows from the left continuity of $\mathcal{B}_1^E$ with respect to $E$, see \cite[Section VI.10 e)]{Brelot}, and the monotone convergence theorem that the outer Greenian capacity of an arbitrary set $E\subset\Omega$ is the mass of $\Delta\mathcal{B}_1^E$ (in fact, this is the way the outer capacity is defined in \cite[Section VIII.4]{Brelot}). From this and \cite[Theorem~VIII.12]{Brelot}, we deduce in particular that $E$ is polar if and only if $\cp_{\Omega}(E)=0$, justifying a claim made in Section~\ref{ssec_fine}.

If $v$ is the Green potential of a positive Borel measure $\sigma$, then $\mathcal{B}_v^E$ is a Green potential as well \cite[Section~VI.11]{Brelot} and the measure $\sigma^E$ such that $\mathcal{B}_v^E=g(\sigma^E,\Omega;\cdot)$ is called the \emph{balayage} of $\sigma$ relative to $E$. The measure $\sigma^E$ is characterized as the unique measure satisfying
\begin{equation}
\label{char_balayage}
\sigma^E\big(\Omega\setminus b(E)\big) =0 \qandq g(\sigma,\Omega;z) = g\big(\sigma^E,\Omega;z\big), \quad z\in b(E),
\end{equation}
see \cite[Theorem~VIII.3]{Brelot}. From \eqref{char_balayage}, one deduces at once that
\begin{equation}
\label{transit}
(\sigma^E)^F=\sigma^F,\qquad F\subset E\subset\Omega.
\end{equation}
 Moreover, it holds by \cite[Section~VI.12, Equation~(13)]{Brelot} that
\begin{equation}
\label{preservemass}
\sigma^E(B) = \int \delta_x^E(B)d\sigma(x),\qquad B\quad\mathrm{Borel},
\end{equation}
while it  follows from \cite[Section~VI.12, Equation~(9)]{Brelot} and Fubini's theorem that
\begin{equation}
\label{GreenHarmonicMeasure}
g\big(\sigma^E,\Omega;z\big) = \int g(\sigma,\Omega;x)d\delta_z^E(x).
\end{equation}
Since $g(\sigma^E,\Omega;\cdot)\leq g(\sigma,\Omega;\cdot)$ for any Radon measure $\sigma$ by the just discussed properties of balayage,  it follows from Fubini's theorem that $\int\nu d\sigma^E\leq\int\nu d\sigma$ for any Green potential $\nu$. Since any non-negative superharmonic function $\nu$ is an increasing limit of potentials, see \cite[Lemma~1.1]{Const67}, monotone convergence yields that $\int\nu d\sigma^E\leq\int\nu d\sigma$ remains valid for such functions as well, see also \cite[Section~I.3]{Fuglede}.  In particular, the mass of $\sigma^E$ cannot exceed the mass of $\sigma$.

The \emph{fine support} of a Radon measure $\sigma$, denoted by $\supp_\mathsf{f}\,\sigma$ when it exists, is the smallest finely closed carrier of $\sigma$. A sufficient condition for its existence is that $\sigma$ does not charge polar sets, in which case $\supp_\mathsf{f}\,\sigma$ is its own base, see \cite[Theorem~VII.12]{Brelot}. If $\sigma$ is admissible, meaning that $g(\sigma,\Omega;\cdot)\not\equiv+\infty$, see Section~\ref{ssec_pot}, and if $\sigma(F)=0$ for some  polar set $F$, then $\sigma^E(F)=0$  for any $E\subset\Omega$ \cite[Theorem~1]{Fug71} (as usual, $\sigma(F)$ means the outer $\sigma$-measure of $F$ when the latter is not Borel). In particular, if \( \sigma \) is admissible and $\sigma(b(E))=0$, then $\sigma^E$ does not charge polar sets since it is carried by $b(E)$. Thus, $\supp_\mathsf{f}\,\sigma^E$ exists in this case. An important special case is handled by the  following lemma, the first  item of which follows from the preceding discussion.

\begin{lemma} \cite[Corollary~1 to Theorem~4]{Fug71} \cite[Corollaries~2 \&~3 to Theorem~12.7]{Fuglede}
\label{supfin}
\begin{itemize}
\item[(i)] If $W$ is finely open and either $z\in W$ or $z\in i(\Omega\setminus W)$, then $\delta_z^{\Omega\setminus W}$ does not charge polar sets.
\item[(ii)] Let $V$ be a regular finely open set, $z\in V$, and $V_z$ the fine  component of $V$ containing $z$. Then $V_z$ is regular, and $\Omega\setminus V_z$ is largest among all the bases $B$ such that $\delta_z^B = \delta_z^{\Omega\setminus V}$. Moreover, the fine support of $\delta_z^{\Omega\setminus V}$ exists and
\begin{equation}
\label{baldfi}
\supp_\mathsf{f}\,\delta_z^{\Omega\setminus V}=\partial_\mathsf{f} V_z\subset\partial_\mathsf{f} V.
\end{equation}
\item[(iii)]
Let $U$ be a fine domain and $z\in U$ or $z\in i(\Omega\setminus U)$. Then the fine support of $\delta_z^{\Omega\setminus U}$ exists and
\begin{equation}
\label{detfs}
\supp_\mathsf{f}\,\delta_z^{\Omega\setminus U}=b(\partial_\mathsf{f} U)=b(\Omega\setminus U)\cap \partial_\mathsf{f} U.
\end{equation}
\end{itemize}
\end{lemma}

Let now $O\subset\Omega$ be (Euclidean) open, $z\in O$, and $O_z$ the connected component of $O$ containing $z$. If we let $V$ be the regular finely open set obtained by adjoining to $O$ the polar set $i(\Omega\setminus O)$ and $V_z$ the fine component containing $z$, then we get from \cite[Theorem~6]{Fug71} (see discussion there) that  $O_z=V_z\setminus i(\Omega\setminus O)$. Thus, since the balayage function remains the same if the set relative to which it is defined is altered by a polar set, Lemma \ref{supfin} (ii) implies that
\[
\delta_z^{\Omega\setminus O} = \delta_z^{\Omega\setminus V} = \delta_z^{\Omega\setminus V_z} = \delta_z^{\Omega\setminus O_z}
\]
and the latter is carried by the regular points of $\partial O_z$. Moreover, if $O$ has compact closure in $\Omega$, then $\delta_z^{\Omega\setminus O}$ is a probability measure, and for $h$ a harmonic function in \( O \)  with continuous extension to $\overline{O}$:
\begin{equation}
\label{balcont1}
h(z)=\int h\,d\delta_z^{\Omega\setminus O},\qquad z\in O.
\end{equation}
Indeed, \eqref{balcont1} follows from \cite[Section~VI.12, application~1]{Brelot} since $h$ is the Perron-Wiener-Brelot solution of the Dirichlet problem\footnote{In fact, the right-hand side of \eqref{balcont1} is the  Perron-Wiener-Brelot solution of the Dirichlet  problem on $O$ with boundary data $h$  as soon as the latter is summable against $\delta_z^{\Omega\setminus O}$ for one (and then any) $z$ in each component of $O$.  } in $O$ with boundary data $h_{\mathcal{b}\partial O}$; see \cite[Section~VI.6, item $\gamma)$]{Brelot}. Equality in \eqref{balcont1} shows that the measure \( \delta_z^{\Omega\setminus O} \) does not depend on $\Omega$, provided that the latter is hyperbolic and compactly contains \( \overline O \). It is called the  \emph{harmonic measure} for \( O\) (at $z$).  More general versions of \eqref{balcont1}, involving the fine Dirichlet problem and cases where $\overline{O}$ is non-compact, are stated  in Theorem~\ref{thm_fineDirichlet} and Lemma~\ref{lem:mass-balayage} further below.

When $\Omega\subset \C$ and $\overline{O}$ is compact in $\Omega$, it follows from \cite[Chapter II, Theorem 5.1]{SaffTotik} that
\begin{equation}
\label{balogh}
g(\sigma,O,z)=V^{\sigma}(z)-V^{\sigma^{\Omega\setminus O}}(z),\qquad z\in\Omega,
\end{equation}
where $V^\sigma$ is the logarithmic potential of $\sigma$ and the left-hand side is interpreted as $0$ for $z\in b(\Omega\setminus O)$. More general versions when $\infty\in\Omega$ may be found in \cite{SaffTotik}. If moreover $O$ is a domain with $K\subset O$ a non-polar compact set such that $\Omega\setminus K$ regular and $O\setminus K$ is non-thin at every point of $K$, then 
\begin{equation}
\label{fsuppeqm}
\supp_\mathsf{f}\,\mu_{O,K}=K.
\end{equation}
Indeed, $K$ is its own base and $K=\partial_{\mathsf{f}}(O\setminus K)$ by assumption, while $\mu_{O,K}$ is the balayage onto $K$ of the equilibrium measure on the plate $\partial O$ of the condenser $(\partial O,K)$ \cite[Chapter VIII, Theorem 2.6]{SaffTotik}. Thus, \eqref{fsuppeqm} follows from \eqref{preservemass} and \eqref{detfs}.

\subsection{Green Potentials in \( D \) and on \( \RS \)}

In this subsection, we connect Green functions and potentials on the domain $D$ and surface $\RS$ defined in Section \ref{laclasse}. First, let us show that
\begin{equation}
\label{rapFG}
g_D (x,y) = \sum_{z\in p^{-1}(x)} m(z)g_{\RS}(z,w),
\end{equation}
where \( m(z) \) is the ramification order of $\RS$ at \( z \) and \( w \) is an arbitrary  element of the fiber \( p^{-1}(y) \). To check \eqref{rapFG}, note that for fixed $y\notin p(\mathbf{rp}(\RS))$ and $w\in \RS$ with $p(w)=y$, the right-hand side is well defined and harmonic as a function of $x\in D\setminus (p(\mathbf{rp}(\RS)) \cup \{y\})$. Thus, it is harmonic for $x\in D\setminus \{y\}$ by the \hyperref[rt]{Removability Theorem} and the continuity of Green functions off the diagonal. Moreover, the right-hand side clearly has a logarithmic singularity at $y$, and since $\lim_{z\to\partial \RS}g_{\RS}(z,w)=0$ by the regularity of $\RS$ in $\RS_*$, its largest harmonic minorant is zero. This proves  \eqref{rapFG} when $y\notin p(\mathbf{rp}(\RS))$, and the general case follows by continuity of Green functions off the diagonal. Consequently, if \( \nu \) is a Radon measure on \( \RS \) and $p_*(\nu)$ denotes its pushforward under $p$ (the measure on $D$ such that $p_*(\nu)(B)=\nu(p^{-1}(B))$ for a Borel set $B$), integrating \eqref{rapFG} against $p_*(\nu)$ with respect to $y$ and changing variables yields
\begin{equation}
\label{push-down}
g(p_*(\nu),D;x) = \sum_{z\in p^{-1}(x)} m(z)g(\nu,\RS;z).
\end{equation}

In the other direction, for a Radon measure $\sigma$ on $D$, let us define $\widehat{\sigma}$  by
\begin{equation}
\label{lifted-measureA}
\widehat{\sigma}(B) := \int_{p(B)}\sum_{z\in p^{-1}(x)\cap B}m(z)d\sigma(x), \quad B\subset \RS,\quad B\mathrm{\ Borel}.
\end{equation}
As $p(B)$ is Borel when $B$ is Borel, one easily checks that \eqref{lifted-measureA} defines a Radon measure on $\RS$. In fact, one can verify that $\widehat{\sigma}=\sigma^*+\sum_{z\in\mathbf{rp}(\RS)}m(z)\sigma(\{p(z)\})\delta_z$, where $\sigma^*$ is the pullback measure resulting from Carathéodory's construction as applied to the map $B\mapsto\sigma(p(B)\setminus p(\mathbf{rp}(\RS)))$ defined on Borel subsets of $\RS$, see \cite[Theorem 2.10.10]{Federer}.

Partitioning $D\setminus p(\mathbf{rp}(\RS))$ into countably many Borel sets $B_k$ such that $p: p^{-1}(B_k)\to B_k$ induces a homeomorphism on each connected component of $p^{-1}(B_k)$, and invoking the \hyperref[rt]{Removability Theorem} to proceed by superharmonicity from the case where $\zeta\notin \mathbf{rp}(\RS)$, one deduces from \eqref{lifted-measureA} and \eqref{rapFG} that
\begin{equation}
\label{remontp}
g\left(\widehat{\sigma},\RS;\zeta\right) = g(\sigma,D;p(\zeta)), \quad \zeta\in\RS.
\end{equation}

As a consequence of definition \eqref{lifted-measureA}, we claim that if a sequence $\{\mu_n\}$ of finite positive measures supported on a fixed compact set $K\subset D$ converges weak$^*$ to $\mu$ on $D$, then the sequence $\{\widehat{\mu}_n\}$ converges weak$^*$ to \( \widehat\mu \) on $\RS$. Indeed, the total mass of $\mu_n$ is  necessarily bounded independently of $n$  by some $C>0$ (this follows from the Banach-Steinhaus principle) and therefore, in view of \eqref{lifted-measureA}, the total mass of $\widehat{\mu}_n$ is  bounded by $MC$, where $M$ is the number of sheets of  $\RS$. Hence, an arbitrary subsequence of $\{{\widehat\mu}_n\}$ has a subsequence, say $\{{\widehat\mu}_{n_k}\}$, that converges weak$^*$ on $\RS$ to some finite measure, say $s$. It follows from the \hyperref[PDLEG]{Lower Envelope Theorem} that $\liminf_k g(\widehat{\mu}_{n_k},\RS;z)=g(s,\RS;z)$ for quasi every $z\in\RS$. Similarly, \eqref{remontp} and the \hyperref[PDLEG]{Lower Envelope Theorem}, applied this time to $\{\mu_{n_k}\}$, yield that $\liminf_{k} g(\widehat{\mu}_{n_k},\RS;z) = g(\mu,D;p(z)) = g(\widehat{\mu},\RS;z)$ for quasi every $z\in\RS$. Thus, $g(s,\RS;\cdot)=g(\widehat{\mu},\RS;\cdot)$ quasi everywhere on \( \RS \), and the claim follows by taking Laplacians on both sides of this equality. 

In the previous construction, $\RS$ may of course be replaced by another saturated connected bordered surface \( \mathcal S\subset \RS_* \) with bounded projection such that \( \overline\RS\subset \mathcal S \). Therefore,
\[
\Bigl(\mu_n\overset{w*}{\to}\mu \quad \text{in} \quad \overline D\Bigr) \quad \Rightarrow \quad \Bigl(\widehat{\mu}_n\overset{w*}{\to}\widehat{\mu} \quad \text{in} \quad \overline \RS\Bigr), 
\]
because \( \{\mu_n \}\) also converges weak$^*$ in \( p(\mathcal S)\supset \overline D \) whence the measures $\widehat{\mu}_n$ converge weak$^*$ on \( \mathcal S \), while having their supports contained in \( \overline \RS \). The ``hat measure'' constructed in \eqref{lifted-measureA} is instrumental both in the proof of Lemma \ref{lem8a} and of the following technical result, needed in  the paper.

\begin{lemma}
\label{saturated}
Let  \( \sigma \) be a measure in \( D \). Given \( E \subset D \), set \( \widehat E:=p^{-1}(E) \). Then it holds that
\[
g\big( \sigma^E,D;p(z) \big) = \int g( \sigma,D;p(x)) d\delta_z^{\widehat E}(x).
\]
\end{lemma}
\begin{proof}
If $u$ is superharmonic on $D$, it is obvious from \eqref{remontp} that
\[
\Bigl(u(z)\geq g( \sigma,D;z),  \quad z\in E \Bigr)\quad \Rightarrow  \quad \Bigl(u(p(\xi))\geq g(\widehat\sigma,\RS;\xi),  \quad \xi\in \widehat E \Bigr).
\]
As $u\circ p$ is superharmonic on \( \RS \), this and the definition of balayage imply that
\begin{equation}
\label{satur2}
g\big( \sigma^E,D;p(\zeta)\big) \geq g\big(\widehat\sigma^{\widehat E},\RS;\zeta\big), \quad \zeta\in \RS.
\end{equation}
Conversely, averaging over $\zeta\in p^{-1}(\{z\})$, equation \eqref{remontp} yields that
\begin{equation}
\label{satur1}
g( \sigma,D;z) = \frac 1M\sum_{\zeta\in p^{-1}(\{z\})}m(\zeta) g(\widehat\sigma,\RS;\zeta),\quad z\in D,
\end{equation}
with \( M \) being the total number of sheets of \( \RS \), entailing when $v$ is superharmonic on $\RS$ that
\[
\Bigl(v(\zeta) \geq g(\widehat\sigma,\RS;\zeta), \quad \zeta\in\widehat{E}\Bigr) \quad \Rightarrow \quad  \Bigl(\frac1M\sum_{\zeta\in p^{-1}(z)}m(\zeta)v(\zeta) \geq g( \sigma,D;z), \quad z\in D\Bigr).
\]
Now, the function $z\mapsto \sum_{\zeta\in p^{-1}(z)}m(\zeta)v(\zeta)$ is well-defined and superharmonic on $D\setminus p(\mathbf{rp}(\RS))$ and therefore on the whole domain \( D \) by the \hyperref[rt]{Removability Theorem}. Thus,  by the definition of balayage, we obtain when $v$ is superharmonic on $\RS$ that  
\[
\Bigl(v(\zeta) \geq g(\widehat\sigma,\RS;\zeta), \quad \zeta\in\widehat{E}\Bigr) \quad \Rightarrow \quad  \Bigl(\frac1M\sum_{\zeta\in p^{-1}(z)}m(\zeta)v(\zeta) \geq g( \sigma^E,D;z), \quad z\in D\Bigr),
\]
and taking the infimum over $v$ before taking the lower semi-continuous regularization gives us, by virtue of the \hyperref[sdp]{Strong Domination Principle}, that
\[
\frac1M\sum_{\zeta\in p^{-1}(z)}m(\zeta)g(\widehat{\sigma}^{\widehat{E}},\RS;\zeta) \geq g( \sigma^E,D;z), \quad z\in D.
\]
Combining the above estimate with \eqref{satur2}, we deduce that
\[
g\big( \sigma^E,D;p(z)\big) = g\big(\widehat\sigma^{\widehat E},\RS;z\big), \quad z\in \RS,
\]
and the conclusion now follows from \eqref{GreenHarmonicMeasure} and \eqref{remontp}.
\end{proof}

\subsection{Dirichlet problem}

The Dirichlet problem on a domain consists in finding a harmonic function in that domain with given boundary data. In the fine Dirichlet problem, one looks for a finely harmonic function on a fine domain to meet prescribed boundary data. A real-valued function \( h \) on a fine domain  $V$ is \emph{finely harmonic} if it is finely continuous, and if the fine topology on $V$ has a basis comprised of finely open sets $E$ with $\clos_\mathsf{f}(E)\subset V$ such that
\[
h(z)=\int h d\delta_z^{\Omega\setminus E} \quad \text{for every} \quad z\in E
\]
(in particular $h$ must be integrable with respect to  \( \delta_z^{\Omega\setminus E} \) for all $z\in E$ and each $E$); one  may even assume that $E$ is regular and has compact closure (with respect to the Euclidean topology) in $V$, see \cite[Sections~8 \&~14]{Fuglede}. Note that a function harmonic in a domain is finely harmonic on any fine subdomain, see \cite[Theorem~8.7]{Fuglede}.

If $V$ is a regular finely open set (recall that it means
$\Omega\setminus V$ is its own base), then $b(\partial_\mathsf{f}V)=\partial_\mathsf{f}V$ by \eqref{expfb} whence the result below is a special case of \cite[Theorem~14.1]{Fuglede} and its proof.

\begin{theorem}
\label{thm_fineDirichlet}
Let $V\subset\Omega$ be a finely open set such that $\Omega\setminus V$ is non-thin at every point of itself, i.e., such that $V$ is regular. If $\psi$ is a finely continuous function on $\partial_\mathsf{f}V$, majorized in absolute value there by a finite semi-bounded potential on $\Omega$, say \( g \), then
\begin{equation}
\label{fineDirichlet}
h_\psi(z) := \int \psi d\delta_z^{\Omega\setminus V} = \int \psi d\delta_z^{\partial_\mathsf{f} V},  \quad z\in V,
\end{equation}
is the unique finely continuous extension of $\psi$ to $\clos_{\mathsf{f}}(V)$ that is finely harmonic  in \( V \) and is majorized in absolute value there by a semi-bounded potential. In fact, it holds that $|h_\psi|\leq g$ on $\clos_{\mathsf{f}}(V)$.
\end{theorem}

The lemma below addresses the question as to when constant functions solve the fine Dirichlet problem on $V\subset\Omega$ or, equivalently by \eqref{fineDirichlet}, as to when the balayage of \( \delta_z \) out of \( V \), \( z\in V \), has unit mass. We recall that an overline, as in $\overline{V}$, or a $\partial$ sign, as in $\partial\Omega$, refer respectively to the closure and boundary with respect to the Euclidean topology induced by the ambient Riemanian manifold ($\RS_*$ or $\C$). In contrast, fine closures and fine boundaries as in \(\partial_\mathsf{f} V\)  and $\clos_{\mathsf{f}}(V)$ refer to the fine topology on $\Omega$; thus, \( \partial \Omega \) is ``invisible'' from the point of view of fine topology in \( \Omega\), and if $V\subset \Omega$ then $\partial V\cap \partial\Omega$ is disjoint from \(\partial_\mathsf{f} V\) as the latter is included in $\Omega$.

\begin{lemma}
\label{lem:mass-balayage}
Let $V$ be a proper nonempty regular fine domain in $\Omega$, which itself is regular within the ambient Riemann surface (\( \RS_* \) or \( \C \)). Then it holds for \( z\in V \) that
\begin{equation}
\label{mass-balayage}
\int d\delta_z^{\Omega\setminus V} \left\{
\begin{array}{ll}
= 1 & \text{if }\overline V \cap \partial\Omega =\varnothing, \medskip \\
<1  & \text{if } \overline{ \partial_\mathsf{f} V } \cap \partial\Omega=\varnothing \text{ and } \overline{V}\cap\partial \Omega \neq \varnothing.
\end{array}
\right.
\end{equation}
Moreover, if either condition on the right-hand side of \eqref{mass-balayage} holds, then for any harmonic function $h$  on \( \Omega \) one has
\begin{equation}
\label{reproduction}
h(z) = \int hd\delta_z^{\Omega\setminus V},\qquad z\in V,
\end{equation}
provided that $|h|$ is majorized on $V$ by a semi-bounded potential in $\Omega$ when \( \overline V \cap\partial\Omega \neq \varnothing \). 
\end{lemma}
\begin{proof} 
  Let \( K\subset \Omega \) be non-polar and compact. Set, for brevity, \( g_K :=\cp_\Omega(K) g(\mu_{\Omega,K},\Omega;\cdot) \), where  \( \mu_{\Omega,K} \) indicates, as in Section~\ref{ssec_cap}, the Green equilibrium distribution on \( K \). Since \( \mu_{\Omega,K} \) has finite energy, \( g_K \) is semi-bounded.  As \( \partial\Omega \) is regular in the ambient Riemann surface, \( g_K \) extends continuously by zero to \( \partial \Omega \),  see Section~\ref{ssec_reg}. Since \(\cp_\Omega(K)>0\),  it holds that \( g_K\leq 1 \) in \( \Omega \), see the paragraph after \eqref{GreenCap}, and \( g_K = 1 \) on \( b(K) \) because \( g_K = \mathcal B_1^K \), see Section~\ref{ssec_bal}.  Moreover, \( g_K<1 \) in each connected component \( U \) of \( \Omega\setminus K \) such that \( \partial U \cap \partial \Omega \neq \varnothing \) by the maximum principle for harmonic functions.

  When $\overline V$ is compactly included in $\Omega$, we may put $K:=\overline V$ in what precedes, and then $\cp_\Omega(K)>0$ as otherwise $V$ itself would be polar and therefore empty, since it is finely open. The infimum of \( g_K \) on \( K \) is attained by lower semi-continuity, and it is strictly positive because nonzero Green potentials are never zero. Therefore, if $h$ is harmonic on $\Omega$, the potential $c g_K$ majorizes $|h|$ on $K$ for sufficiently large $c>0$. The uniqueness part of Theorem~\ref{thm_fineDirichlet} now implies that $h_{\mathcal{b}V}$ is the solution of the fine Dirichlet problem with boundary data $h_{\mathcal b\partial_\mathsf{f}V}$. Hence, \eqref{reproduction} is just \eqref{fineDirichlet} while the upper equality in \eqref{mass-balayage} follows by taking $h\equiv1$. 

  Assume next that \(\overline V \cap \partial\Omega \neq\varnothing\) and \(\overline{ \partial_\mathsf{f} V } \cap \partial\Omega = \varnothing\). Then $K:=\overline{\partial_\mathsf{f} V}$ is a compact subset of $\Omega$ which is non-polar, for if $\partial_\mathsf{f} V$ were polar, then either $V$ or $\Omega\setminus V$ would be polar \cite[Theorem~2]{Fug71} and  $V$ would be either empty or irregular, a contradiction. Note that a subdomain of \( \Omega\setminus K \) is also a fine domain \cite[Theorem~2]{Fug71} and thus, if it contains both a point in $V$ and a point in $\Omega\setminus V$, then it must contain a point in $\partial_\mathsf{f} V$ which is impossible by the definition of \( K \). Hence, $V\setminus K$ is  Euclidean open.

Let \( U \) be a connected component of $V\setminus K$ such that $\partial U\cap\partial\Omega\neq\varnothing$; it exists because \(\overline V \cap \partial\Omega \neq\varnothing\).  Since $\partial_\mathsf{f} V=b(\partial_\mathsf{f} V)\subseteq b(K)$ by assumption, \( g_K = 1 \) on \( \partial_\mathsf{f} V \) (see the beginning of the proof). Hence, as $g_K$ is a semi-bounded potential on $\Omega$,  it follows from Theorem~\ref{thm_fineDirichlet} that $h_{\bf 1}\leq g_K \) in \( V \), where $h_{\bf 1}(z)$ is the solution of the fine Dirichlet problem on $V$ with boundary data identically $1$ on $\partial_\mathsf{f} V$, see \eqref{fineDirichlet}. In particular, \( h_{\bf 1}\leq g_K<1 \) in \( U \) by the maximum principle for harmonic functions. 

Let \( F \subset U\) be a closed disk. Observe that  $V\setminus F$ is a fine domain as otherwise $F$ would finely disconnect $U$ whereas $U\setminus F$ is a domain and therefore also a fine domain. Put $\delta_z^* := \delta_z^{(\Omega\setminus V)\cup F}$,  \(z\in V\setminus F\), and observe that \(\delta_{z\mathcal b\partial_\mathsf{f} F}^*\) is  a non-trivial measure by \eqref{detfs} and since \( \partial_\mathsf{f}((\Omega\setminus V)\cup F)=\partial_\mathsf{f}V\cup \partial_\mathsf{f}F \), where the union is disjoint and \( b(\partial_\mathsf{f}F)=\partial_\mathsf{f}F=\partial F \). We now get from \eqref{transit} and \eqref{detfs}, applied to the fine domain $V\setminus F$, that
\begin{equation}
  \label{labaldec}
  \delta_z^{\Omega\setminus V} = \big(\delta^*_z\big)^{\Omega\setminus V} =  \big(\delta_{z\mathcal b\partial_\mathsf{f} V}^*\big)^{\Omega\setminus V} + \big(\delta_{z\mathcal b\partial_\mathsf{f} F}^*\big)^{\Omega\setminus V} =  \delta_{z\mathcal b\partial_\mathsf{f} V}^* + \big(\delta_{z\mathcal b\partial_\mathsf{f} F}^*\big)^{\Omega\setminus V},
\end{equation}
where we observe that the balayage out of \( V \) does not change measures supported on \( b(\partial_\mathsf{f}V) = \partial_\mathsf{f}V\) by \eqref{char_balayage}. Since $h_{\bf 1}<1$ on $\partial_\mathsf{f} F$ as \( \partial_\mathsf{f}F\subset U \), we get from \eqref{preservemass} and \eqref{baldfi} that
\[
0<\big(\delta_{z\mathcal b\partial_\mathsf{f} F}^*\big)^{\Omega\setminus V}(\partial_\mathsf{f} V) = \int \delta_x^{\Omega\setminus V}(\partial_\mathsf{f} V) d\delta_{z\mathcal b\partial_\mathsf{f} F}^*(x) = \int h_{\bf 1}(x) d\delta_{z\mathcal b\partial_\mathsf{f} F}^*(x) < \delta_{z\mathcal b\partial_\mathsf{f} F}^*(\partial_\mathsf{f} F)
\]
which implies, in view of \eqref{labaldec}, that \( h_{\bf 1}(z)=\delta_z^{\Omega\setminus V}(\partial_\mathsf{f} V)<\delta_z^*((\Omega\setminus V)\cup F)\leq 1\) as claimed.

Finally, let $h$ be a harmonic function on $\Omega$ which is majorized on $V$ in absolute value by a semi-bounded potential. Because $h$ is also finely harmonic on $V$ and finely continuous on $\clos_\mathsf{f}(V)$, we deduce from Theorem~\ref{thm_fineDirichlet} that it is the solution of the fine Dirichlet problem on $V$ with boundary data $h_{\mathcal b\partial_\mathsf{f} V}$ and that \eqref{reproduction} holds.
\end{proof}

\begin{lemma}
\label{casparta}
Let   $V\subset\RS$ be a proper regular fine domain such that $\overline{ \partial_\mathsf{f} V } \cap \partial\RS=\varnothing$ and $\overline{V}\cap\partial \RS \neq \varnothing$ (the lower assumption on the right-hand side of \eqref{mass-balayage} when \( \Omega=\RS\)). Let further  \( h \) be a  harmonic function in $\RS$ such that \(\lim_{z\to\xi}h(z)=0 \) for every  $\xi\in \overline{V}\cap\partial\RS$. Then identity \eqref{reproduction} holds.
\end{lemma}
\begin{proof}
In view of Lemma \ref{lem:mass-balayage}, it is enough to show that $h$ is majorized on $V$ by a semi-bounded potential. Note, as in the proof of Lemma \ref{lem:mass-balayage}, that $\partial_\mathsf{f} V$ is non-polar. Let us show that $\partial\RS\cap\overline{V}$ consists of a union of connected components of $\partial\RS$. Indeed, any such component $\Gamma$ is a 1-dimensional compact topological submanifold  of $\RS_*$, and as such it has a tubular neighborhood $N$ that may be chosen so thin that $N\cap \overline{ \partial_\mathsf{f} V }=\varnothing$. Then,  if $\zeta_1, \zeta_2\in\Gamma$  and  $\zeta_1\in \overline{V}$ while $\zeta_2\notin \overline{V}$, we can find  $z_1\in N\cap  V$ close to $\zeta_1$ and $z_2\in N\cap \Omega\setminus V$ close to $\zeta_2$. The points $z_1$, $z_2$ can be joined by a smooth arc contained in $N$. However, such an arc is finely connected \cite[Theorem~7]{Fug71}, but cannot meet $\overline{ \partial_\mathsf{f} V }$ by construction, a contradiction that proves our claim.

Assume first that $D=\D$ is the unit disk. Any function \( u \) harmonic in an annular region \( \{r<|z|<1 \} \) that extends continuously to \( \T \) by zero can be harmonically extended to \( \{ r<|z|<1/r \} \) by reflection, i.e., by setting \( u(z) := -u(1/\bar z) \) for \( z\in\{1<|z|<1/r \} \). Due to the smoothness of this extension it necessarily holds that \( |u(z)| \leq C_\rho (1-|z|) \) for \( r<\rho\leq |z| \leq 1 \). As \( h \) is harmonic on \( \RS \), this principle used around each of the finitely many connected components of \( \overline V\cap \partial \RS \) yields that $|h(z)|\leq C (1-|p(z)|)$ for $z\in V$ and some constant \( C>0 \). On the other hand, the function \( g_r(z) := -\log\max\{ r,|z| \} \) is a continuous (thus bounded and therefore semi-bounded) potential in \( \D \) for any \( r\in(0,1) \) (this is the Green equilibrium potential of \( \{|z|\leq r\} \)). It can be readily verified that \( g_r(z) \geq (1-|z|) \) in \( \D \) when \( r\leq e^{-1} \). Thus, \( |h(z)|\leq C g_r(p(z)) \), \( z\in V \), for any such \( r \). As \( g_r(p(z)) \) is a (bounded) potential on \( \RS \) by \eqref{remontp}, the claim of the corollary follows.

In the general case, let \( \phi:\D\to D \) be a conformal map. Recall that $\phi$ extends to a homeomorphism from \( \overline \D \) onto \( \overline D \). Clearly,  $\phi$ is also a homeomorphism for the fine topology since $v$ is superharmonic (resp. harmonic) on $D$ if and only if so is $v\circ\phi$  on $\D$. Moreover, $g$ is a bounded potential on $D$ if and only if $g\circ\phi$ is such a potential on $\D$. Hence, the result just proven on the disk carries over to $D$ by conformal mapping. 
\end{proof}

\small

\bibliographystyle{plain}
\bibliography{optimal}

\end{document}